\documentclass[11pt]{amsart}
\usepackage[utf8]{inputenc}
\usepackage{mlmodern}
\usepackage[english]{babel}
\usepackage{caption}
\usepackage{subcaption}
\usepackage{tikz}
\usepackage{pgfplots}
\pgfplotsset{compat=1.18}
\usepackage[dvipsnames]{xcolor}
\usetikzlibrary{angles,quotes}
\usetikzlibrary{arrows.meta}
\usepackage{graphicx,amssymb,latexsym,amsmath,amscd,amsthm}
\usepackage{verbatim, color}
\usepackage{mathtools}
\usepackage[top=1in, bottom=1in, left=1in, right=1in]{geometry}
\usepackage[noadjust]{cite}
\usepackage[foot]{amsaddr}

\newtheorem{lemma}{Lemma}
\newtheorem{prop}{Proposition}[section]
\newtheorem{thm}{Theorem}[section]
\newtheorem{cor}{Corollary}[section]
\theoremstyle{definition}
\newtheorem{Def}{Definition}[section]

\theoremstyle{remark}
\newtheorem{rem}{Remark}[section]

\definecolor{CitePurple}{RGB}{128,9,158}
\definecolor{CiteBlue}{RGB}{2,95,176}
\definecolor{LinkRed}{rgb}{0.7,0,0}
\usepackage[colorlinks=true,
            citecolor=CitePurple,
            linkcolor=LinkRed,
            urlcolor=CiteBlue]{hyperref}

\usepackage{tikz-cd}
\begin{document}
\title[Gromov-Hausdorff versus Hausdorff distances for Euclidean submanifolds]{Gromov-Hausdorff versus Hausdorff distances for Euclidean submanifolds with positive reach}
\author{Prasanna Brahme$^{1}$}
\email{pnbrahme@uh.edu}
\address{$^{1}$Department of Mathematics, University of Houston, Houston, Texas, USA}
\author{William Ott$^{1}$}
\email{william.ott.math@gmail.com}

\begin{abstract}
This paper is about comparing the Gromov-Hausdorff distance between a pair of sets to the Hausdorff distance between the pair.
Recent work on this problem has focused on two settings, Riemannian manifolds and metric graphs.
Here, we focus on embedded submanifolds of Euclidean space endowed with the Euclidean metric.
We prove that when $\mathcal{M}\subset\mathbb{R}^n$ is a compact connected embedded submanifold without boundary, with positive reach, and $\mathrm{A}\subset\mathcal{M}$ is a subset that is sufficiently dense in $\mathcal{M}$ with respect to the reach, the Gromov-Hausdorff distance between $\mathrm{A}$ and $\mathcal{M}$ is bounded below by $1/2$ of the Hausdorff distance between them.
We extend this result in two ways.
First, we show that the factor $1/2$ can be improved using Jung's theorem.
Second, we prove bounds on the Gromov-Hausdorff distance for pairs of subsets of $\mathcal{M}$.
When $\mathbb{S}^{1}$ is equipped with the two-dimensional Euclidean metric, the Gromov-Hausdorff distance between $\mathbb{S}^{1}$ and any sufficiently dense subset of $\mathbb{S}^{1}$ can be expressed explicitly as a function of the Hausdorff distance between the two.
We use interleavings between geometric simplicial complexes, nerve lemmas, properties of reach and the nearest-point projection map, and the fundamental class to produce the inequalities for general $\mathcal{M}$.
We use the Borsuk-Ulam theorem for $\mathcal{M}=\mathbb{S}^1\subset \mathbb{R}^2$.
\end{abstract}
\maketitle
\section{Introduction}\label{sec:Intro}
The Gromov–Hausdorff distance was introduced independently by Edwards \cite{EDWARDS1975121} as a metric on the set of
isometry classes of metric spaces, and by Gromov \cite{Gromov1981-mn} to prove results on groups of polynomial growth.
It metrizes the set of isometry classes of compact metric spaces and the set of finite metric spaces is everywhere
dense in this metric space.
The Gromov-Hausdorff distance has since been used in a variety of computational applications, such as comparing
point clouds sampled from shapes \cite{Memoli2012-ra}.
This distance is preferable to the Hausdorff distance for non-rigid shape matching problems, where the
Hausdorff distance can miss similarities of interest \cite{Schmiedl2017}.
On the other hand, computing the Gromov-Hausdorff distance between
finite metric spaces is known to lead to NP-hard problems \cite{Schmiedl2017}.
Consequently, significant recent work has focused on computing Gromov-Hausdorff distances for specific types of metric spaces and finding theoretical lower bounds on Gromov-Hausdorff distances \cite{adams2025hausdorffvsgromovhausdorffdistances,https://doi.org/10.48550/arxiv.2301.00246,Agarwal2018,adams2025lowerboundinggromovhausdorffdistance,4563074}.

In \cite{adams2025hausdorffvsgromovhausdorffdistances,adams2025lowerboundinggromovhausdorffdistance}, the authors find lower bounds on Gromov-Hausdorff distances in terms of Hausdorff distances in the contexts of compact Riemannian manifolds \cite{adams2025hausdorffvsgromovhausdorffdistances} and metric graphs \cite{adams2025lowerboundinggromovhausdorffdistance}.
In \cite[Section 9 Question 3]{adams2025hausdorffvsgromovhausdorffdistances}, the authors ask if bounds of this type hold for manifolds embedded in $\mathbb{R}^n$ when equipped with the ambient Euclidean metric.
Here, we answer this question affirmatively for compact connected embedded submanifolds of $\mathbb{R}^{n}$ with positive reach.

\subsection{Gromov-Hausdorff distance and Hausdorff distance can be unequal}

Let ($\mathrm{X}, \mathrm{d}$) be a metric space. For any $x \in \mathrm{X}$, we let $\mathbb{B}(x ; r)=\left\{x^{\prime} \in \mathrm{X} \mid \mathrm{d}\left(x, x^{\prime}\right)<r\right\}$ denote the metric ball of radius $r$ about $x$. For any $r \geq 0$ and $\mathrm{U} \subset \mathrm{X}$, we let $\mathbb{B}(\mathrm{U} ; r)=\cup_{x \in \mathrm{U}} \mathbb{B}(x ; r)$ denote the union of metric balls of radius $r$ centered at the points of $\mathrm{U}$. The diameter of a subset $\mathrm{U} \subset \mathrm{X}$ is defined by
$
\operatorname{diam}(\mathrm{U})=\sup \left\{\mathrm{d}\left(x, x^{\prime}\right) \mid x, x^{\prime} \in \mathrm{U}\right\}.
$

If $\mathrm{X}$ and $\mathrm{Y}$ are two subspaces of a metric space $\mathrm{Z}$ then the \textbf{Hausdorff distance} between $\mathrm{X}$ and $\mathrm{Y}$ is defined as
$$
\mathrm{d}_{\mathrm{H}}(\mathrm{X}, \mathrm{Y})=\inf \{r \geq 0 \mid \mathrm{X} \subset \mathbb{B}(\mathrm{Y} ;r) \text { and } \mathrm{Y} \subset\mathbb{B}(\mathrm{X} ;r)\},
$$
where the metric balls are taken in the common metric space $\mathrm{Z}$.
The \textbf{Gromov-Hausdorff distance} $\mathrm{d}_{\mathrm{GH}}(\mathrm{X}, \mathrm{Y})$ between two metric spaces ($\mathrm{X}, \mathrm{d}_{\mathrm{X}}$) and ($\mathrm{Y}, \mathrm{d}_{\mathrm{Y}}$) is defined as
$$ \mathrm{d}_{\mathrm{GH}}(\mathrm{X},\mathrm{Y}) = \inf\limits_{\gamma,\varphi,\mathrm{Z}}\mathrm{d}_\mathrm{H}(\gamma(\mathrm{X}),\varphi(\mathrm{Y})),
$$
where the infimum is taken over all metric spaces $\mathrm{Z}$ and isometric embeddings $\gamma: \mathrm{X} \rightarrow \mathrm{Z}$ and $\varphi: \mathrm{Y} \rightarrow \mathrm{Z}$.

Figure \ref{fig:threepoint} shows a three-point subset of $\mathbb{S}^1$ for which the Gromov-Hausdorff distance is strictly smaller than the Hausdorff distance. 

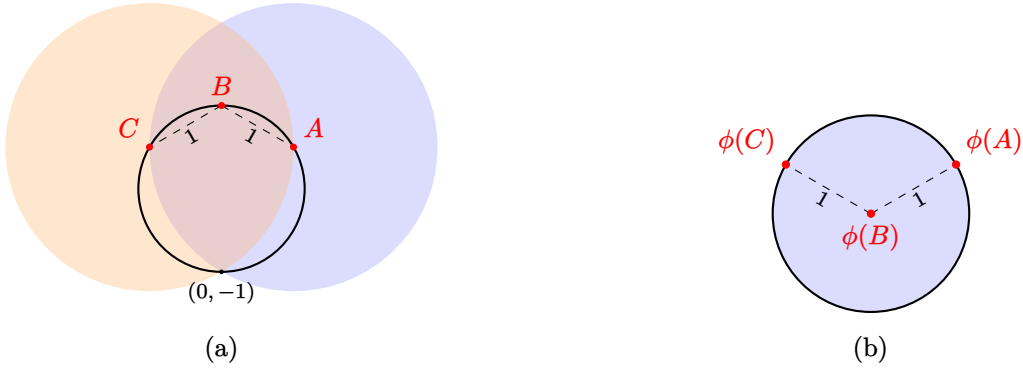
\begin{figure}[htb]
\centering
\begin{subfigure}{0.48\textwidth}
\centering
\begin{tikzpicture}[scale=1.1]
\fill[blue!45,opacity=0.30] (0.8660254,0.5) circle (1.7320508);
\fill[orange!55,opacity=0.35] (-0.8660254,0.5) circle (1.7320508);
\draw[black,thick] (0,0) circle (1);
\draw[dashed] (0,1)--(0.8660254,0.5) node[midway,sloped,below,inner sep=1.5pt,font=\scriptsize] (){$1$};
\draw[dashed] (0,1)--(-0.8660254,0.5) node[midway,sloped,below,inner sep=1.5pt,font=\scriptsize] (){$1$};
\fill[red] (0.8660254,0.5) circle (1.2pt);
\draw[red] (0.8660254,0.5) node[anchor=south west,font=\small] (){$A$};
\fill[red] (0,1) circle (1.2pt);
\draw[red] (0,1) node[anchor=south,font=\small] (){$B$};
\fill[red] (-0.8660254,0.5) circle (1.2pt);
\draw[red] (-0.8660254,0.5) node[anchor=south east,font=\small] (){$C$};
\fill (0,-1) circle (0.8pt);
\draw (0,-1) node[anchor=north,font=\scriptsize] (){$(0,-1)$};
\end{tikzpicture}
\caption*{(a)}
\end{subfigure}\hfill
\begin{subfigure}{0.48\textwidth}
\centering
\begin{tikzpicture}[scale=1.3]
\fill[blue!45,opacity=0.30] (0,0) circle (1);
\draw[black,thick] (0,0) circle (1);
\draw[dashed] (0,0)--(0.8660254,0.5) node[midway,sloped,below,inner sep=1.5pt,font=\scriptsize] (){$1$};
\draw[dashed] (0,0)--(-0.8660254,0.5) node[midway,sloped,below,inner sep=1.5pt,font=\scriptsize] (){$1$};
\fill[red] (0.8660254,0.5) circle (1.2pt);
\draw[red] (0.8660254,0.5) node[anchor=south west,font=\small] (){$\phi(A)$};
\fill[red] (-0.8660254,0.5) circle (1.2pt);
\draw[red] (-0.8660254,0.5) node[anchor=south east,font=\small] (){$\phi(C)$};
\fill[red] (0,0) circle (1.2pt);
\draw[red] (0,0) node[anchor=north,font=\small] (){$\phi(B)$};
\end{tikzpicture}
\caption*{(b)}
\end{subfigure}
\caption{(a) $\mathrm{S}=\{A,B,C\}\subset\mathbb{S}^1$ (Euclidean metric); the balls of radius $\sqrt3$ at $A,C$ reach $(0,-1)$, the farthest circle point, so $\operatorname{d}_\mathrm{H}(\mathrm{S},\mathbb{S}^1)=\sqrt3$. (b) The isometric copy $\phi(\mathrm{S})$ with $\phi(B)$ at the center gives $\operatorname{d}_\mathrm{H}(\phi(\mathrm{S}),\mathbb{S}^1)=1$; here $\phi$ fixes $A$ and $C$ and preserves both edges of length $1$, so it is an isometry. As $\operatorname{d}_\mathrm{GH}$ minimizes the Hausdorff distance over isometric placements, this one placement gives $\operatorname{d}_\mathrm{GH}(\mathrm{S},\mathbb{S}^1)\leq1<\sqrt3=\operatorname{d}_\mathrm{H}$.}
\label{fig:threepoint}
\end{figure}

\subsection{Statement of results}
Throughout the article, by closed manifold we mean compact manifold without boundary.
Submanifolds of $\mathbb{R}^{n}$ are embedded submanifolds equipped with the ambient Euclidean metric (unless specifically stated otherwise).
We write $\operatorname{r}_1=\operatorname{r}_1(\mathcal{M})$ for the reach of the embedded submanifold $\mathcal{M} \subset \mathbb{R}^{n}$ and $\lambda_n=\sqrt{\frac{n}{2(n+1)}}$ for Jung's constant.

\vspace{5pt}
\noindent\textbf{Theorem \ref{thm:firstbound}.}
Let $\mathcal{M}\subset\mathbb{R}^n$ be a closed connected $d$-dimensional submanifold of reach $\operatorname{r}_1(\mathcal{M})=\operatorname{r}_1>0$.
Let $\mathrm{A}$ be a subset of $\mathcal{M}$.
We have
$$
\operatorname{d}_{\mathrm{GH}}(\mathrm{A},\mathcal{M})\geq\min\left\{\frac{1}{2}\operatorname{d}_{\mathrm{H}}(\mathrm{A},\mathcal{M}),\frac{\sqrt2}{6}\operatorname{r}_1(\mathcal{M})\right\} .
$$

This bound is proved by interleaving the ambient \v{C}ech complex of $\mathcal{M}$ with the intrinsic \v{C}ech complex of $\mathrm{A}$.
A correspondence of small distortion together with Jung's theorem produces simplicial maps carrying the fundamental class $[\mathcal{M}]$.
Once the scale is small relative to $\operatorname{r}_1$, a nerve lemma forces the top homology of the intrinsic complex to vanish, which gives the inequality.
In general, a nerve lemma establishes a homotopy equivalence between the geometric realization of a nerve (a simplicial complex) constructed from a cover and a space constructed from the union of the covering sets \cite{Alexandroff1928,Borsuk1948,MR1867354,Bauer_2023}.

For the circle we obtain a sharper bound by a Borsuk-Ulam type argument in place of the nerve lemma.
Our argument is similar to the proof of Theorem 5.1 in \cite{adams2025lowerboundinggromovhausdorffdistance}.
In the following theorem, $\psi(t)=2\sin(2\sin^{-1}(t/2))$ is the function sending the length of the chord of an arc to the length of the chord of the arc of twice the angle.

\vspace{5pt}
\noindent\textbf{Theorem \ref{theorem:circbu}.} Let $\mathbb{S}^1\subset\mathbb{R}^2$ be the unit circle equipped with the two-dimensional Euclidean metric.
Then for any subset $\mathrm{A}\subset\mathbb{S}^1$ with $\operatorname{d}_{\mathrm{H}}\left(\mathrm{A},\mathbb{S}^1\right)\leq\sqrt2$, we have
$$
\operatorname{d}_{\mathrm{GH}}\left(\mathrm{A},\mathbb{S}^1\right)\geq\min\left\{\frac{1}{2}\psi\left(\operatorname{d}_{\mathrm{H}}\left(\mathrm{A},\mathbb{S}^1\right)\right),\frac{c_0}{2}\right\} ,
$$
where $c_0\in(0,2)$ is the solution of the equation $\psi^{-1}(x)+x=2$ ($c_0\approx1.305$).
\vspace{5pt}

\noindent\textbf{Corollary \ref{cor:circeq}.}
Let $\mathrm{A}\subset\mathbb{S}^1$.
If $\psi(\operatorname{d}_{\mathrm{H}}(\mathrm{A},\mathbb{S}^1))\leq c_0$, then
    $\operatorname{d}_{\mathrm{GH}}(\mathrm{A},\mathbb{S}^1)=\frac{1}{2}\psi(\operatorname{d}_{\mathrm{H}}(\mathrm{A},\mathbb{S}^1))$ (here $c_0$ is as defined in Theorem \ref{theorem:circbu}).
\vspace{5pt}

Theorem \ref{theorem:circbu} shows that the Gromov-Hausdorff distance can be bounded below by a function of the Hausdorff distance if the subset is sufficiently dense. To prove Corollary \ref{cor:circeq}, we construct a correspondence with distortion that achieves this bound.

To compare the Gromov-Hausdorff distance between two subsets of the same submanifold with the Hausdorff distance between them, we use Vietoris-Rips complexes. This also gives us better lower bounds, as using Vietoris-Rips complexes allows us to construct tighter interleavings.

\vspace{5pt}
\noindent\textbf{Theorem \ref{theorem:main2set}.} Let $\mathcal{M}\subset\mathbb{R}^n$ be a connected closed $d$-dimensional submanifold of positive reach $\operatorname{r}_1>0$. For each $\alpha\in(0,1]$, set $c_\alpha^n=\sqrt{\frac{2}{1+\sqrt{1-\alpha^2}}}\,\lambda_n$ and $c=\min\!\left(\alpha,\tfrac{1}{1+2c_\alpha^n}\right)$. Then for all subsets $\mathrm{A},\mathrm{B}\subset\mathcal{M}$,
$$\operatorname{d}_{\mathrm{GH}}(\mathrm{A},\mathrm{B})\geq\min\left\{\tfrac{1}{2c_\alpha^n}\operatorname{d}_{\mathrm{H}}(\mathrm{A},\mathcal{M})-\operatorname{d}_{\mathrm{H}}(\mathrm{B},\mathcal{M}),\ \tfrac{c}{2}\operatorname{r}_1-\operatorname{d}_{\mathrm{H}}(\mathrm{B},\mathcal{M})\right\} .$$
\noindent As $c_\alpha^n<1$, the coefficient $\tfrac{1}{2c_\alpha^n}$ exceeds $\tfrac12$, so this improves Theorem \ref{thm:firstbound} when $\mathrm{A}$ is Hausdorff-close to $\mathcal{M}$.
\vspace{5pt}

In Section \ref{sec:ambient} we sharpen these estimates by working with ambient \v{C}ech complexes.
As ambient \v{C}ech complexes need not be homotopy equivalent to $\mathcal{M}$, we instead include an ambient \v{C}ech complex into an intrinsic one through the nearest-point projection controlled by a quantity that measures how far the projection displaces points near the manifold.
Then we combine this with a tighter Vietoris-Rips to ambient \v{C}ech interleaving using results from \cite{federer1959curvature,boissonnat2019reach} to sharpen our estimates.
These techniques produce the following theorem in particular.

\vspace{5pt}
\noindent\textbf{Theorem \ref{theorem:spherebound}.}
 Consider two subsets $\mathrm{A},\mathrm{B}\subset\mathbb{S}^{n-1}\subset\mathbb{R}^n$ with $\operatorname{d}_\mathrm{H}\left(\mathrm{A},\mathbb{S}^{n-1}\right)\leq\sqrt2$. Let\\ $c=\min\left\{\sqrt2,\frac{1}{\sqrt2\lambda_n(2\lambda_{n-1}+1)}\right\}$. Then the following hold.
\begin{align*}
&1.) \;\operatorname{d}_{\mathrm{GH}} \left(\mathrm{A},\mathbb{S}^{n-1} \right)\geq\operatorname{min}\left\{\frac{1}{2\lambda_{n-1}}\psi_{\mathbb{S}^{n-1}}^{-1}\left(\operatorname{d}_\mathrm{H} \left(\mathrm{A},\mathbb{S}^{n-1} \right)\right),\frac{c}{2}\right\}\;\\
&2.) \;\operatorname{d}_{\mathrm{GH}} \left(\mathrm{A},\mathrm{B} \right)\geq\operatorname{min}\left\{\frac{1}{2\lambda_{n-1}}\psi_{\mathbb{S}^{n-1}}^{-1}\left(\operatorname{d}_\mathrm{H} \left(\mathrm{A},\mathbb{S}^{n-1} \right)\right)-\operatorname{d}_{\mathrm{H}} \left(\mathrm{B},\mathbb{S}^{n-1} \right),\frac{c}{2}-\operatorname{d}_{\mathrm{H}} \left(\mathrm{B},\mathbb{S}^{n-1} \right)\right\}\;
\end{align*}
 \noindent Here, $\psi_{\mathbb{S}^{n-1}}(r)=\sqrt{2-2\sqrt{1-r^2}}$.
\vspace{5pt} 

\subsection{On the differences between the Euclidean case and the Riemannian case}

There are two important differences between the previously-studied Riemannian case, and the Euclidean case we analyze in this work.
First, ambient \v{C}ech complexes constructed for subsets of Euclidean submanifolds may not have trivial top homology even when the union of balls does not cover the manifold.
Thus, the use of ambient \v{C}ech complexes and the functorial nerve lemma as in \cite{adams2025hausdorffvsgromovhausdorffdistances} does not apply directly. We work around this using results from \cite{kim2019homotopy,federer1959curvature} and the nearest-point projection map.

Second, the theorem by Hausmann (Theorem 3.5 of \cite{Hausmann+2016+175+188}) establishes a homotopy equivalence between the Vietoris-Rips complex on a compact Riemannian manifold and the underlying manifold. Such a result is not direct in the case of Euclidean submanifolds. To use Vietoris-Rips complexes in the Euclidean case, we collect results in \cite{gillespie2024vietoris,adams2019metric,adams2022vietoristhickeningscomplexesisomorphic,adams2025vietoristhickeningscomplexesmanifolds} and prove that the Vietoris-Rips complexes constructed on Euclidean submanifolds are homotopy equivalent to the underlying manifold at small enough scales.

\section{Background}\label{sec:background}
\subsection{Metric spaces and manifolds}

We begin by collecting the basic notions of metric geometry used in this article. Unless stated otherwise, $\mathbb{R}^n$ is equipped with the standard Euclidean inner product $\langle\cdot,\cdot\rangle$ and the induced norm $\lVert x\rVert=\sqrt{\langle x,x\rangle}$ and $\operatorname{d}\left(x,y\right)=\lVert x-y\rVert$ denotes the Euclidean distance between $x,y\in\mathbb{R}^n$. We write $\mathbb{B}\left(x,r\right)=\{y\in\mathbb{R}^n\mid\operatorname{d}\left(x,y\right)<r\}$ for the open Euclidean ball of radius $r>0$ centered at $x$.

\begin{Def}
    A \textbf{metric space} is a pair $(\mathrm{X},\operatorname{d}_\mathrm{X})$ consisting of a set $\mathrm{X}$ together with a function $\operatorname{d}_\mathrm{X}:\mathrm{X}\times\mathrm{X}\to\mathbb{R}^{\geq0}$, called the \textbf{metric} (or distance), satisfying for all $x,y,z\in\mathrm{X}$:
    \begin{align*}
        &\;\;\operatorname{d}_\mathrm{X}\left(x,y\right)=0\iff x=y \hspace{1cm}\\
        &\;\;\operatorname{d}_\mathrm{X}\left(x,y\right)=\operatorname{d}_\mathrm{X}\left(y,x\right)\hspace{1.55cm}\text{(symmetry),}\\
        &\;\;\operatorname{d}_\mathrm{X}\left(x,z\right)\leq\operatorname{d}_\mathrm{X}\left(x,y\right)+\operatorname{d}_\mathrm{X}\left(y,z\right)\hspace{0.4cm}\text{(triangle inequality).}
    \end{align*}
    A subset $\mathrm{A}\subset\mathrm{X}$ inherits the structure of a metric space by restricting $\operatorname{d}_\mathrm{X}$ to $\mathrm{A}\times\mathrm{A}$. We call $(\mathrm{A},\operatorname{d}_\mathrm{X}|_{\mathrm{A}\times\mathrm{A}})$ a \textbf{metric subspace} of $\mathrm{X}$. The \textbf{diameter} of $\mathrm{A}$ is $\operatorname{diam}\left(\mathrm{A}\right)=\sup\limits_{x,y\in\mathrm{A}}\operatorname{d}_\mathrm{X}\left(x,y\right)$ and $\mathrm{A}$ is \textbf{bounded} if $\operatorname{diam}\left(\mathrm{A}\right)<\infty$.
\end{Def}

For a Euclidean submanifold $\mathcal{M}\subset\mathbb{R}^n$ we always regard $\mathcal{M}$ as a metric space with the restriction of the ambient Euclidean metric, that is, with the \textbf{extrinsic} distance $\operatorname{d}\left(x,y\right)=\lVert x-y\rVert$ for $x,y\in\mathcal{M}$. This is distinct from the \textbf{intrinsic} (geodesic) distance $\operatorname{d}_g\left(x,y\right)$, defined as the infimum of the lengths of paths in $\mathcal{M}$ joining $x$ and $y$. The two quantities agree only infinitesimally and in general satisfy $\operatorname{d}\left(x,y\right)\leq\operatorname{d}_g\left(x,y\right)$.

\begin{Def}\label{def:distortion}
    Let $(\mathrm{X},\operatorname{d}_\mathrm{X})$ and $(\mathrm{Y},\operatorname{d}_\mathrm{Y})$ be metric spaces and $f:\mathrm{X}\to\mathrm{Y}$ a map. The \textbf{distortion} of $f$ is
    $$\operatorname{dis}\left(f\right)=\sup\limits_{x,x'\in\mathrm{X}}\left|\operatorname{d}_\mathrm{Y}\left(f\left(x\right),f\left(x'\right)\right)-\operatorname{d}_\mathrm{X}\left(x,x'\right)\right|.$$
    The map $f$ is an \textbf{isometric embedding} if $\operatorname{dis}\left(f\right)=0$ and an \textbf{isometry} if, in addition, it is a bijection. $f$ is \textbf{Lipschitz} with constant $L\geq0$ if $\operatorname{d}_\mathrm{Y}\left(f\left(x\right),f\left(x'\right)\right)\leq L\,\operatorname{d}_\mathrm{X}\left(x,x'\right)$ for all $x,x'\in\mathrm{X}$.
\end{Def}

\begin{Def}
    Let $\mathcal{M}\subset\mathbb{R}^n$. We say $\mathcal{M}$ is a \textbf{topological $d$-manifold} (without boundary) if it is a second countable Hausdorff space (in the subspace topology) such that every point $x\in\mathcal{M}$ has an open neighborhood $\mathrm{U}\subset\mathcal{M}$ homeomorphic to an open subset of $\mathbb{R}^d$. The integer $d$ is the \textbf{dimension} of $\mathcal{M}$. A homeomorphism $\varphi:\mathrm{U}\to\varphi\left(\mathrm{U}\right)\subset\mathbb{R}^d$ is a \textbf{chart} and a collection of charts whose domains cover $\mathcal{M}$ is an \textbf{atlas}.
\end{Def}

\begin{Def}\label{def:regularity}
    Let $\mathcal{M}\subset\mathbb{R}^n$ be a topological $d$-manifold and let $k\in\mathbb{Z}^{\geq0}$. An atlas $\{(\mathrm{U}_\alpha,\varphi_\alpha)\}_\alpha$ is of class $C^{k}$ if for every pair of charts the transition map $\varphi_\beta\circ\varphi_\alpha^{-1}$ is $k$-times continuously differentiable on its domain. We say $\mathcal{M}$ is a \textbf{$C^k$ submanifold} (a \textbf{smooth} submanifold when $k=\infty$) if it admits such an atlas. A submanifold is said to be $C^{1,1}$ if it is a $C^1$ submanifold and the first derivatives of the charts are locally Lipschitz.  For a $C^1$ submanifold $\mathcal{M}\subset\mathbb{R}^n$ and $x\in\mathcal{M}$ we write $T_x\mathcal{M}\subset\mathbb{R}^n$ for the \textbf{tangent space} at $x$, the $d$-dimensional linear subspace of directions tangent to $\mathcal{M}$ at $x$ and $T_x\mathcal{M}^\perp$ for its orthogonal complement, the \textbf{normal space}. We denote by $\pi_{T_x\mathcal{M}}:\mathbb{R}^n\to T_x\mathcal{M}$ the orthogonal projection onto the (affine) tangent space at $x$.

\end{Def}

\begin{Def}
    For a compact subset $\mathcal{M} \subset \mathbb{R}^n$ we define the \textbf{distance function} to $\mathcal{M}$, $\operatorname{d}_\mathcal{M}\left(\cdot\right):\mathbb{R}^n\to \mathbb{R}^{\geq0}$, by $\operatorname{d}_{\mathcal{M}}\left(x\right) = \operatorname{inf}_{y\in \mathcal{M}}\operatorname{d}\left(x,y\right)$, where $\operatorname{d}\left(\cdot,\cdot\right)$ is the Euclidean distance.


\end{Def}


\begin{Def}\label{def:reach}
    For a closed set $\mathcal{S}\subset\mathbb{R}^n$ the \textbf{medial axis} of $\mathcal{S}$ is the set
    $$\operatorname{Med}_{\mathcal{S}}=\overline{\left\{z\in\mathbb{R}^n\,\middle|\,\exists\, x,y\in\mathcal{S},\; x\neq y,\;\operatorname{d}_\mathcal{S}\left(z\right)=\operatorname{d}\left(x,z\right)=\operatorname{d}\left(y,z\right)\right\}},$$
    that is, the closure of the set of points having at least two distinct nearest points in $\mathcal{S}$. The \textbf{reach} of $\mathcal{S}$ is
    $$\operatorname{reach}\left(\mathcal{S}\right)=\inf\limits_{x\in\mathcal{S}}\operatorname{d}_{\operatorname{Med}_{\mathcal{S}}}\left(x\right),$$
    the smallest distance from $\mathcal{S}$ to its medial axis. Equivalently, $\operatorname{reach}\left(\mathcal{S}\right)$ is the supremum of all $r\geq0$ such that every point $x$ with $\operatorname{d}_\mathcal{S}\left(x\right)<r$ has a unique nearest point in $\mathcal{S}$. From here on we denote the reach of a manifold $\mathcal{M}$ by $\operatorname{r_1}(\mathcal{M})$ or $\operatorname{r_1}$ when there is no ambiguity.
\end{Def}


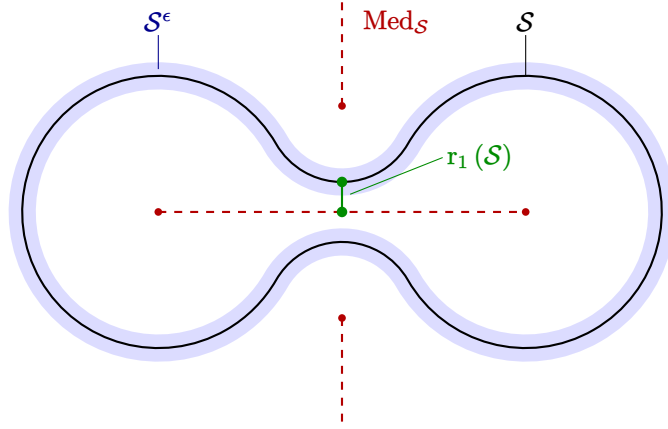
\begin{figure}[htb]
\centering
\begin{tikzpicture}[scale=1.8,line join=round]

\draw[blue!45,opacity=0.30,line width=3.6mm]
  (0.483975,-0.5)
  arc[start angle=-150, end angle=150,  radius=1]
  arc[start angle=-30,  end angle=-150, radius=0.558846]
  arc[start angle=30,   end angle=330,  radius=1]
  arc[start angle=150,  end angle=30,   radius=0.558846] -- cycle;

\draw[red!70!black,dashed,thick] (-1.35,0)--(1.35,0);
\draw[red!70!black,dashed,thick] (0,0.779423)--(0,1.55);
\draw[red!70!black,dashed,thick] (0,-0.779423)--(0,-1.55);
\fill[red!70!black] (-1.35,0) circle (0.8pt);
\fill[red!70!black] (1.35,0) circle (0.8pt);
\fill[red!70!black] (0,0.779423) circle (0.8pt);
\fill[red!70!black] (0,-0.779423) circle (0.8pt);

\draw[black,thick]
  (0.483975,-0.5)
  arc[start angle=-150, end angle=150,  radius=1]
  arc[start angle=-30,  end angle=-150, radius=0.558846]
  arc[start angle=30,   end angle=330,  radius=1]
  arc[start angle=150,  end angle=30,   radius=0.558846] -- cycle;

\draw[green!55!black,thick] (0,0)--(0,0.220577);
\fill[green!55!black] (0,0) circle (1.1pt);
\fill[green!55!black] (0,0.220577) circle (1.1pt);
\draw[green!55!black,thin] (0.06,0.13)--(0.74,0.40)
  node[anchor=west,font=\small,inner sep=1.5pt] {$\operatorname{r}_1\left(\mathcal{S}\right)$};

\draw[black,thin] (1.35,1)--(1.35,1.30)
  node[anchor=south,font=\small,inner sep=1.5pt] {$\mathcal{S}$};
\draw[blue!55!black,thin] (-1.35,1.05)--(-1.35,1.30)
  node[anchor=south,font=\small,inner sep=1.5pt] {$\mathcal{S}^\epsilon$};
\node[red!70!black,font=\small,anchor=west] at (0.08,1.40) {$\operatorname{Med}_\mathcal{S}$};
\end{tikzpicture}
\caption{A dumbbell-shaped closed curve $\mathcal{S}\subset\mathbb{R}^2$. Dashed: the medial axis $\operatorname{Med}_\mathcal{S}$, consisting of the segment joining the two circle centres inside $\mathcal{S}$ together with two vertical rays outside it. Shaded: the offset $\mathcal{S}^\epsilon$ for an $\epsilon$ below the reach, so that it is a tube about $\mathcal{S}$ and does not self-intersect. The reach $\operatorname{r}_1\left(\mathcal{S}\right)$ is attained in the handle, as the distance from the innermost point of a joining arc to the medial axis.}
\label{fig:dumbbell}
\end{figure}
\begin{rem}
     Every compact $C^{1,1}$ submanifold of $\mathbb{R}^n$ has positive reach and every (closed) convex set has infinite reach, with the medial axis empty. For a $C^2$ submanifold the reach is bounded above by the inverse of the maximal principal curvature (see \cite{Niyogi2008}), so the reach roughly quantifies how sharply $\mathcal{M}$ curves and how close distinct ``sheets" of $\mathcal{M}$ come to one another.
     
     
\end{rem}

If a compact set $\mathcal{S}\subset\mathbb{R}^n$ has positive reach $\operatorname{r}_1>0$, then the nearest-point map is well defined and continuous on the open $r$-offset for every $r<\operatorname{r}_1$. We denote it by $\pi_\mathcal{S}:\mathcal{S}^{\operatorname{r}_1}\to\mathcal{S}$, $\pi_\mathcal{S}\left(y\right)=\operatorname{argmin}\limits_{x\in\mathcal{S}}\operatorname{d}\left(x,y\right)$ and refer to $\pi_\mathcal{S}$ as the \textbf{metric projection} onto $\mathcal{S}$.
\begin{Def}
For a compact subset $\mathcal{M}\subset\mathbb{R}^n$ the \textbf{$r$-offset} $\mathcal{M}^r$ as $\mathcal{M}^r=\bigcup\limits_{x\in\mathcal{M}}\mathbb{B}\left(x,r\right)=\{y\in\mathbb{R}^n\mid\operatorname{d}_\mathcal{M}\left(y\right)<r\}$, the set of points at Euclidean distance less than $r$ from $\mathcal{M}$.
\end{Def}

\subsection{Hausdorff and Gromov-Hausdorff distances}

The Hausdorff distance compares two subsets that already lie in a common metric space, while the Gromov-Hausdorff distance compares two metric spaces in the abstract, regardless of any ambient space. We collect their definitions together with the formulations through correspondences and through the distortion and co-distortion of maps. Throughout, the $\epsilon$-offset of a subset $\mathrm{A}$ of a metric space $(\mathrm{Z},\operatorname{d}_\mathrm{Z})$ is $\mathrm{A}^\epsilon=\{z\in\mathrm{Z}\mid\operatorname{d}_\mathrm{Z}\left(z,\mathrm{A}\right)<\epsilon\}$, where $\operatorname{d}_\mathrm{Z}\left(z,\mathrm{A}\right)=\inf\limits_{a\in\mathrm{A}}\operatorname{d}_\mathrm{Z}\left(z,a\right)$.

\begin{Def}
    Let $(\mathrm{Z},\operatorname{d}_\mathrm{Z})$ be a metric space and $\mathrm{A},\mathrm{B}\subset\mathrm{Z}$ two nonempty subsets. The \textbf{Hausdorff distance} between $\mathrm{A}$ and $\mathrm{B}$ is
    $$\operatorname{d}_\mathrm{H}\left(\mathrm{A},\mathrm{B}\right)=\max\left\{\sup\limits_{a\in\mathrm{A}}\operatorname{d}_\mathrm{Z}\left(a,\mathrm{B}\right),\;\sup\limits_{b\in\mathrm{B}}\operatorname{d}_\mathrm{Z}\left(b,\mathrm{A}\right)\right\}=\inf\left\{\epsilon>0\mid\mathrm{A}\subset\mathrm{B}^\epsilon\;\text{and}\;\mathrm{B}\subset\mathrm{A}^\epsilon\right\}.$$
\end{Def}

 The Gromov-Hausdorff distance is obtained by isometrically embedding both spaces into a common one and minimizing the Hausdorff distance over all such embeddings.

\begin{Def}
    Let $(\mathrm{X},\operatorname{d}_\mathrm{X})$ and $(\mathrm{Y},\operatorname{d}_\mathrm{Y})$ be compact metric spaces. The \textbf{Gromov-Hausdorff distance} between $\mathrm{X}$ and $\mathrm{Y}$ is
    $$\operatorname{d}_\mathrm{GH}\left(\mathrm{X},\mathrm{Y}\right)=\inf\limits_{(\mathrm{Z},\iota_\mathrm{X},\iota_\mathrm{Y})}\operatorname{d}_\mathrm{H}^\mathrm{Z}\!\left(\iota_\mathrm{X}\left(\mathrm{X}\right),\iota_\mathrm{Y}\left(\mathrm{Y}\right)\right),$$
    where the infimum is taken over all metric spaces $(\mathrm{Z},\operatorname{d}_\mathrm{Z})$ together with isometric embeddings $\iota_\mathrm{X}:\mathrm{X}\hookrightarrow\mathrm{Z}$ and $\iota_\mathrm{Y}:\mathrm{Y}\hookrightarrow\mathrm{Z}$ and $\operatorname{d}_\mathrm{H}^\mathrm{Z}$ is the Hausdorff distance in $\mathrm{Z}$.
\end{Def}
The embedding definition of the Gromov-Hausdorff distance is usually not convenient for computation. The following reformulation (see\cite{Burago2001}) through correspondences avoids any reference to an ambient space and is the form used throughout the sections that follow.

\begin{Def}
    Let $\mathrm{X}$ and $\mathrm{Y}$ be sets. A \textbf{correspondence} between $\mathrm{X}$ and $\mathrm{Y}$ is a subset $\mathrm{C}\subset\mathrm{X}\times\mathrm{Y}$ whose projections to each factor are surjective, that is, for every $x\in\mathrm{X}$ there exists $y\in\mathrm{Y}$ with $\left(x,y\right)\in\mathrm{C}$ and for every $y\in\mathrm{Y}$ there exists $x\in\mathrm{X}$ with $\left(x,y\right)\in\mathrm{C}$. If $(\mathrm{X},\operatorname{d}_\mathrm{X})$ and $(\mathrm{Y},\operatorname{d}_\mathrm{Y})$ are metric spaces, the \textbf{distortion} of a correspondence $\mathrm{C}$ is
    $$\operatorname{dis}\left(\mathrm{C}\right)=\sup\left\{\left|\operatorname{d}_\mathrm{X}\left(x,x'\right)-\operatorname{d}_\mathrm{Y}\left(y,y'\right)\right|\;\mid\;\left(x,y\right),\left(x',y'\right)\in\mathrm{C}\right\}.$$
\end{Def}

\begin{prop}\label{prop:ghcorr}
    For compact metric spaces $\mathrm{X}$ and $\mathrm{Y}$,
    $$\operatorname{d}_\mathrm{GH}\left(\mathrm{X},\mathrm{Y}\right)=\frac{1}{2}\inf\limits_{\mathrm{C}}\operatorname{dis}\left(\mathrm{C}\right),$$
    where the infimum ranges over all correspondences $\mathrm{C}\subset\mathrm{X}\times\mathrm{Y}$ and the infimum is attained.
\end{prop}

Finally we record the equivalent description through pairs of maps, which makes explicit the distortion and co-distortion quantities (see Chapter 7 \cite{Burago2001}). Recall the distortion $\operatorname{dis}\left(f\right)$ of a single map $f:\mathrm{X}\to\mathrm{Y}$ from Definition \ref{def:distortion}.

\begin{Def}
    Let $(\mathrm{X},\operatorname{d}_\mathrm{X})$ and $(\mathrm{Y},\operatorname{d}_\mathrm{Y})$ be metric spaces and let $f:\mathrm{X}\to\mathrm{Y}$ and $g:\mathrm{Y}\to\mathrm{X}$ be (not necessarily continuous) maps. The \textbf{co-distortion} of the pair $(f,g)$ is
    $$\operatorname{dis}\left(f,g\right)=\sup\limits_{\substack{x\in\mathrm{X}\\ y\in\mathrm{Y}}}\left|\operatorname{d}_\mathrm{X}\left(x,g\left(y\right)\right)-\operatorname{d}_\mathrm{Y}\left(f\left(x\right),y\right)\right|.$$
\end{Def}

\begin{prop}\label{prop:ghmaps}
    For compact metric spaces $\mathrm{X}$ and $\mathrm{Y}$,
    $$\operatorname{d}_\mathrm{GH}\left(\mathrm{X},\mathrm{Y}\right)=\frac{1}{2}\inf\limits_{f,g}\max\left\{\operatorname{dis}\left(f\right),\,\operatorname{dis}\left(g\right),\,\operatorname{dis}\left(f,g\right)\right\},$$
    where the infimum is taken over all pairs of maps $f:\mathrm{X}\to\mathrm{Y}$ and $g:\mathrm{Y}\to\mathrm{X}$.
\end{prop}
\noindent The three formulations are interchangeable. A correspondence $\mathrm{C}$ determines a pair of maps by choosing, for each $x\in\mathrm{X}$, a point $f\left(x\right)\in\mathrm{Y}$ with $\left(x,f\left(x\right)\right)\in\mathrm{C}$ and similarly $g$. Conversely a pair $(f,g)$ determines the correspondence $\mathrm{C}=\{\left(x,f\left(x\right)\right)\mid x\in\mathrm{X}\}\cup\{\left(g\left(y\right),y\right)\mid y\in\mathrm{Y}\}$ and $\operatorname{dis}\left(\mathrm{C}\right)=\max\{\operatorname{dis}\left(f\right),\operatorname{dis}\left(g\right),\operatorname{dis}\left(f,g\right)\}$. This idea is used repeatedly in the sections that follow: a bound $r>2\operatorname{d}_\mathrm{GH}\left(\mathrm{A},\mathcal{M}\right)$ yields a correspondence, hence maps $h:\mathcal{M}\to\mathrm{A}$ and $g:\mathrm{A}\to\mathcal{M}$ of controlled distortion, which in turn induce simplicial maps between the associated complexes.

\subsection{Homology, homotopy and homology manifolds}

The bounds on Gromov-Hausdorff distance in later sections are obtained by tracking a fundamental class through a sequence of induced maps on homology, so we briefly recall the relevant algebraic topology notions, following \cite{MR1867354}. Throughout, homology is taken with coefficients in an abelian group $G$. We use $G=\mathbb{Z}_2$ unless stated otherwise and write $\mathrm{H}_n\left(\mathrm{X};G\right)$ for the $n$-th singular homology group of a topological space $\mathrm{X}$. A continuous map $f:\mathrm{X}\to\mathrm{Y}$ induces a homomorphism $f_*:\mathrm{H}_n\left(\mathrm{X};G\right)\to\mathrm{H}_n\left(\mathrm{Y};G\right)$ for each $n$, in a functorial way: $\left(g\circ f\right)_*=g_*\circ f_*$ and $\left(\operatorname{id}_\mathrm{X}\right)_*=\operatorname{id}$. For a pair $(\mathrm{X},\mathrm{A})$ with $\mathrm{A}\subset\mathrm{X}$ we write $\mathrm{H}_n\left(\left(\mathrm{X},\mathrm{A}\right);G\right)$ for the relative homology groups. When $\mathrm{X}$ is the geometric realization of a simplicial complex $\mathrm{K}$, singular homology agrees with the simplicial homology of $\mathrm{K}$ and for any abelian group $G$ the $d$-th homology depends only on the $\left(d+1\right)$-skeleton, that is $\mathrm{H}_d\left(\mathrm{K};G\right)=\mathrm{H}_d\left(\mathrm{K}^{\left(d+1\right)};G\right)$.

\begin{Def}
    Two continuous maps $f,g:\mathrm{X}\to\mathrm{Y}$ are \textbf{homotopic}, written $f\simeq g$, if there is a continuous map $\mathrm{F}:\mathrm{X}\times[0,1]\to\mathrm{Y}$ with $\mathrm{F}\left(\cdot,0\right)=f$ and $\mathrm{F}\left(\cdot,1\right)=g$. A map $f:\mathrm{X}\to\mathrm{Y}$ is a \textbf{homotopy equivalence} if there exists $g:\mathrm{Y}\to\mathrm{X}$ with $g\circ f\simeq\operatorname{id}_\mathrm{X}$ and $f\circ g\simeq\operatorname{id}_\mathrm{Y}$. We then say $\mathrm{X}$ and $\mathrm{Y}$ are \textbf{homotopy equivalent} and write $\mathrm{X}\simeq\mathrm{Y}$. Homotopic maps induce the same homomorphism on homology, so a homotopy equivalence induces an isomorphism $f_*:\mathrm{H}_n\left(\mathrm{X};G\right)\xrightarrow{\;\cong\;}\mathrm{H}_n\left(\mathrm{Y};G\right)$ for every $n$.
\end{Def}

\begin{Def}
    Let $\mathrm{A}\subset\mathrm{X}$. A \textbf{deformation retract} of $\mathrm{X}$ onto $\mathrm{A}$ is a continuous map $\mathrm{F}:\mathrm{X}\times[0,1]\to\mathrm{X}$ such that, for all $x\in\mathrm{X}$ and $a\in\mathrm{A}$, $\mathrm{F}\left(x,0\right)=x$, $\mathrm{F}\left(x,1\right)\in\mathrm{A}$ and $\mathrm{F}\left(a,1\right)=a$. If such an $\mathrm{F}$ exists we say $\mathrm{X}$ \textbf{deformation retracts} onto $\mathrm{A}$. The inclusion $\iota:\mathrm{A}\hookrightarrow\mathrm{X}$ then, is a homotopy equivalence and in particular $\iota_*:\mathrm{H}_n\left(\mathrm{A};G\right)\to\mathrm{H}_n\left(\mathrm{X};G\right)$ is an isomorphism for every $n$.
\end{Def}

\begin{Def}
    A set $\mathcal{M}\subset\mathbb{R}^n $ is a homology $d$-manifold if,
    $$\operatorname{H}_i \left(\left(\mathcal{M},\mathcal{M}\backslash\{x\}\right);\mathbb{Z}_2\right)=\begin{cases}
\mathbb{Z}_2 & i=d\\
\\
0 & i \neq d
\end{cases}$$
for any $x\in\mathcal{M}$. If the subset is compact and connected, the fundamental class, denoted by $\left[\mathcal{M}\right]$, is well defined and is the unique generator of the $d$-homology group with coefficients in $\mathbb{Z}_2$. Moreover, for a connected homology $d$-manifold, we have
$$\operatorname{H}_d \left(\mathcal{M};\mathbb{Z}_2\right)=\begin{cases}
\mathbb{Z}_2 & \text{if}\;\;\mathcal{M}\;\;\text{is compact}\\
\\
0 & \text{otherwise}
\end{cases}$$
\end{Def}

\begin{lemma}\label{lemma:contra}
    If $\mathcal{M}\subset \mathbb{R}^n$ is a compact connected homology $d$-manifold and $\mathrm{A}\subsetneq\mathcal{M}$ is open in the subspace topology, then
    $\mathrm{H}_d \left(\mathrm{A};\mathbb{Z}_2\right)=0$.
    \begin{proof}
       As $\mathrm{A}$ is a proper open subset of $\mathcal{M}$, for each $x\in\mathrm{A}\;\exists\; \mathrm{U}_x\underset{\text{open}}{\subset}\mathcal{M}$ with $x\in\mathrm{U}_x\subset \mathrm{A}$, thus,
       $$\operatorname{H}_i \left(\left(\mathrm{A},\mathrm{A}\backslash\{x\}\right);\mathbb{Z}_2\right)=\operatorname{H}_i \left(\left(\mathrm{U}_x,\mathrm{U}_x\backslash\{x\}\right);\mathbb{Z}_2\right)=\operatorname{H}_i \left(\left(\mathcal{M},\mathcal{M}\backslash\{x\}\right);\mathbb{Z}_2\right)\;\forall\,i$$ That is, $\mathrm{A}$ is a homology $d$-manifold. Note that $\mathrm{A}$ cannot be compact because if it were, then it would be closed and open in $\mathcal{M}$, contradicting the assumption that $\mathcal{M}$ is connected. Hence,  $\mathrm{H}_d \left(\mathrm{A};\mathbb{Z}_2\right)=0$
    \end{proof}
\end{lemma}



\subsection{Simplicial complexes, geometric realization and \v{C}ech complexes}

The bounds proved later are combinatorial in nature: they pass through simplicial complexes built on a metric space and the simplicial maps between them. We recall the relevant notions and then describe the construction, which may be viewed as a counterpart to the simplicial approximation theorem.

\begin{Def}
    An \textbf{abstract simplicial complex} $\mathrm{K}$ on a vertex set $\mathrm{V}$ is a collection of finite nonempty subsets of $\mathrm{V}$, called \textbf{simplices}, such that $\{v\}\in\mathrm{K}$ for every $v\in\mathrm{V}$ and every nonempty subset of a simplex is again a simplex (a \textbf{face} of it). A simplex with $k+1$ vertices is a $k$-simplex. The \textbf{$p$-skeleton} $\mathrm{K}^{\left(p\right)}$ is the subcomplex consisting of all simplices of dimension at most $p$.
\end{Def}

\begin{Def}
    The \textbf{geometric realization} $\left|\mathrm{K}\right|$ of a simplicial complex $\mathrm{K}$ is the topological space obtained by assigning to each vertex $v\in\mathrm{V}$ a basis vector $e_v$ of the real vector space $\mathbb{R}^{\left(\mathrm{V}\right)}$ and realizing each simplex $\sigma=\left[v_0,\dots,v_k\right]$ as the convex hull of $\{e_{v_0},\dots,e_{v_k}\}$, with the weak topology. When $\mathrm{K}$ is finite, $\left|\mathrm{K}\right|$ is a compact subset of a Euclidean space. We use the same symbol for $\mathrm{K}$ and its realization $\left|\mathrm{K}\right|$ when no confusion arises.
\end{Def}

\begin{Def}
    A \textbf{simplicial map} $\varphi:\mathrm{K}\to\mathrm{L}$ is a map of vertex sets $\varphi:\mathrm{V}_\mathrm{K}\to\mathrm{V}_\mathrm{L}$ such that $\varphi\left(\sigma\right)\in\mathrm{L}$ for every simplex $\sigma\in\mathrm{K}$. It induces a continuous map $\left|\mathrm{K}\right|\to\left|\mathrm{L}\right|$ by affine extension on each simplex. Two simplicial maps $\varphi,\psi:\mathrm{K}\to\mathrm{L}$ are \textbf{contiguous} if for every simplex $\sigma\in\mathrm{K}$ the set $\varphi\left(\sigma\right)\cup\psi\left(\sigma\right)$ is a simplex of $\mathrm{L}$. Contiguous simplicial maps induce homotopic maps on realizations and hence the same homomorphism on homology.
\end{Def}

\begin{Def}
    Let $\mathcal{U}=\{\mathrm{U}_i\}_{i\in\mathrm{I}}$ be a cover of a topological space. The \textbf{nerve} $\mathrm{N}\left(\mathcal{U}\right)$ is the abstract simplicial complex with vertex set $\mathrm{I}$ in which a finite subset $\sigma\subset\mathrm{I}$ is a simplex if and only if $\bigcap_{i\in\sigma}\mathrm{U}_i\neq\emptyset$. The cover is a \textbf{good cover} if every nonempty finite intersection $\bigcap_{i\in\sigma}\mathrm{U}_i$ is contractible. The Nerve Lemma then asserts that $\mathrm{N}\left(\mathcal{U}\right)$ is homotopy equivalent to $\bigcup_{i\in\mathrm{I}}\mathrm{U}_i$.
\end{Def}

\begin{Def}
    For $\mathcal{M}$ a subset of $\mathbb{R}^n$, $\mathrm{A}\subset\mathcal{M}$ and $r>0$, the \textbf{ambient \v{C}ech complex} of $\mathrm{A}$ at scale $r$, $\check{\mathrm{C}}\left(\mathrm{A};r\right)$, is the simplicial complex with vertex set $\mathrm{A}$ in which a finite subset $\left[x_0,\dots,x_k\right]\subset\mathrm{A}$ is a simplex if there is some $z\in\mathbb{R}^n$ with $\operatorname{d}\left(z,x_i\right)<r$ for all $0\leq i\leq k$, that is, if $\bigcap_{i=0}^{k}\mathbb{B}\left(x_i,r\right)\neq\emptyset$. The \textbf{intrinsic \v{C}ech complex} (intrinsic with respect to $\mathcal{M}$), $\check{\mathrm{C}}_\mathcal{M}\left(\mathrm{A};r\right)$, has the same vertex set $\mathrm{A}$ and admits $\left[x_0,\dots,x_k\right]$ as a simplex if there is some $z\in\mathcal{M}$ with $\operatorname{d}\left(z,x_i\right)<r$ for all $0\leq i\leq k$, that is, if $\bigcap_{i=0}^{k}\mathbb{B}\left(x_i,r\right)\cap\mathcal{M}\neq\emptyset$. The ambient and intrinsic \v{C}ech complexes are the nerves of the covers $\{\mathbb{B}\left(x,r\right)\}_{x\in\mathrm{A}}$ and $\{\mathbb{B}\left(x,r\right)\cap\mathcal{M}\}_{x\in\mathrm{A}}$ respectively. The \textbf{Vietoris-Rips complex} of $\mathrm{A}$ at scale $r>0$, $\mathrm{VR}\left(\mathrm{A};r\right)$, is the simplicial complex with vertex set $\mathrm{A}$ in which a finite subset $\sigma=\left[x_0,\dots,x_k\right]\subset\mathrm{A}$ is a simplex if $\operatorname{diam}\left(\sigma\right)<r$.
\end{Def}

In the following sections the argument (similar to \cite{adams2025hausdorffvsgromovhausdorffdistances}) is as follows:  build a simplicial complex, either an ambient or intrinsic \v{C}ech complex, or a Vietoris-Rips complex, on a subset $\mathrm{A}\subset\mathcal{M}$. Then we show that this complex is homotopy equivalent to $\mathcal{M}$ by a nerve-type lemma provided the scale $r$ is small relative to the reach $\operatorname{r}_1\left(\mathcal{M}\right)$. A correspondence of controlled distortion between $\mathcal{M}$ and $\mathrm{A}$ gives a vertex map between the two spaces. This map between the spaces themselves need not be continuous. As the distortion is bounded, this map nevertheless carries simplices to simplices and so extends to a simplicial map, whose geometric realization is continuous. Composing such maps and using that contiguous simplicial maps agree on homology, one tracks the fundamental class $\left[\mathcal{M}\right]$ through the resulting diagram and compares the induced homomorphisms with the inclusion. When the scale of the constructed complex is large compared with the Hausdorff (or Gromov-Hausdorff) distance, a nerve-type lemma forces the relevant top homology to vanish and the resulting contradiction yields the desired inequality between the two distances.

Figure \ref{fig:nerveball} gives a brief visual analogy, note that, the nerve complexes used in the proofs in the next sections do not triangulate the manifold, their geometric realizations are homotopic to the manifolds. If the Gromov-Hausdorff distance between the manifold($\mathcal{M}$) and a subset $\mathrm{A}$ is small then we can find a map of small distortion from $\mathcal{M}$ to $\mathrm{A}$.  For a triangulated manifold ( depicted as a deflated ball made of patches in Figure \ref{fig:nerveball}a) the chain represented by the sum of the
top-dimensional simplices (here triangular patches) is the fundamental class of the manifold [$\mathcal{M}$]. As the distortion of mapping $\mathcal{M}$ to $\mathrm{A}$ is small the triangulation given by the distorted simplices (skewed patches in Figure \ref{fig:nerveball}b) should preserve the fundamental class (inflate). But if the Hausdorff distance between $\mathcal{M}$ and $\mathrm{A}$ is too large compared to the Gromov-Hausdorff distance then the chosen map does not stretch (distort) the simplices enough to cover the manifold and thus should not support the fundamental class (should not inflate as Figure \ref{fig:nerveball}c) which leads to a contradiction.
\begin{figure}[htb]
\centering
\input{figure_ball_v1_final}
\caption{  }
\label{fig:nerveball}
\end{figure}

\subsection{Nerve lemmas and homological obstructions}
We apply the following Theorem of \cite{kim2019homotopy} to obtain a Nerve lemma for \v{C}ech complexes of compact subsets with positive $\mu$-reach.
\begin{thm}(Theorem 12 \cite{kim2019homotopy})
    Let $\mathcal{M} \subset \mathbb{R}^n$ be a subset with positive $\mu$-reach $\operatorname{r}_{\mu} \left(\mathcal{M}\right)>0$. For $r \leq \operatorname{r}_{\mu} \left(\mathcal{M}\right)$, the $r$-offset $\mathcal{M}^{r}$ deformation retracts to $\mathcal{M}$. In particular, $\mathcal{M}$ and $\mathcal{M}^{r}$ are homotopy equivalent.
\end{thm}

\begin{lemma}(Nerve Lemma for compact Euclidean subsets)\label{lemma:nervekim}
    Let $\mathcal{M}\subset\mathbb{R}^n$ be a compact subset with positive $\mu$-reach and let $0<r<\operatorname{r}_\mu \left(\mathcal{M}\right)$. Let $\mathcal{U} = \{{\mathbb{B} \left(x,r\right)}\}_{x\in\mathcal{M}}$ and $\mathcal{M}^r = \bigcup\limits_{x\in\mathcal{M}}\mathbb{B} \left(x,r\right) =  \bigcup\limits_{U\in\mathcal{U}}U$. Then the \v{C}ech nerve of the cover $\mathcal{U}$, $\check{\mathrm{C}} \left(\mathcal{M};r\right)$, is homotopy equivalent to the subset $\mathcal{M}$. Also, for $0<r'\leq r$ the inclusion $\iota:\check{\mathrm{C}} \left(\mathcal{M};r'\right)\hookrightarrow\check{\mathrm{C}} \left(\mathcal{M};r\right)$ induces a homotopy equivalence.
\begin{proof}
    Since the cover $\mathcal{U}$ is a good cover of $\mathcal{M}^r$, the nerve of $ \mathcal{U} $ is homotopy equivalent to $\mathcal{M}^r$. From the preceding theorem we have $\check{\mathrm{C}} \left(\mathcal{M};r\right)= N \left(\mathcal{U}\right)\simeq \bigcup\limits_{x\in\mathcal{M}}\mathbb{B} \left(x,r\right)=\mathcal{M}^r\simeq \mathcal{M}$. Note that the cover $\{\mathbb{B} \left(x,r'\right)\}_{x\in \mathcal{M}}$ is a refinement of $\mathcal{U}$ for $r'<r$. Thus, from the Nerve Lemma on paracompact topological spaces (\cite{Borsuk1948,Alexandroff1928,MR1867354}), the inclusion between the \v{C}ech nerves, $\iota:\check{\mathrm{C}} \left(\mathcal{M},r'\right)\hookrightarrow\check{\mathrm{C}} \left(\mathcal{M},r\right)$ induces a homotopy equivalence.
\end{proof}
\end{lemma}

The nerve of the cover $\left\{\mathbb{B}\left(x,r\right)\right\}_{x\in \mathrm{A}}$ for a subset $\mathrm{A}\subset\mathcal{M}$ may not have trivial $d$-homology even when it does not cover $\mathcal{M}$.  In \cite{kim2019homotopy} the authors prove a Nerve Lemma for the intrinsic \v{C}ech complexes of manifolds. We use this result to prove that if the cover $\left\{\mathbb{B}\left(x,r\right)\cap\mathcal{M}\right\}_{x\in \mathrm{A}}$ does not contain $\mathcal{M}$, then the $d$-homology of the \v{C}ech complex is trivial ($d$ being the dimension of $\mathcal{M}$). We use the following lemma from \cite{kim2019homotopy}.

\begin{lemma}[Theorem 9 \cite{kim2019homotopy}]\label{lemma:contractible}
 Let $\mathcal{M}\subset\mathbb{R}^n$ be a compact Euclidean submanifold with reach $\operatorname{r}_1(\mathcal{M})>0$ and $\left\{\mathbb{B}(x_i,r_i)\right\}_{0\leq i\leq k}$ be a finite collection of balls of radii $0<r_i<\sqrt{2}\operatorname{r}_1(\mathcal{M})$ centered at $x_i\in\mathcal{M}\;\;\forall\;\;0\leq i \leq k$. Then $\bigcap\limits_{i=0}^k \mathbb{B}(x_i,r_i)\cap\mathcal{M}$ is contractible whenever it is nonempty.
\end{lemma}
Using these lemmas we have the following lemma (see Corollary 10 (\cite{kim2019homotopy})) for intrinsic \v{C}ech complexes of a subset of a submanifold.
\begin{lemma}\label{lemma:obstruct}
    Let $\mathcal{M}\subset\mathbb{R}^n$ be a compact connected $d$-dimensional Euclidean submanifold with reach $\operatorname{r}_1(\mathcal{M})>0$, $\mathrm{A}\subset\mathcal{M}$ be a subset and $\mathcal{U}^{\mathcal{M}}_{r}$ be the cover $\left\{\mathbb{B}(x,r)\cap\mathcal{M}\right\}_{x\in\mathrm{A}}$ with $0<r<\sqrt2\operatorname{r}_1$. If  \hspace{1pt} $\mathcal{U}^{\mathcal{M}}_{r}$ does not cover $\mathcal{M}$, then $\mathrm{H}_d\left(\check{\mathrm{C}}_\mathcal{M}\left(\mathrm{A};r\right);\mathbb{Z}_2\right)=0$.
\begin{proof}
    From Lemma \ref{lemma:contractible} the cover $\mathcal{U}^{\mathcal{M}}_{r}$ is a good open cover. Thus the nerve of the cover that is, the \v{C}ech complex $\check{\mathrm{C}}_\mathcal{M}\left(\mathrm{A};r\right)$ is homotopy equivalent to the union $\bigcup\limits_{x\in\mathrm{A}}\mathbb{B}(x,r)\cap\mathcal{M}$. From the assumption that $\mathcal{U}^{\mathcal{M}}_{r}$ does not cover $\mathcal{M}$, $\bigcup\limits_{x\in\mathrm{A}}\mathbb{B}(x,r)\cap\mathcal{M}$ is a proper open subset of $\mathcal{M}$. Therefore using Lemma \ref{lemma:contra},\\ $\mathrm{H}_d\left(\check{\mathrm{C}}_\mathcal{M}\left(\mathrm{A};r\right);\mathbb{Z}_2\right)=\mathrm{H}_d\left(\bigcup\limits_{x\in\mathrm{A}}\mathbb{B}(x,r)\cap\mathcal{M};\mathbb{Z}_2\right)=0$.

    \end{proof}
\end{lemma}
The Nerve lemma for intrinsic \v{C}ech complexes allows the covering balls to have different radii. It identifies the homotopy type of the intrinsic \v{C}ech complex when the cover given by intrinsic balls with centers on $\mathrm{A}$ covers $\mathcal{M}$ and leads to the vanishing of the top homology when it does not.
\begin{lemma}\label{lemma:mapzero}
  Let $\mathcal{M}\subset\mathbb{R}^n$ be a compact connected $d$-dimensional Euclidean submanifold with reach $\operatorname{r}_1(\mathcal{M})>0$, $\mathrm{A}\subset\mathcal{M}$ a subset and $\mathcal{U}^{\mathcal{M}}_{r_\mathrm{A}}$ the cover $\left\{\mathbb{B}(x,r_x)\cap\mathcal{M}\right\}_{x\in\mathrm{A}}$ with $0<r_x<\sqrt2\operatorname{r}_1$. If $r_x<\operatorname{r}_1(\mathcal{M})$ for all $x$ and  $\mathcal{U}^{\mathcal{M}}_{r_\mathrm{A}}$ does not cover $\mathcal{M}$, then $\mathrm{H}_d\left(\check{\mathrm{C}}_\mathcal{M}\left(\mathrm{A};r_\mathrm{A}\right);\mathbb{Z}_2\right)=0$. If $r_x<\operatorname{r}_1(\mathcal{M})$ for all $x$ and  $\mathcal{U}^{\mathcal{M}}_{r_\mathrm{A}}$ covers $\mathcal{M}$, then $\check{\mathrm{C}}_\mathcal{M}\left(\mathrm{A};r_\mathrm{A}\right)\simeq\mathcal{M}$. Here, $\check{\mathrm{C}}_\mathcal{M}\left(\mathrm{A};r_\mathrm{A}\right)=\{\sigma=[z_1\dots z_k] \mid z_i\in \mathrm{A} \;\forall\; i\in\{1,\dots,k\}\;\forall k\in\mathbb{Z}^{>0}, \bigcap\limits_{i=1}^k\mathbb{B}(z_i,r_{z_i})\cap\mathcal{M}\neq \emptyset\}$ is the nerve of the cover $\mathcal{U}^{\mathcal{M}}_{r_\mathrm{A}}$.
\begin{proof}
    By Lemma \ref{lemma:contractible} every nonempty finite intersection $\bigcap\limits_{i=1}^k\mathbb{B}(z_i,r_{z_i})\cap\mathcal{M}$ is contractible, so $\mathcal{U}^{\mathcal{M}}_{r_\mathrm{A}}$ is a good open cover and its nerve $\check{\mathrm{C}}_\mathcal{M}\left(\mathrm{A};r_\mathrm{A}\right)$ is homotopy equivalent to the union $\bigcup\limits_{x\in\mathrm{A}}\mathbb{B}(x,r_x)\cap\mathcal{M}$. If this union equals $\mathcal{M}$, then $\check{\mathrm{C}}_\mathcal{M}\left(\mathrm{A};r_\mathrm{A}\right)\simeq\mathcal{M}$. Otherwise the union is a proper open subset of $\mathcal{M}$ and Lemma \ref{lemma:contra} gives $\mathrm{H}_d\left(\check{\mathrm{C}}_\mathcal{M}\left(\mathrm{A};r_\mathrm{A}\right);\mathbb{Z}_2\right)=\mathrm{H}_d\left(\bigcup\limits_{x\in\mathrm{A}}\mathbb{B}(x,r_x)\cap\mathcal{M};\mathbb{Z}_2\right)=0$.
\end{proof}
\end{lemma}

\section{Inequalities using intrinsic \v{C}ech complexes}\label{sec:intrincech}
Throughout this section $\mathcal{M}\subset\mathbb{R}^n$ is a closed, connected, $d$-dimensional submanifold of positive reach $\operatorname{r}_1(\mathcal{M})>0$ and $\mathrm{A}\subset\mathcal{M}$ is a subset. In this section we prove a lower bound for the Gromov-Hausdorff distance between such a submanifold and a subset, in terms of their Hausdorff distance, without any further assumption on the subset. The argument uses Jung's theorem to produce, from a correspondence of small distortion, simplicial maps interleaving the ambient \v{C}ech complex of $\mathcal{M}$ with the intrinsic \v{C}ech complex of $\mathrm{A}$.
\begin{thm}[Jung's theorem (\cite{Jung1901})]\label{theorem:Jung}
    Let $\mathrm{K}\subset\mathbb{R}^n$ be a compact set with diameter $\operatorname{diam}(\mathrm{K}) = \sup\limits_{x,y\in\mathrm{K}}\operatorname{d}(x,y)$. There exists a closed ball with radius less than or equal to $\sqrt{\frac{n}{2(n+1)}}\operatorname{diam}(\mathrm{K})$ containing $\mathrm{K}$.
\end{thm}
\noindent We will also require the following result from \cite{kim2019homotopy},
\begin{lemma}[Lemma 33 and 34 in \cite{kim2019homotopy}]\label{lemma:estimate}
    Let $\mathcal{M}\subset\mathbb{R}^n$ be a compact Euclidean submanifold with reach $\operatorname{r}_1(\mathcal{M})>0$ then for $x\in\mathcal{M}$ and $y\in\mathbb{R}^n$ such that $\operatorname{d}_\mathcal{M}(y)<\operatorname{r}_1(\mathcal{M})$ we have,
    \begin{align*}
    & 1.) \hspace{1cm}\operatorname{d}(x,\pi_\mathcal{M}(y))\leq\sqrt{\frac{\operatorname{r}_1(\mathcal{M})}{\operatorname{r}_1(\mathcal{M})-\operatorname{d}_\mathcal{M}(y)}\left(\operatorname{d}(x,y)^2-\operatorname{d}_\mathcal{M}(y)^2\right)}\leq \sqrt{\frac{\operatorname{r}_1(\mathcal{M})}{\operatorname{r}_1(\mathcal{M})-\operatorname{d}_\mathcal{M}(y)}}\operatorname{d}(x,y)\\
    & 2.) \hspace{1cm}\operatorname{d}(x,\pi_\mathcal{M}(y))\leq \sqrt{\frac{2\operatorname{r}_1(\mathcal{M})}{\operatorname{r}_1(\mathcal{M})+\sqrt{\operatorname{r}_1(\mathcal{M})^2-\operatorname{d}(x,y)^2}}}\operatorname{d}(x,y) \text{\;\;\;\;if $\operatorname{d}(x,y)<\operatorname{r}_1(\mathcal{M})$}
    \end{align*}
    
\end{lemma}

\begin{lemma}[Correspondences and intrinsic \v{C}ech complex]\label{lemma:intrinsicCech}
    Consider a Euclidean submanifold $\mathcal{M}\subset\mathbb{R}^n$ with positive reach $\operatorname{r}_1>0$ and a subset $\mathrm{A}\subset\mathcal{M}$. For $r>2 \operatorname{d}_{\mathrm{GH}} \left(\mathrm{A},\mathcal{M}\right)$ and $\epsilon<\operatorname{r}_1(\mathcal{M})$ there exist simplicial maps $\bar{h},\bar{g}$,
$$\check{\mathrm{C}} \left(\mathcal{M};\epsilon\right)\xrightarrow[\hspace{1cm}]{\bar{h}} \check{\mathrm{C}}_\mathcal{M} \left(\mathrm{A};r+\sqrt{\frac{\operatorname{r}_1}{\operatorname{r}_1-\epsilon}}\,\epsilon\right)\xrightarrow[\hspace{1cm}]{\bar{g}}\check{\mathrm{C}} \left(\mathcal{M};\frac{3}{\sqrt2}r+\sqrt{\frac{2\operatorname{r}_1}{\operatorname{r}_1-\epsilon}}\,\epsilon\right)$$
and the composition $\bar{g}\circ\bar{h}$ is contiguous to inclusion
$\iota:\check{\mathrm{C}} \left(\mathcal{M};\epsilon\right)\xhookrightarrow[\hspace{0.5cm}]{}\check{\mathrm{C}} \left(\mathcal{M};\frac{3}{\sqrt2}r+\sqrt{\frac{2\operatorname{r}_1}{\operatorname{r}_1-\epsilon}}\,\epsilon\right)$.
\begin{proof}
As $r>2\operatorname{d}_{\mathrm{GH}}(\mathrm{A},\mathcal{M})$ there exists a correspondence $\mathrm{C}\subset\mathrm{A}\times\mathcal{M}$ with $2\operatorname{d}_{\mathrm{GH}}(\mathrm{A},\mathcal{M})<\operatorname{dis}(\mathrm{C})<r$ thus, there exists maps $h:\mathcal{M}\to \mathrm{A}$ and $g:\mathrm{A}\to \mathcal{M}$, such that, $2\operatorname{d}_{\mathrm{GH}}(\mathrm{A},\mathcal{M})<\operatorname{dis}(h)<r$ and $2\operatorname{d}_{\mathrm{GH}}(\mathrm{A},\mathcal{M})<\operatorname{dis}(g)<r$. Seen as a vertex map we show that $h$ induces a simplicial map.
For a simplex $\sigma = [z_1,\dots,z_k]$ in $\check{\mathrm{C}} \left(\mathcal{M};\epsilon\right)$, there exists a ball in $\mathbb{R}^n$, such that $\sigma \subset \mathbb{B} \left(a_{\sigma},\epsilon\right)$ hence, $\operatorname{d} \left(a_\sigma,z_j\right)<\epsilon\;\;\forall j \in\{1,2,\dots,k\}$. As $\epsilon<\operatorname{r}_1$ from Lemma \ref{lemma:estimate} we have $\operatorname{d}(\pi_\mathcal{M}(a_\sigma),z_i)<\sqrt{\frac{\operatorname{r}_1}{\operatorname{r}_1-\epsilon}}\,\epsilon$. As $\operatorname{dis} \left(h\right)<r$, for all $1\leq i \leq k$
    \begin{align*}
        &|\operatorname{d} \left(h\circ\pi_{\mathcal{M}}(a_\sigma),h(z_i)\right)-\operatorname{d} \left(\pi_{\mathcal{M}}(a_\sigma),z_i\right)|<r,\\
        &\operatorname{d} \left(h\circ\pi_{\mathcal{M}}(a_\sigma),h(z_i)\right) < r+ \operatorname{d} \left(\pi_{\mathcal{M}}(a_\sigma),z_i\right)
    \end{align*}
Thus, $\bar{h}(\sigma)\subset\mathbb{B}(h\circ\pi_{\mathcal{M}}(a_\sigma),r+\sqrt{\frac{\operatorname{r}_1}{\operatorname{r}_1-\epsilon}}\,\epsilon)$. As $\pi_\mathcal{M}(a_\sigma)\in\mathcal{M}$ we also have $h\circ\pi_\mathcal{M}(a_\sigma)\in\mathrm{A}\subset\mathcal{M}$, therefore, $\bar{h}(\sigma)\in \check{\mathrm{C}}_\mathcal{M} \left(\mathrm{A};r+\sqrt{\frac{\operatorname{r}_1}{\operatorname{r}_1-\epsilon}}\epsilon\right)$.
Now for any $\tau=[y_1,\dots,y_k]\in  \check{\mathrm{C}}_\mathcal{M} \left(\mathrm{A};r+\sqrt{\frac{\operatorname{r}_1}{\operatorname{r}_1-\epsilon}}\epsilon\right)$
there exists $b_\tau\in\mathcal{M}\subset\mathbb{R}^n$ such that $\operatorname{d}(b_\tau,y_i)<r+\sqrt{\frac{\operatorname{r}_1}{\operatorname{r}_1-\epsilon}}\,\epsilon$ for each $1\leq i\leq k$, thus $\operatorname{d}(y_i,y_j)< 2\left(r+\sqrt{\frac{\operatorname{r}_1}{\operatorname{r}_1-\epsilon}}\,\epsilon\right)$ for all $1\leq i,j\leq k$. As $\operatorname{dis}(g)<r$ we also have $\operatorname{d}\left(g(y_i),g\left(y_j\right)\right)<3r+2\sqrt{\frac{\operatorname{r}_1}{\operatorname{r}_1-\epsilon}}\,\epsilon$. By Jung's theorem (Theorem \ref{theorem:Jung}) there exists a $y\in\mathbb{R}^n$ such that $\operatorname{d}(g(y_i),y)<\frac{1}{\sqrt2}(3r+2\sqrt{\frac{\operatorname{r}_1}{\operatorname{r}_1-\epsilon}}\,\epsilon)$. Thus $\bar{g}(\tau)\in\check{\mathrm{C}} \left(\mathcal{M};\frac{3}{\sqrt2}r+\sqrt{\frac{2\operatorname{r}_1}{\operatorname{r}_1-\epsilon}}\,\epsilon\right)$.
Also note that, for $\sigma=[z_1,\dots,z_k]\in \check{\mathrm{C}} \left(\mathcal{M};\epsilon\right)$ and $a_\sigma$ as above, for each $1\leq i\leq k$, we have $(h(z_i),g\circ h(z_i)),(h\circ\pi_{\mathcal{M}}(a_\sigma),\pi_{\mathcal{M}}(a_\sigma))\in \mathrm{C}$. Thus, for all $1\leq i\leq k$,
\begin{align*}
    &\operatorname{d}(\pi_{\mathcal{M}}(a_\sigma),g\circ h(z_i))<\operatorname{d}(h\circ\pi_{\mathcal{M}}(a_\sigma),h(z_i))+r\\
    &\phantom{\operatorname{d}(\pi_{\mathcal{M}}(a_\sigma),g\circ h(z_i))}< 2r+\sqrt{\frac{\operatorname{r}_1}{\operatorname{r}_1-\epsilon}}\,\epsilon<\frac{3}{\sqrt2}r+\sqrt{\frac{2\operatorname{r}_1}{\operatorname{r}_1-\epsilon}}\,\epsilon
\end{align*}
Thus, $\iota(\sigma)\cup\bar{g}\circ\bar{h}(\sigma)\subset\mathbb{B}(\pi_{\mathcal{M}}(a_\sigma),\frac{3}{\sqrt2}r+\sqrt{\frac{2\operatorname{r}_1}{\operatorname{r}_1-\epsilon}}\,\epsilon)$. So the union  $\iota(\sigma)\cup\bar{g}\circ\bar{h}(\sigma)$ forms a simplex in the \v{C}ech complex $\check{\mathrm{C}} \left(\mathcal{M};\frac{3}{\sqrt2}r+\sqrt{\frac{2\operatorname{r}_1}{\operatorname{r}_1-\epsilon}}\,\epsilon\right)$. Thus, $\iota$ and $\bar{g}\circ\bar{h}$ are contiguous simplicial maps.
\end{proof}
\end{lemma}
We are now ready to prove the first result about lower bounding Gromov-Hausdorff distances with Hausdorff distances in Euclidean submanifolds.
\begin{thm}\label{thm:firstbound}
    Let $\mathcal{M}\subset\mathbb{R}^n$ be a closed connected $d$-dimensional submanifold of reach $\operatorname{r}_1(\mathcal{M})=\operatorname{r}_1>0$ and $\mathrm{A}$ be a subset of $\mathcal{M}$ then,
    $$\operatorname{d}_{\mathrm{GH}}(\mathrm{A},\mathcal{M})\geq\min\left\{\frac{1}{2}\operatorname{d}_{\mathrm{H}}(\mathrm{A},\mathcal{M}),\frac{\sqrt2}{6}\operatorname{r}_1(\mathcal{M})\right\}$$

\begin{proof}
    Assume that $\operatorname{d}_{\mathrm{GH}}(\mathrm{A},\mathcal{M})<\frac{\sqrt2}{6}\operatorname{r}_1$. Pick $r>0$ such that $2\operatorname{d}_{\mathrm{GH}}(\mathrm{A},\mathcal{M})<r<\frac{\sqrt2}{3}\operatorname{r}_1$ and let $\epsilon>0$ be small enough such that,  $\frac{3}{\sqrt2}r+\sqrt{\frac{2\operatorname{r}_1}{\operatorname{r}_1-\epsilon}}\,\epsilon<\operatorname{r}_1$ (Note that, $\sqrt{\frac{2\operatorname{r}_1}{\operatorname{r}_1-\epsilon}}\,\epsilon$, is an increasing function of $\epsilon$ and is 0 when $\epsilon=0$).
 As $r>2\operatorname{d}_{\mathrm{GH}}(\mathrm{A},\mathcal{M})$, from Lemma \ref{lemma:intrinsicCech} we have simplicial maps $\bar{h}$ and $\bar{g}$ such that
 $$\check{\mathrm{C}} \left(\mathcal{M};\epsilon\right)\xrightarrow[\hspace{1cm}]{\bar{h}} \check{\mathrm{C}}_\mathcal{M} \left(\mathrm{A};r+\sqrt{\frac{\operatorname{r}_1}{\operatorname{r}_1-\epsilon}}\,\epsilon\right)\xrightarrow[\hspace{1cm}]{\bar{g}}\check{\mathrm{C}} \left(\mathcal{M};\frac{3}{\sqrt2}r+\sqrt{\frac{2\operatorname{r}_1}{\operatorname{r}_1-\epsilon}}\,\epsilon\right)$$
   and $\bar{g}\circ\bar{h}$ is contiguous to the inclusion $\iota:\check{\mathrm{C}} \left(\mathcal{M};\epsilon\right)\xhookrightarrow[\hspace{0.5cm}]{}\check{\mathrm{C}} \left(\mathcal{M};\frac{3}{\sqrt2}r+\sqrt{\frac{2\operatorname{r}_1}{\operatorname{r}_1-\epsilon}}\,\epsilon\right)$. As $\frac{3}{\sqrt2}r+\sqrt{\frac{2\operatorname{r}_1}{\operatorname{r}_1-\epsilon}}\,\epsilon<\operatorname{r}_1$,  from Lemma \ref{lemma:nervekim}, we have, $\check{\mathrm{C}} \left(\mathcal{M};\epsilon\right)\simeq\check{\mathrm{C}} \left(\mathcal{M};\frac{3}{\sqrt2}r+\sqrt{\frac{2\operatorname{r}_1}{\operatorname{r}_1-\epsilon}}\,\epsilon\right)\simeq\mathcal{M}$ where the inclusion $\iota$ is a homotopy equivalence. As $\bar{g}\circ\bar{h}$  is contiguous to $\iota$ they induce the same isomorphism on homology $\iota_*=\bar{g}_*\circ\bar{h}_*$. Thus the fundamental class of $\mathcal{M}$, $[\mathcal{M}]\in\mathrm{H}_d(\mathcal{M},\mathbb{Z}_2)$ is mapped non-trivially by $\bar{g}_*\circ\bar{h}_*$. Now if, $r+\sqrt{\frac{\operatorname{r}_1}{\operatorname{r}_1-\epsilon}}\,\epsilon <\operatorname{d}_{\mathrm{H}}(\mathrm{A},\mathcal{M})$ then there exists $z\in\mathcal{M}$ such that $\operatorname{d}_\mathrm{A}(z)>r+\sqrt{\frac{\operatorname{r}_1}{\operatorname{r}_1-\epsilon}}\,\epsilon$ so the union $\bigcup\limits_{x\in \mathrm{A}}\mathbb{B}\left(x,r+\sqrt{\frac{\operatorname{r}_1}{\operatorname{r}_1-\epsilon}}\,\epsilon\right)\cap\mathcal{M}$ does not cover $\mathcal{M}$ and therefore is a proper open set of $\mathcal{M}$. From Lemma \ref{lemma:obstruct} we have $\mathrm{H}_d(\check{\mathrm{C}}_\mathcal{M} \left(\mathrm{A};r+\sqrt{\frac{\operatorname{r}_1}{\operatorname{r}_1-\epsilon}}\,\epsilon\right);\mathbb{Z}_2)=0$, which contradicts the fact that $\bar{g}_*\circ\bar{h}_*$ map the fundamental class non-trivially. Hence, $r+\sqrt{\frac{\operatorname{r}_1}{\operatorname{r}_1-\epsilon}}\,\epsilon \geq\operatorname{d}_{\mathrm{H}}(\mathrm{A},\mathcal{M})$ for all $2\operatorname{d}_{\mathrm{GH}}(\mathrm{A},\mathcal{M})<r<\frac{\sqrt2}{3}\operatorname{r}_1$ and $\epsilon>0$ small enough. Hence, $\operatorname{d}_{\mathrm{GH}}(\mathrm{A},\mathcal{M}) \geq\frac{1}{2}\operatorname{d}_{\mathrm{H}}(\mathrm{A},\mathcal{M})$.
\end{proof}
\end{thm}
\section{The case of the circle}\label{sec:circle}
    Before we further improve bounds in the theorem above and construct bounds for the Gromov-Hausdorff distance we try to get better bounds in the case where the underlying manifold is the circle. In \cite{adams2025lowerboundinggromovhausdorffdistance} the authors produce bounds for the circle (with the length space metric) using the Borsuk-Ulam theorem. We next extend this method to the circle seen as a metric subspace of $\mathbb{R}^2$. In \cite{adams2025lowerboundinggromovhausdorffdistance} the authors prove that the Gromov-Hausdorff distance of a subset of a circle from the circle is equal to its Hausdorff distance from the circle, we obtain a similar equality between the Gromov-Hausdorff distance and a function of the Hausdorff distance. In this section we prove these bounds.
\begin{Def}
For a continuous map $f:\mathbb{S}^1\to\mathbb{S}^1$, let the homomorphism induced on the homology groups with coefficients in $\mathbb{Z}$ be $f_*:\mathbb{Z}\cong\mathrm{H}_1(\mathbb{S}^1)\to\mathrm{H}_1(\mathbb{S}^1)\cong\mathbb{Z}$ and let $f_*(1)=\lambda\in\mathbb{Z}$ then $\lambda$ is said to be
the \textbf{degree of the map}, $\operatorname{deg}(f)$.
\end{Def}
\noindent We will also use the following property of maps between circles.
\begin{prop}[Problem 21 Chapter 9 \cite{Armstrong2010-jj}]\label{prop:degree}
If a continuous map $f:\mathbb{S}^1\to\mathbb{S}^1$ has even degree then there exists $x\in \mathbb{S}^1$ such that $f(x)=f(-x)$.
\end{prop}
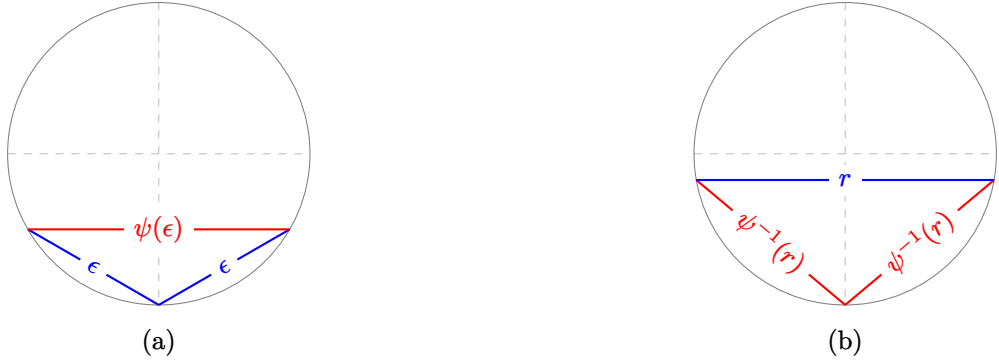
\begin{figure}[htb]
\centering
\begin{subfigure}{0.45\textwidth}
\centering
\begin{tikzpicture}[scale=1]
\draw[dashed, lightgray] (-2, 0) -- (2, 0);
\draw[dashed, lightgray] (0, 2) -- (0, -2);
\draw[gray] (0,0) circle (2);
\path[draw,blue,thick,name=1] (210:2)--(270:2) node[midway,sloped,fill=white] (){\small$\epsilon$};
\path[draw,blue,thick,name=2] (270:2)--(330:2) node[midway,sloped,fill=white] (){\small$\epsilon$};
\path[draw,red,thick,name=3] (210:2)--(330:2) node[midway,sloped,fill=white] (){\small$\psi(\epsilon)$};
\end{tikzpicture}
\caption*{(a)}
\end{subfigure}\hfill
\begin{subfigure}{0.45\textwidth}
\centering
\begin{tikzpicture}[scale=1]
\draw[dashed, lightgray] (-2, 0) -- (2, 0);
\draw[dashed, lightgray] (0, 2) -- (0, -2);
\draw[gray] (0,0) circle (2);
\path[draw,red,thick,name=1] (190:2)--(270:2) node[midway,sloped,fill=white] (){\small$\psi^{-1}(r)$};
\path[draw,red,thick,name=2] (270:2)--(350:2) node[midway,sloped,fill=white] (){\small$\psi^{-1}(r)$};
\path[draw,blue,thick,name=3] (190:2)--(350:2) node[midway,sloped,fill=white] (){\small$r$};
\end{tikzpicture}
\caption*{(b)}
\end{subfigure}\hfill
\caption{The chord functions $\psi$ and $\psi^{-1}$, drawn on an enlarged circle. In (a) two chords of length $\epsilon$ subtend equal arcs and the chord spanning both has length $\psi(\epsilon)$. In (b) a chord of length $r$ spans two chords of length $\psi^{-1}(r)$, each subtending half its arc.}
\label{fig:psidef}
\end{figure}
\begin{Def}
For $\epsilon\in[0,\sqrt2]$ define $\psi:[0,\sqrt2]\to[0,2]$ where $\psi(\epsilon)= 2\sin(2\sin^{-1}(\epsilon/2))$ and its inverse $\psi^{-1}:[0,2]\to[0,\sqrt2]$ as $\psi^{-1}(r)= 2\sin\left(\frac{1}{2}\sin^{-1}\left(\frac{r}{2}\right)\right)$. These represent the lengths of chords as shown in Figure \ref{fig:psidef}. Note that both $\psi$ and $\psi^{-1}$ are increasing.
\end{Def}

\begin{thm}\label{theorem:circbu}
Let $\mathbb{S}^1\subset\mathbb{R}^2$ be the metric space denoting the unit circle equipped with the Euclidean metric. Then for any subset $\mathrm{A}\subset\mathbb{S}^1$ with $\operatorname{d}_{\mathrm{H}}\left(\mathrm{A},\mathbb{S}^1\right)\leq\sqrt2$ we have,
$$\operatorname{d}_{\mathrm{GH}}\left(\mathrm{A},\mathbb{S}^1\right)\geq\min\left\{\frac{1}{2}\psi\left(\operatorname{d}_{\mathrm{H}}\left(\mathrm{A},\mathbb{S}^1\right)\right),\frac{c_0}{2}\right\}$$ where $c_0\in(0,2)$ is the solution of the equation $\psi^{-1}(x)+x=2$ ($c_0\approx1.305$).
\begin{proof}
    Assume that $\operatorname{d}_{\mathrm{GH}}\left(\mathrm{A},\mathbb{S}^1\right)<\frac{c_0}{2}$. Then let $f:\mathbb{S}^1\to\mathrm{A}$ be any map such that $2\operatorname{d}_{\mathrm{GH}}\left(\mathrm{A},\mathbb{S}^1\right)<\operatorname{dis}(f)<c_0$, choose $0<\epsilon<1$ such that $\operatorname{dis}(f)+\epsilon<c_0$. Let $\mathrm{K}$ be a one-dimensional finite simplicial complex triangulating $\mathbb{S}^1$ in which the diameter of each $1$-simplex is less than $\epsilon$ (any regular $m$-gon inscribed in the unit circle with side length less than $\epsilon$ can be chosen as $\mathrm{K}$). Let the vertex set of $\mathrm{K}$ be ordered counterclockwise as $\mathrm{K}^{(0)}=\{v_1,v_2,\dots v_m\}$. We define the map $\bar{f}:\mathrm{K}\to\mathbb{S}^1$
    $$\bar{f}(x)= \begin{cases} f(x) & \text { for } x\in\mathrm{K}^{(0)} \\  \frac{\lambda f(v_i)+(1-\lambda)f(v_{i+1})}{\lVert\lambda f(v_i)+(1-\lambda)f(v_{i+1})\rVert} & \text{ for } x=\lambda v_i+(1-\lambda)v_{i+1}, \forall\lambda \in[0,1] \end{cases}$$
    Now, $\bar{f}:\mathrm{K}\to\mathbb{S}^1$ is a continuous map. If $\operatorname{dis}(f)+\epsilon<\psi\left(\operatorname{d}_{\mathrm{H}}\left(\mathrm{A},\mathbb{S}^1\right)\right)$ then there exists $x\in\mathbb{S}^1$ such that $\mathbb{B}(x,\psi^{-1}(\operatorname{dis}(f)+\epsilon))\cap\mathrm{A}=\emptyset$. Thus, there exists no vertex $v_i\in\mathrm{K}^{(0)}$ such that $f(v_i)\in\mathbb{B}(x,\psi^{-1}\left(\operatorname{dis}(f)+\epsilon\right))\cap\mathbb{S}^1$.
    Now if $x\in \bar{f}(\mathbb{S}^1)$ then there exists $i\in\{1,2\dots,m\}$ such that $x\in \bar{f}([v_i,v_{i+1}])$ and the line joining the origin and $x$ intersects the segment $\{\lambda f(v_i)+(1-\lambda)f(v_{i+1})\,|\,\lambda\in[0,1]\}$. As $\operatorname{d}(f(v_i),x)>\psi^{-1}\left(\operatorname{dis}(f)+\epsilon\right)$ and $\operatorname{d}(f(v_{i+1}),x)>\psi^{-1}\left(\operatorname{dis}(f)+\epsilon\right)$ we would have $\operatorname{d}(f(v_i),f(v_{i+1}))\geq\psi(\psi^{-1}\left(\operatorname{dis}(f)+\epsilon\right))=\operatorname{dis}(f)+\epsilon$ (see Figure \ref{fig:int}a) which would be a contradiction. Thus $x\notin\bar{f}(\mathbb{S}^1)$, that is, $\bar{f}$ is not surjective, so $\operatorname{deg}(\bar{f})=0$ and by Proposition \ref{prop:degree}, $\bar{f}$ identifies a pair of antipodal points say $u,-u\in\mathrm{K}$, with $\bar{f}(u)=\bar{f}(-u)=p$. Say $u\in[v_i,v_{i+1}]$ and $-u\in[v_j,v_{j+1}]$. Each of the chords $[f(v_i),f(v_{i+1})]$ and $[f(v_j),f(v_{j+1})]$ meets the ray $\{tp\mid t>0\}$, so $p$ lies on the arc that chord cuts off on the far side of the origin (see Figure \ref{fig:int}b). Both arcs therefore contain $p$ and, as the length of these chords is less than $\operatorname{dis}(f)+\epsilon$, both have angular width less than $2\sin^{-1}\left(\frac{\operatorname{dis}(f)+\epsilon}{2}\right)$, so some endpoint of the first and some endpoint of the second are separated by at most half that angle. Hence there exists $v\in\{v_i,v_{i+1}\}$ and $w\in\{v_j,v_{j+1}\}$ such that $\operatorname{d}(f(v),f(w))\leq \psi^{-1}(\operatorname{dis}(f)+\epsilon)$. Note that as $u\in [v_i,v_{i+1}]$ and $-u\in[v_j,v_{j+1}]$ we have $$\operatorname{d}(v,w)\geq 2\sin(\frac{1}{2}(\pi-4\sin^{-1}(\epsilon/2)))=2\cos(2\sin^{-1}(\epsilon/2))=2(1-2\sin^2(\sin^{-1}(\epsilon/2)))=2-\epsilon^2\geq 2-\epsilon$$
    Hence,
    \begin{align*}
        &\operatorname{dis}(f)\geq \operatorname{d}(v,w)-\operatorname{d}(f(v),f(w))\geq 2-\epsilon-\psi^{-1}(\operatorname{dis}(f)+\epsilon)\\
        &\operatorname{dis}(f)+\epsilon + \psi^{-1}(\operatorname{dis}(f)+\epsilon)\geq2
    \end{align*}
As $\psi^{-1}(x)+x$ is an increasing function, this would imply that $\operatorname{dis}(f)+\epsilon\geq c_0$ which is a contradiction. Therefore, $\operatorname{dis}(f)\geq \psi\left(\operatorname{d}_{\mathrm{H}}\left(\mathrm{A},\mathbb{S}^1\right)\right)$ for any $f:\mathbb{S}^1\to \mathrm{A}$ with $2\operatorname{d}_{\mathrm{GH}}\left(\mathrm{A},\mathbb{S}^1\right)<\operatorname{dis}(f)<c_0$, therefore $\operatorname{d}_{\mathrm{GH}}\left(\mathrm{A},\mathbb{S}^1\right)\geq\frac{1}{2}\psi\left(\operatorname{d}_{\mathrm{H}}\left(\mathrm{A},\mathbb{S}^1\right)\right)$
\end{proof}
\end{thm}
\begin{figure}[htb]
\centering
\begin{subfigure}{0.45\textwidth}
\centering
\begin{tikzpicture}[scale=1]
\coordinate (o) at (0,0);
\coordinate (x) at (45:4);
\coordinate (y) at (10:4);
\coordinate (z) at (80:4);
\coordinate (a) at (5:4);
\coordinate (b) at (85:4);
\fill (o) circle (1pt);
\draw (o) node[node font=\tiny,anchor=north east] {$O$};
\draw[red!50!black,dashed] (10:4) arc(10:80:4) node[midway,above,sloped,fill=white,inner sep=8pt,node font=\tiny] (){$\mathbb{S}^1\setminus\mathrm{A}$};
\draw[black] (80:4) arc(80:90:4) ;
\draw[black] (0:4) arc(0:10:4);
\draw[black](0:4)  node[node font=\tiny,anchor=north] {$\mathrm{A}$};
\draw[black] (90:4) node[node font=\tiny,anchor=east] {$\mathrm{A}$};
\draw[blue](85:4) node[node font=\tiny,anchor=south,inner sep=4pt] {$f(v_{i+1})$};
\draw[blue] (5:4) node[node font=\tiny,anchor=north west,inner sep=2pt] {$f(v_i)$};
\draw[red!50!black] (x) node[node font=\tiny,inner sep=1pt,anchor=south west] {$x$};
\path[draw,dotted,black] (o)--(x);
\path[draw,dashed,red!50!black] (x)--(y) node[midway,fill=white,inner sep=0pt,sloped,node font=\tiny] (){$\operatorname{d}_{\mathrm{H}}$};
\path[draw,dashed,red!50!black] (y)--(z) node[midway,fill=white,inner sep=0pt,sloped,node font=\tiny] (){$\psi(\operatorname{d}_{\mathrm{H}})$};
\path[draw,dashed,blue] (a)--(b) node[midway,fill=white,inner sep=0pt,sloped,node font=\tiny] (){};

\fill[blue] (5:4) circle(1pt);
\fill[blue] (85:4) circle(1pt);
\fill[red!50!black] (x) circle(1pt);
\end{tikzpicture}
\caption*{(a)}
\end{subfigure}\hfill
\begin{subfigure}{0.45\textwidth}
\centering
\begin{tikzpicture}[scale=1]
\draw[black] (-10:4) arc(-10:85:4);
\draw[black,dashed] (85:4) arc(85:90:4);
\coordinate (a) at (-10:4);
\coordinate (b) at (60:4);
\coordinate (c) at (10:4);
\coordinate (d) at (85:4);
\coordinate (e) at (47.5:4);
\fill[blue] (a) circle (1pt);
\fill[blue] (b) circle (1pt);
\fill[red] (c) circle (1pt);
\fill[red] (d) circle (1pt);
\fill[green] (e) circle (1pt);
\draw[blue] (a) node[node font=\tiny,anchor=west] {$f(v_i)$};
\draw[blue] (b) node[node font=\tiny,anchor=south west] {$f(v_{i+1})=f(v)$};
\draw[red] (c) node[node font=\tiny,anchor=south west] {$f(v_j)$};
\draw[red] (d) node[node font=\tiny,anchor=south] {$f(v_{j+1})=f(w)$};
\path[draw,blue,thin](a)--(b) node[midway,sloped,fill=white,inner sep=0pt,node font=\tiny] (){$r$};
\path[draw,red,thin] (c)--(d) node[midway,sloped,fill=white,inner sep=0pt,node font=\tiny] (){$r$};
\path[draw,,dashed,green!70!black,thin] (c)--(e) node[midway,fill=white,inner sep=0pt,sloped,node font=\tiny] (){$\psi^{-1}(r)$};
\path[draw,dashed,green!70!black,thin] (d)--(e);
\path[draw,dashed,red!50!black,thin] (b)--(d);
\end{tikzpicture}
\caption*{(b)}
\end{subfigure}\hfill
\caption{(a) shows that if $x\in\bar{f}(\mathrm{K})$ then $\operatorname{d}(f(v_i),f(v_{i+1}))\geq\operatorname{dis}(f)+\epsilon$. In (b) note that $\operatorname{d}(f(v),f(w))\leq\psi^{-1}(r)$ for $r=\operatorname{dis}(f)+\epsilon$.}
\label{fig:int}
\end{figure}
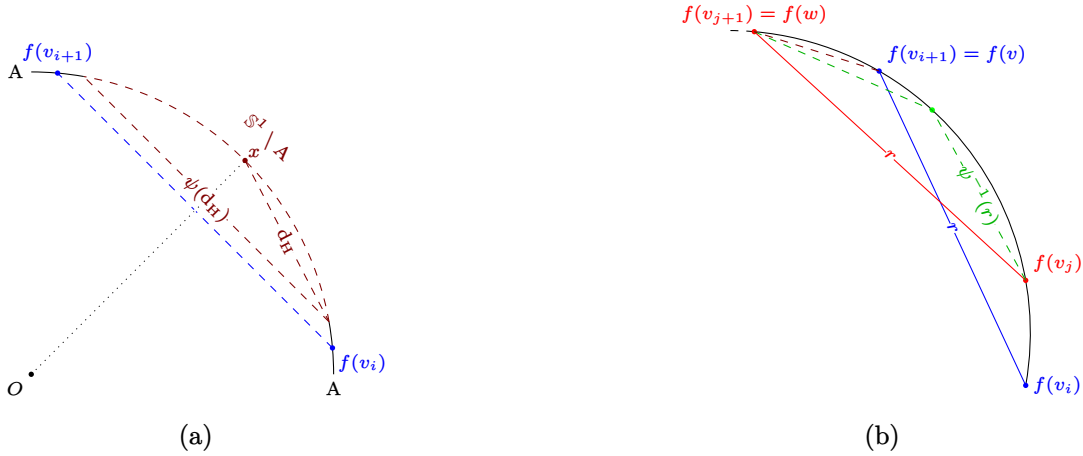

We next see that we can construct correspondence with distortions close to the lower bound mentioned in Theorem \ref{theorem:circbu}.
\begin{lemma}\label{lemma:mincor}
    Let $\mathrm{P}\subset\mathbb{S}^1$ be a finite union of disjoint closed arcs $\{I_i\}_{i=1}^n$ ordered counterclockwise with $I_i=[a_i,b_i]$ where $a_i,b_i\in[0,2\pi]$ and $\psi(\operatorname{d}_\mathrm{H}(\mathrm{P},\mathbb{S}^1))<\sqrt2$. Then there exists a correspondence $\mathcal{R}_0\subset\mathrm{P}\times\mathbb{S}^1$ such that $\operatorname{dis}(\mathcal{R}_0)=\psi(\operatorname{d}_\mathrm{H}(\mathrm{P},\mathbb{S}^1))$
    \begin{proof}
         Let $\mathrm{Q} =\mathbb{S}^1\setminus\mathrm{P}$ then $\mathrm{Q}=\bigsqcup_{i=1}^{n}J_i$ where $J_i=(b_i,a_{i+1})$ (indices taken modulo $n$, so that $a_{n+1}=a_1$). For any arc $I$, let $|I|$ denote the angle subtended at the center, define, $\operatorname{chd}(I)=2\sin(\frac{|I|}{2})$ as the length of the chord joining the endpoints of its closure. Let $\theta_i$ be the angle subtended by $J_i$.  Define the correspondence $\mathcal{R}_0$ as $\mathcal{R}_0=\{(x,x)|\;\forall\;x\in \mathrm{P}\}\cup\bigcup_{i=1}^n\{(b_i,y)|\;\forall\;y \in J_i\}$, note that the projections on the coordinates $\pi_{\mathbb{S}^1}(\mathcal{R}_0)=\mathbb{S}^1$ and $\pi_{\mathrm{P}}(\mathcal{R}_0)=\mathrm{P}$. As $\mathrm{P}$ is a disjoint union of closed arcs the Hausdorff distance $\operatorname{d}_\mathrm{H}(\mathrm{P},\mathbb{S}^1)=\max_{i=1}^n\{\psi^{-1}\circ \operatorname{chd}(J_i)\}$. From the assumption $\psi(\operatorname{d}_\mathrm{H}(\mathrm{P},\mathbb{S}^1))<\sqrt2$ we also have $\theta_i<\frac{\pi}{2}$. Let $\operatorname{dif}_\mathrm{P}:\mathrm{P}\times\mathrm{P}\to\mathbb{R}$ be defined as $\operatorname{dif}_\mathrm{P}(x,x')=\sup\limits_{\mathcal{R}_{x,x'}}\{|\operatorname{d}(x,x')-\operatorname{d}(y,y')|\}$ where supremum is taken over the set $\mathcal{R}_{x,x'}=\mathcal{R}_0(x)\times\mathcal{R}_0(x')$. Then, $\operatorname{dis}(\mathcal{R}_0)=\sup\limits_{(x,x')\in\mathrm{P}\times\mathrm{P}}\operatorname{dif}_\mathrm{P}(x,x')$. To find the distortion of $\mathcal{R}_0$ we split $\mathrm{P}\times\mathrm{P}$ into the cases.\\
        \noindent  Case 1: $x \in \cup_{i=1}^n I_i\setminus\{b_i\}$ and $x' \in \cup_{i=1}^n I_i\setminus\{b_i\}$. Then as $\mathcal{R}_0(x)=\{y\}$ and $\mathcal{R}_0(x')=\{y'\}$ we have $\operatorname{dif}_\mathrm{P}(x,x')=0$.\\
        \noindent Case 2: $x=x'=b_i$ for some $i\in\{1,2\dots n\}$ then, $\operatorname{dif}_\mathrm{P}(x,x')=\sup\limits_{y,y'\in J_i}\{\operatorname{d}(y,y')\}\leq\psi(\operatorname{d}_\mathrm{H}(\mathrm{P},\mathbb{S}^1))$\\
        \noindent Case 3: $x=b_i$ and $x'=b_j$. Re-label if necessary, such that the counterclockwise arc from $b_1 \in\{b_i,b_j\}$ to $b_2\in\{b_i,b_j\}\setminus b_1$ has arc length at most $\pi$. Now, let the angle between $b_1$ and $b_2$ be $\theta_1+\alpha$ for $\alpha\geq0$, so that $\operatorname{d}(b_1,b_2)=2\sin(\frac{\theta_1+\alpha}{2})$. Then $\alpha\leq\pi$ and $\theta_1+\alpha \leq \pi$ by the choice of $b_1$ and $b_2$. For any $y_1\in\mathcal{R}_0(b_1)$ let the angle  between $b_1$ and $y_1$ be $t$ and for any $y_2\in\mathcal{R}_0(b_2)$ let the angle between $b_2$ and $y_2$ be $s$. Then we would have,
        \[\operatorname{dif}_\mathrm{P}(b_i,b_j)=\sup\limits_{\substack{\mathllap{0\le} t \mathrlap{\le \theta_1} \\ \mathllap{0 \le} s \mathrlap{\le \theta_2}}}\hspace{10pt}\left|2\sin\left(\frac{\theta_1+\alpha}{2}\right)-2\sin\left(\frac{\theta_1+\alpha+s-t}{2}\right)\right|\]
        Let $\delta =s-t\in[-\theta_1,\theta_2]$ and $\phi(\delta)=2\sin(\frac{\theta_1+\alpha}{2})-2\sin(\frac{\theta_1+\alpha+\delta}{2})$ (see Figure \ref{fig:mincor}a). Then $\phi'(\delta) = -\cos(\frac{\theta_1+\alpha+\delta}{2})$ vanishes only at $\delta^*=\pi-\theta_1-\alpha$. To see the maximum (or minimum) value of $\phi$ we check the values at the endpoints of the interval $[-\theta_1,\theta_2]$ and at the point $\delta^*=\pi-\theta_1-\alpha$. Now  note that $\phi(\delta)$ decreases from $\delta =-\theta_1$ to $\delta = \delta^*$ and increases from $\delta=\delta^*$ to $\delta=\theta_2$, so $\sup\lvert\phi\rvert$ is attained at $\delta=-\theta_1$, at $\delta=\delta^*$ or at $\delta=\theta_2$. At $\delta^*$ the function $\phi$ attains its minimum, so there $\lvert\phi(\delta^*)\rvert=2-2\sin(\frac{\theta_1+\alpha}{2})$. Also, if $\delta^*\in[-\theta_1,\theta_2]$ then $\theta_1+\theta_2+\alpha\geq \pi$, rearranging we have $\frac{\pi-\theta_2}{2}\leq\frac{\theta_1+\alpha}{2}$ thus from the choice of $i,j$ and $\alpha$ we get, $\frac{\pi-\theta_2}{2}\leq\frac{\theta_1+\alpha}{2}\leq\frac{\pi}{2}$. As $\sin$ increases in this interval $\sin(\frac{\theta_1+\alpha}{2})\geq\cos(\frac{\theta_2}{2})\geq\cos(\frac{\theta}{2})$ with $\theta=\max\{\theta_i,\theta_j\}$. Now as $\theta_i<\frac{\pi}{2}$ for each $i$, we have $\frac{\theta}{2}<\frac{\pi}{4}$ which implies $2(\sin\frac{\theta}{2}+\cos\frac{\theta}{2})\geq2$. Thus,
        \begin{align*}
           &|\phi(\delta^*)|=2-2\sin(\frac{\theta_1+\alpha}{2})\leq 2-2\cos\frac{\theta}{2}\\
           &\phantom{|\phi(\delta^*)|=2-2\sin(\frac{\theta_1+\alpha}{2})}\leq2\sin\frac{\theta}{2}\leq\psi(\operatorname{d}_\mathrm{H}(\mathrm{P},\mathbb{S}^1)) 
        \end{align*}
                
         At the endpoints, $|\phi(-\theta_1)|=4|\cos(\frac{\theta_1}{4}+\frac{\alpha}{2})||\sin(\frac{\theta_1}{4})|\leq4\cos(\frac{\theta_1}{4})\sin(\frac{\theta_1}{4})=2\sin\frac{\theta_1}{2}\leq2\sin\frac{\theta}{2}$, as $\lvert\cos(\frac{\theta_1}{4}+\frac{\alpha}{2})\rvert\leq\cos(\frac{\theta_1}{4})$ given $\frac{\theta_1}{4}+\frac{\alpha}{2}\leq\frac{\theta_1}{2}+\frac{\alpha}{2}\leq\frac{\pi}{2}$. Similarly, at the other endpoint $\phi(\theta_2)=-4\cos\left(\frac{\theta_1+\alpha}{2}+\frac{\theta_2}{4}\right)\sin\left(\frac{\theta_2}{4}\right)$. Note that $\lvert\cos u\rvert\leq\cos c$ whenever $c\leq u\leq\pi-c$ for $c\in\left[0,\frac{\pi}{2}\right]$. With $c=\frac{\theta_2}{4}$ this applies, since $\frac{\theta_2}{4}\leq\frac{\theta_1+\alpha}{2}+\frac{\theta_2}{4}\leq\frac{\pi}{2}+\frac{\theta_2}{4}\leq\pi-\frac{\theta_2}{4}$. Hence $\lvert\phi(\theta_2)\rvert\leq4\cos\left(\frac{\theta_2}{4}\right)\sin\left(\frac{\theta_2}{4}\right)=2\sin\frac{\theta_2}{2}\leq2\sin\frac{\theta}{2}$. Therefore $\operatorname{dif}_\mathrm{P}(b_i,b_j)\leq2\sin\frac{\theta}{2}\leq\psi(\operatorname{d}_\mathrm{H}(\mathrm{P},\mathbb{S}^1))$.\\
        \noindent Case 4: $x \in \cup_{i=1}^n I_i\setminus\{b_i\}$ and $x'=b_j$ then, $$\operatorname{dif}_{\mathrm{P}}(x,b_j)=\begin{cases}\operatorname{d}(x,a_{j+1}) - \operatorname{d}(x,b_j)& \text { for }x\in[-a_{j+1},b_j]\\
        2-\operatorname{d}(x,b_j) & \text { for } x\in\left[\frac{-(b_j+a_{j+1})}{2},-a_{j+1}\right]\\
        \max\left(2-\operatorname{d}(x,b_j),\operatorname{d}(x,b_j)-\operatorname{d}(x,a_{j+1})\right) & \text { for } x\in\left[-b_j,\frac{-(b_j+a_{j+1})}{2}\right]\\
        \operatorname{d}(x,b_j)-\operatorname{d}(x,a_{j+1})  & \text { for }x\in\left[a_{j+1},-b_j\right]
        \end{cases}$$
        In the first sub-interval let $\alpha\in[0,\pi-\theta_j]$ be the angle subtended by the chord $x\,b_j$ then, $\operatorname{d}(x,a_{j+1}) - \operatorname{d}(x,b_j)=2\sin\left(\frac{\theta_j+\alpha}{2}\right)-2\sin\left(\frac{\alpha}{2}\right)$. Now, $2\sin\left(\frac{\theta_j+\alpha}{2}\right)-2\sin\left(\frac{\alpha}{2}\right)$ is decreasing and has maximum value $2\sin\left(\frac{\theta}{2}\right)\leq\psi(\operatorname{d}_\mathrm{H}(\mathrm{P},\mathbb{S}^1))$. In the second sub-interval $2-\operatorname{d}(x,b_j)\leq2-2\sin\left(\frac{\pi-\theta_j}{2}\right)\leq2\sin\left(\frac{\theta_j}{2}\right)\leq\psi(\operatorname{d}_\mathrm{H}(\mathrm{P},\mathbb{S}^1))$. In the third sub-interval that is when $\alpha\in\left[\pi-\frac{\theta_j}{2},\pi\right]$, we get, 
     \begin{align*}
         &\max\left(2-\operatorname{d}(x,b_j),\operatorname{d}(x,b_j)-\operatorname{d}(x,a_{j+1})\right)\leq \max\left(2-2\sin\left(\frac{\alpha}{2}\right),2\sin\left(\frac{\alpha}{2}\right)-2\sin\left(\frac{2\pi-\theta_j-\alpha)}{2}\right)\right)\\ 
         &\phantom{\max\left(2-\operatorname{d}(x,b_j),\operatorname{d}(x,b_j)-\operatorname{d}(x,a_{j+1})\right)\,}\leq2\sin\left(\frac{\theta_j}{2}\right)\leq\psi(\operatorname{d}_\mathrm{H}(\mathrm{P},\mathbb{S}^1))
     \end{align*}        
     Therefore, $\operatorname{dis}(\mathcal{R}_0)\leq\psi(\operatorname{d}_\mathrm{H}(\mathrm{P},\mathbb{S}^1))$. Conversely, let $i$ be an index with $\operatorname{chd}(J_i)=\max\limits_{j}\operatorname{chd}(J_j)$. By Case 2, $\operatorname{dif}_\mathrm{P}(b_i,b_i)=\sup\limits_{y,y'\in J_i}\{\operatorname{d}(y,y')\}=\operatorname{chd}(J_i)$ and, as $\operatorname{d}_\mathrm{H}(\mathrm{P},\mathbb{S}^1)=\max\limits_{j}\{\psi^{-1}\circ\operatorname{chd}(J_j)\}$, we have $\operatorname{chd}(J_i)=\psi(\operatorname{d}_\mathrm{H}(\mathrm{P},\mathbb{S}^1))$. Hence $\operatorname{dis}(\mathcal{R}_0)=\psi(\operatorname{d}_\mathrm{H}(\mathrm{P},\mathbb{S}^1))$.
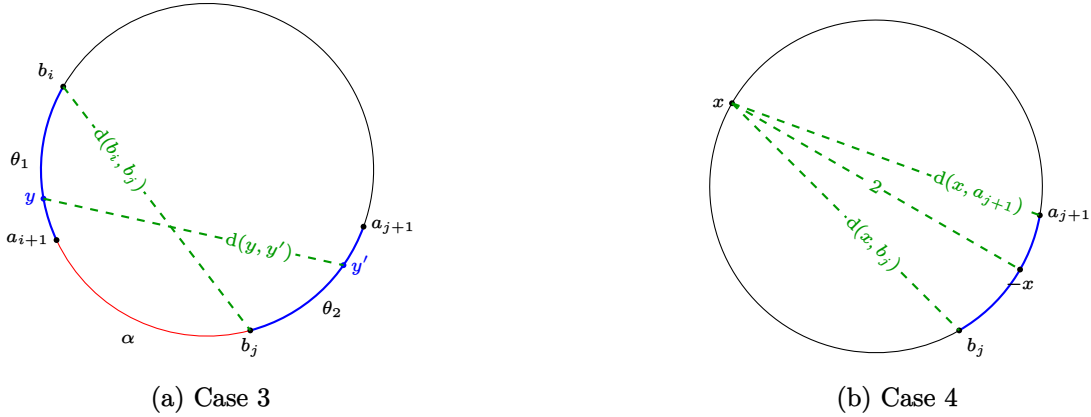
\begin{figure}[htb]
\centering
\begin{subfigure}{0.45\textwidth}
\centering
\begin{tikzpicture}[scale=1.1]
\draw[black] (-20:2) arc(-20:150:2);
\draw[blue,thick] (150:2) arc(150:205:2);
\draw[blue,thick] (285:2) arc(285:340:2);
\draw[red] (205:2) arc(205:285:2);
\fill (150:2) circle(1pt); \draw (150:2) node[node font=\tiny,anchor=south east] (){$b_i$};
\fill (205:2) circle(1pt); \draw (205:2) node[node font=\tiny,anchor=east] (){$a_{i+1}$};
\fill (285:2) circle(1pt); \draw (285:2) node[node font=\tiny,anchor=north] (){$b_j$};
\fill (340:2) circle(1pt); \draw (340:2) node[node font=\tiny,anchor=west] (){$a_{j+1}$};
\fill[blue] (190:2) circle(1pt); \draw[blue] (190:2) node[node font=\tiny,anchor=east] (){$y$};
\fill[blue] (325:2) circle(1pt); \draw[blue] (325:2) node[node font=\tiny,anchor=west] (){$y'$};
\path[draw,green!60!black,dashed,thick] (150:2)--(285:2) node[pos=0.3,sloped,fill=white,inner sep=1pt,node font=\tiny] (){$\operatorname{d}(b_i,b_j)$};
\path[draw,dashed,green!60!black,thick] (190:2)--(325:2) node[pos=0.72,sloped,fill=white,inner sep=1pt,node font=\tiny] (){$\operatorname{d}(y,y')$};
\draw (177:2.25) node[node font=\tiny] (){$\theta_1$};
\draw (312:2.25) node[node font=\tiny] (){$\theta_2$};
\draw (245:2.25) node[node font=\tiny] (){$\alpha$};
\end{tikzpicture}
\caption*{(a) Case 3}
\end{subfigure}\hfill
\begin{subfigure}{0.45\textwidth}
\centering
\begin{tikzpicture}[scale=1.1]
\draw[black] (-10:2) arc(-10:300:2);
\draw[blue,thick] (300:2) arc(300:350:2);
\fill (300:2) circle(1pt); \draw (300:2) node[node font=\tiny,anchor=north west] (){$b_j$};
\fill (350:2) circle(1pt); \draw (350:2) node[node font=\tiny,anchor=west] (){$a_{j+1}$};
\fill[black] (150:2) circle(1pt); \draw[black] (150:2) node[node font=\tiny,anchor=east] (){$x$};
\fill (330:2) circle(1pt); \draw (330:2) node[node font=\tiny,anchor=north] (){$-x$};
\path[draw,green!60!black,dashed,thick] (150:2)--(300:2) node[pos=0.62,sloped,fill=white,inner sep=1pt,node font=\tiny] (){$\operatorname{d}(x,b_j)$};
\path[draw,green!60!black,dashed,thick] (150:2)--(350:2) node[pos=0.8,sloped,fill=white,inner sep=1pt,node font=\tiny] (){$\operatorname{d}(x,a_{j+1})$};
\path[draw,green!60!black,dashed,thick] (150:2)--(330:2) node[pos=0.5,sloped,fill=white,inner sep=1pt,node font=\tiny] (){$2$};
\end{tikzpicture}
\caption*{(b) Case 4}
\end{subfigure}\hfill
\caption{Cases 3 and 4 of Lemma \ref{lemma:mincor}. Blue arcs are the gaps $J_i=(b_i,a_{i+1})$ and $J_j=(b_j,a_{j+1})$ (the complement of $\mathrm{P}$); $\mathcal{R}_0$ sends $b_i$ and $b_j$ to their whole gap arcs. In (a) the reference chord $\operatorname{d}(b_i,b_j)$ subtends the shorter arc $\theta_1+\alpha$, while $y,y'$ range over the two gaps; the red arc from $a_{i+1}$ to $b_j$ is the part of $\mathrm{P}$ separating them. In (b) $x\in\mathrm{P}$ corresponds to itself and $b_j$ to the whole gap $J_j$ (in blue); when the antipode $-x$ lies in $J_j$ the distance $\operatorname{d}(x,y)$ for $y\in J_j$ attains the diameter $2$.}
\label{fig:mincor}
\end{figure}
    \end{proof}
\end{lemma}

\begin{cor}\label{cor:circeq}
    For a subset $\mathrm{A}\subset\mathbb{S}^1$ if $\psi(\operatorname{d}_{\mathrm{H}}(\mathrm{A},\mathbb{S}^1))\leq c_0$, then
    $\operatorname{d}_{\mathrm{GH}}(\mathrm{A},\mathbb{S}^1)=\frac{1}{2}\psi(\operatorname{d}_{\mathrm{H}}(\mathrm{A},\mathbb{S}^1))$ (here $c_0$ is as defined in Theorem \ref{theorem:circbu}).
\begin{proof}
    From Theorem \ref{theorem:circbu} we have $\operatorname{d}_{\mathrm{GH}}(\mathrm{A},\mathbb{S}^1)\geq\frac{1}{2}\psi(\operatorname{d}_{\mathrm{H}}(\mathrm{A},\mathbb{S}^1))$.
    For any set $\mathrm{A}$ and any $0<\epsilon<\frac{1}{2}\operatorname{d}_\mathrm{H}(\mathrm{A},\mathbb{S}^1)$, let $\overline{\mathrm{A}^\epsilon_{\mathbb{S}^1}}=\overline{\mathrm{A}^\epsilon\cap\mathbb{S}^1}$. Note that $\mathrm{A}^\epsilon\cap\mathbb{S}^1$ is open in $\mathbb{S}^1$ and is a union of finite open arcs with diameters (in the Euclidean metric) greater than $\psi(\epsilon)$. Thus, $\overline{\mathrm{A}^\epsilon_{\mathbb{S}^1}}$ is a finite union of disjoint closed arcs. Also, as $\mathrm{A}\subset\overline{\mathrm{A}^\epsilon_{\mathbb{S}^1}}\subset\mathbb{S}^1$, we have $\operatorname{d}_\mathrm{H}(\overline{\mathrm{A}^\epsilon_{\mathbb{S}^1}},\mathbb{S}^1)\leq \operatorname{d}_\mathrm{H}(\mathrm{A},\mathbb{S}^1)$. Now, as $\mathrm{A}\subset\overline{\mathrm{A}^\epsilon_{\mathbb{S}^1}}\subset \overline{\mathrm{A}^\epsilon}$, $\operatorname{d}_{\mathrm{GH}}(\mathrm{A},\overline{\mathrm{A}^\epsilon_{\mathbb{S}^1}})\leq\operatorname{d}_\mathrm{H}(\mathrm{A},\overline{\mathrm{A}^\epsilon_{\mathbb{S}^1}})\leq\operatorname{d}_\mathrm{H}(\mathrm{A},\overline{\mathrm{A}^\epsilon})\leq\epsilon$. Now from Lemma \ref{lemma:mincor} we have $\operatorname{d}_{\mathrm{GH}}(\overline{\mathrm{A}^\epsilon_{\mathbb{S}^1}},\mathbb{S}^1)\leq\frac{1}{2}\psi(\operatorname{d}_{\mathrm{H}}(\overline{\mathrm{A}^\epsilon_{\mathbb{S}^1}},\mathbb{S}^1))\leq\frac{1}{2}\psi(\operatorname{d}_{\mathrm{H}}(\mathrm{A},\mathbb{S}^1))$. Now $$(\operatorname{d}_{\mathrm{GH}}(\mathrm{A},\mathbb{S}^1))\leq\operatorname{d}_{\mathrm{GH}}(\overline{\mathrm{A}^\epsilon_{\mathbb{S}^1}},\mathbb{S}^1)+\operatorname{d}_{\mathrm{GH}}(\mathrm{A},\overline{\mathrm{A}^\epsilon_{\mathbb{S}^1}})\leq\frac{1}{2}\psi(\operatorname{d}_{\mathrm{H}}(\overline{\mathrm{A}^\epsilon_{\mathbb{S}^1}},\mathbb{S}^1))+\epsilon\leq\frac{1}{2}\psi(\operatorname{d}_{\mathrm{H}}(\mathrm{A},\mathbb{S}^1))+\epsilon$$
    As $\epsilon>0$ can be arbitrarily close to zero we have $\operatorname{d}_{\mathrm{GH}}(\mathrm{A},\mathbb{S}^1)=\frac{1}{2}\psi(\operatorname{d}_{\mathrm{H}}(\mathrm{A},\mathbb{S}^1))$.
\end{proof}
\end{cor}

This contrasts with the circle equipped with the geodesic (length-space) metric, for which the Gromov-Hausdorff distance of a subset from the circle coincides with its Hausdorff distance, $\operatorname{d}_{\mathrm{GH}}(\mathrm{A},\mathbb{S}^1)=\operatorname{d}_{\mathrm{H}}(\mathrm{A},\mathbb{S}^1)$ \cite{adams2025lowerboundinggromovhausdorffdistance}. For the Euclidean metric Corollary \ref{cor:circeq} instead gives $\operatorname{d}_{\mathrm{GH}}(\mathrm{A},\mathbb{S}^1)=\frac{1}{2}\psi(\operatorname{d}_{\mathrm{H}}(\mathrm{A},\mathbb{S}^1))$, which is strictly smaller than $\operatorname{d}_{\mathrm{H}}(\mathrm{A},\mathbb{S}^1)$ whenever $\operatorname{d}_{\mathrm{H}}(\mathrm{A},\mathbb{S}^1)>0$. Figure \ref{fig:circbound} plots these bounds.

\begin{figure}[htb]
\centering
\begin{tikzpicture}
\begin{axis}[
    xlabel={$\operatorname{d}_{\mathrm{H}}(\mathrm{A},\mathbb{S}^1)$},
    ylabel={$\operatorname{d}_{\mathrm{GH}}(\mathrm{A},\mathbb{S}^1)$},
    axis lines=left,
    axis line style={-{Stealth[length=2.5mm]}},
    tick style={draw=none},
    xmin=0, xmax=1.4142, ymin=0, ymax=1.05,
    samples=200,
    width=0.72\textwidth, height=0.5\textwidth,
    legend pos=north west,
    legend cell align=left,
]

\addplot[blue, thick, domain=0:0.6958] {0.5*x*sqrt(4-x^2)};
\addlegendentry{$\tfrac{1}{2}\psi(\operatorname{d}_{\mathrm{H}})$}

\addplot[red, thick, domain=0.6958:1.4142] {0.6525};
\addlegendentry{ $\tfrac{c_0}{2}$}

\addplot[green!60!black,dashed, thick,domain=0:1]{x};
\addlegendentry{the line y=x}
\addplot[gray, thick, dashed, domain=0.6958:1.4142] {0.5*x*sqrt(4-x^2)};
\addlegendentry{$\tfrac{1}{2}\psi$ continued}

\addplot[black, dotted, forget plot] coordinates {(0.6958,0) (0.6958,0.6525)};
\end{axis}
\end{tikzpicture}
\caption{The circle bound of Theorem \ref{theorem:circbu} and Corollary \ref{cor:circeq}. For $\operatorname{d}_{\mathrm{H}}\left(\mathrm{A},\mathbb{S}^1\right)\leq\psi^{-1}(c_0)\approx0.696$ (equivalently $\psi(\operatorname{d}_{\mathrm{H}})\leq c_0$) the Gromov-Hausdorff distance is determined exactly, $\operatorname{d}_{\mathrm{GH}}\left(\mathrm{A},\mathbb{S}^1\right)=\tfrac{1}{2}\psi(\operatorname{d}_{\mathrm{H}}\left(\mathrm{A},\mathbb{S}^1\right))=\tfrac{1}{2}\operatorname{d}_{\mathrm{H}}\sqrt{4-\operatorname{d}_{\mathrm{H}}^2}$; beyond that scale Theorem \ref{theorem:circbu} gives only the lower bound $c_0/2$ (the dotted line marks the threshold).}
\label{fig:circbound}
\end{figure}
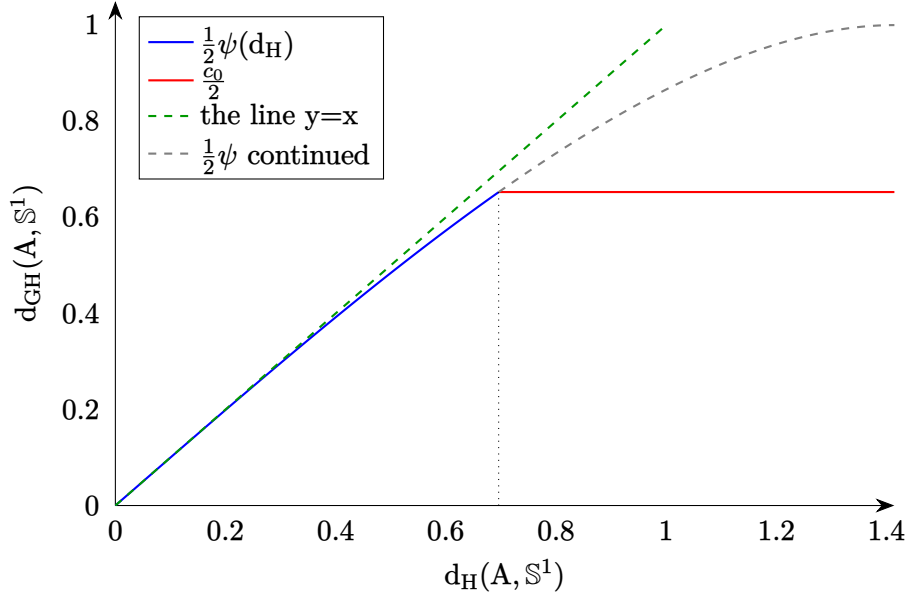
\section{Gromov-Hausdorff distance between two subsets using Vietoris-Rips complexes}\label{sec:GH2rips}

In \cite{adams2025hausdorffvsgromovhausdorffdistances} the Gromov-Hausdorff distance between two subsets of a Riemannian
manifold is estimated by using Vietoris-Rips complexes to build an interleaving
from a correspondence between the subsets. Because the Vietoris-Rips complex
interleaves at a tighter scale than the corresponding \v{C}ech complex, this sharpens
the resulting estimate. There the homotopy equivalence between the Vietoris-Rips
complex and the underlying manifold is supplied by Hausmann's theorem for
Riemannian submanifolds, which is used to construct an isomorphism on homology. The argument then follows the method of Theorem \ref{thm:firstbound}.

In order to do this in the case of Euclidean submanifolds we establish a homotopy equivalence between the Vietoris-Rips complex and the underlying manifold by means of metric thickenings\cite{doi:10.1137/17M1148025,adams2024persistent,adams2019metric}. We first show that, for a $C^{1,1}$ submanifold of positive reach,
the Vietoris-Rips complex is homotopy equivalent to the manifold at all
sufficiently small scales (Lemma \ref{lemma:Hausmanneq}). We then establish an
interleaving between the Vietoris-Rips and intrinsic \v{C}ech complexes of a subset
(Lemma \ref{lemma:interleaving}), in which the scale distortion is controlled by
the ratio of the scale to the reach. Combining these with the maps of
\cite{adams2025hausdorffvsgromovhausdorffdistances} and the reach-based distance
estimates of Lemma \ref{lemma:estimate}, we obtain a sharper lower bound on the
Gromov-Hausdorff distance between two subsets (Theorem~\ref{theorem:main2set}).

\subsection{Vietoris-Rips thickenings and complexes}\label{sec:hausmannthickenings}

\begin{Def}
    For an open cover $\mathcal{U}$ of a metric space $(\mathrm{X},\operatorname{d}_\mathrm{X})$ the \textbf{Vietoris complex} of the cover $\mathcal{V} \left(\mathcal{U} \right)$ is the simplicial complex where a finite subset $\sigma\subset\mathrm{X}$ is a simplex iff $\sigma\subset \mathrm{U}$ for some set $\mathrm{U}\in \mathcal{U}$. For the cover $\mathcal{U}_r=\{\mathrm{V}\mid\mathrm{V}\subset_{\text{open}} \mathrm{X}\text{\;and\;}\operatorname{diam} \left(\mathrm{V} \right)<r\}$ the Vietoris complex is the Vietoris-Rips complex at scale $r$, that is, $\mathcal{V} \left(\mathcal{U}_r \right)=\mathrm{VR} \left(\mathrm{X};r \right)$
\end{Def}

Its metric counterpart replaces simplices by finitely supported probability measures, with the scale recorded by the Wasserstein distance.

\begin{Def}
    For a metric space $(\mathrm{X},\operatorname{d}_\mathrm{X})$, let $\mathcal{P}^{\text {fin }}(\mathrm{X})$ denote the set of all probability measures on $\mathrm{X}$ with finite support. For $\mu, \nu \in \mathcal{P}^{\text {fin }}(\mathrm{X})$, a coupling between $\mu$ and $\nu$ is a probability measure $\gamma$ on $\mathrm{X} \times \mathrm{X}$ such that $\gamma(\mathrm{A} \times \mathrm{X})=\mu(\mathrm{A})$ and $\gamma(\mathrm{X} \times \mathrm{A})=\nu(\mathrm{A})$ for all Borel subsets $\mathrm{A} \in \mathcal{B}(\mathrm{X})$. If $\mathcal{C}$ denotes the set of couplings between $\mu, \nu \in \mathcal{P}^{\text {fin }}(\mathrm{X})$, the 1 -Wasserstein distance between $\mu$ and $\nu$ is
$$
\operatorname{d}_{W}(\mu, \nu)=\inf _{\gamma \in \mathcal{C}} \int_{\mathrm{X} \times \mathrm{X}} \operatorname{d}_{\mathrm{X}}(x, y) \gamma(d x \times d y)
$$
If $\mathcal{U}$ is a cover of $\mathrm{X}$, the \textbf{Vietoris metric thickening} $\mathcal{V}^{m}(\mathcal{U})$ is the metric space of all $\mu \in \mathcal{P}^{\text {fin }}(\mathrm{X})$ such that $\operatorname{supp}(\mu)$ is contained in an element of $\mathcal{U}$, equipped with the 1 -Wasserstein metric. As in the previous definition, for the cover $\mathcal{U}_r=\{\mathrm{V}\mid\mathrm{V}\subset_{\text{open}} \mathrm{X}\text{\; and\;}\operatorname{diam} \left(\mathrm{V} \right)<r\}$ the Vietoris metric thickening is the Vietoris-Rips metric thickening at scale $r$, that is, $\mathcal{V}^m \left(\mathcal{U}_r \right)=\mathrm{VR}^m \left(\mathrm{X};r \right)$
\end{Def}

A weak homotopy equivalence can be constructed between the Vietoris complex and the metric thickening of a cover under some conditions.

\begin{thm}\label{theorem:gillespie}(Theorem 1 \cite{gillespie2024vietoris})
For any uniformly bounded open cover $\mathcal{U}$ of a metric space $\mathrm{X}$, the
natural bijection $\operatorname{id} : |\mathcal{V} \left(\mathcal{U} \right)| \to \mathcal{V}^m \left(\mathcal{U} \right)$ is a weak homotopy equivalence. 
\end{thm}

On spaces of CW type such an equivalence is in fact a homotopy equivalence.

\begin{thm}[Whitehead's theorem (Theorem 4.5 \cite{MR1867354})]
    If topological spaces $\mathrm{X},\mathrm{Y}$ have the homotopy type of a CW-complex, then any weak homotopy equivalence $f:\mathrm{X}\to \mathrm{Y}$ is a homotopy equivalence. 
\end{thm}
\begin{rem}
    For $\lambda_n=\sqrt{\frac{n}{2(n+1)}}$ the length of a side of the $n$-simplex with unit circumradius is $\frac{1}{\lambda_n}$.
\end{rem}
For the sphere this equivalence is known explicitly.

\begin{thm}\label{theorem:snthe}(Proposition 5.3 \cite{adams2024persistent})
  There exists a map $\rho: \mathrm{VR}^{m} \left(\mathbb{S}^{n} ; r \right) \xrightarrow[\hspace{0.5cm}]{} \mathbb{S}^{n}$ such that for $0<r<\frac{1}{\lambda_{n+1}}, \rho: \mathrm{VR}^{m} \left(\mathbb{S}^{n} ; r \right) \xrightarrow[\hspace{0.5cm}]{} \mathbb{S}^{n}$ is a homotopy equivalence.
\end{thm}

For subsets of positive reach the metric thickening recovers the homotopy type of the manifold.

\begin{thm}\label{theorem:mirth}(Theorem 4.6 \cite{adams2019metric})
    For a subset $\mathcal{M}\subset\mathbb{R}^n$ with positive reach $\operatorname{r}_1 \left(\mathcal{M} \right)>0$ the Vietoris-Rips thickening at scale $0<r<\operatorname{r}_1 \left(\mathcal{M} \right)$, $\mathrm{VR}^m \left(\mathcal{M};r \right)$ is homotopy equivalent to $\mathcal{M}$. 
\end{thm}
\begin{rem}
    From the proof of Theorem 4.6 \cite{adams2019metric} we can also see that for $0<r<r'<\operatorname{r}_1 \left(\mathcal{M} \right)$ the inclusion $\iota_m:\mathrm{VR}^m \left(\mathcal{M};r \right)\xhookrightarrow[\hspace{0.5cm}]{}\mathrm{VR}^m \left(\mathcal{M};r' \right)$ is a homotopy equivalence.
\end{rem}
We use the results above to verify a Hausmann-type theorem on Vietoris-Rips complexes of $C^{1,1}$ submanifolds with positive reach.
\begin{lemma}\label{lemma:Hausmanneq}
    Let $\mathcal{M}$ be a compact $C^{1,1}$, $d$-manifold  with positive reach $\operatorname{r}_1 \left(\mathcal{M} \right)>0$, embedded in $\mathbb{R}^n$. For $0<r<r'<\operatorname{r}_1 \left(\mathcal{M} \right)$, $\mathrm{VR} \left(\mathcal{M};r \right)\simeq\mathrm{VR} \left(\mathcal{M};r' \right)\simeq\mathcal{M}$ and the inclusion $\iota:\mathrm{VR} \left(\mathcal{M};r \right)\xhookrightarrow[\hspace{0.5cm}]{}\mathrm{VR} \left(\mathcal{M};r' \right)$ is a homotopy equivalence.
    \begin{proof}
        Note that for the cover $\mathcal{U}_r=\{\mathrm{V}\mid\mathrm{V}\subset_{\text{open}} \mathcal{M}\text{\; and\;}\operatorname{diam} \left(\mathrm{V} \right)<r\}$, $\mathcal{V}^m(\mathcal{U}_r)=\mathrm{VR}^m \left(\mathcal{M};r \right)$ and $\mathcal{V}(\mathcal{U}_r)=\mathrm{VR} \left(\mathcal{M};r \right)$. From Theorem \ref{theorem:gillespie} we get, $\operatorname{id}:\mathrm{VR} \left(\mathcal{M};r \right)\rightarrow\mathrm{VR}^m \left(\mathcal{M};r \right)$ is a weak homotopy equivalence. From Theorem ~\ref{theorem:mirth}, we know that $\mathrm{VR}^m \left(\mathcal{M};r \right)\simeq\mathcal{M}$. As $\mathcal{M}$ is a compact $C^{1,1}$ submanifold, it admits a triangulation. Therefore, it has the homotopy type of a CW-complex. Thus, the space $\mathrm{VR}^m \left(\mathcal{M};r \right)$ has the homotopy type of a CW-complex. As $\mathrm{VR} \left(\mathcal{M};r \right)$ is the geometric realization of a simplicial complex, from Whitehead's theorem, $\operatorname{id}:\mathrm{VR} \left(\mathcal{M};r \right)\to\mathrm{VR}^m \left(\mathcal{M};r \right)$ is a homotopy equivalence and thus, $\mathrm{VR} \left(\mathcal{M};r \right)\simeq\mathcal{M}$. Note that the following diagram commutes,
        \begin{center}
     \begin{tikzcd}      
    \mathrm{VR}^m \left(\mathcal{M};r \right)\arrow[r,"\iota_m",hook]&\mathrm{VR}^m \left(\mathcal{M};r' \right)\\
    \mathrm{VR} \left(\mathcal{M};r \right)\arrow[r,"\iota",hook]\arrow[u,"\operatorname{id}_1"]&\mathrm{VR} \left(\mathcal{M};r' \right)\arrow[u,"\operatorname{id}_2"]
        \end{tikzcd}
\end{center}
where $\operatorname{id}_1$ and $\operatorname{id}_2$ are the natural identity maps between the underlying sets of the Vietoris-Rips complex and the metric thickening. As $\iota_m$, $\operatorname{id}_1$ and $\operatorname{id}_2$ are homotopy equivalences let $g_m:\mathrm{VR}^m \left(\mathcal{M};r' \right)\to\mathrm{VR}^m \left(\mathcal{M};r \right)$, $g_1:\mathrm{VR}^m \left(\mathcal{M};r \right)\to\mathrm{VR} \left(\mathcal{M};r \right)$ and $g_2:\mathrm{VR}^m \left(\mathcal{M};r' \right)\to\mathrm{VR} \left(\mathcal{M};r' \right)$ be their homotopy inverses. Denote by $\mathrm{I}_\mathrm{X}$ the identity map for any space $\mathrm{X}$. Then,
$$
     \left(g_1\circ g_m\circ\operatorname{id}_2 \right)\circ\iota=g_1\circ g_m\circ \left(\operatorname{id}_2\circ\iota \right)= g_1\circ g_m\circ\iota_m\circ\operatorname{id}_1$$
as, $g_m\circ\iota_m\simeq \mathrm{I}_{\mathrm{VR}^m \left(\mathcal{M};r \right)}$, we get  $$ \left(g_1\circ g_m\circ\operatorname{id}_2 \right)\circ\iota=g_1\circ g_m\circ\iota_m\circ\operatorname{id}_1\simeq g_1\circ\operatorname{id}_1\simeq\mathrm{I}_{\mathrm{VR} \left(\mathcal{M};r \right)}$$ Note that, $$\iota\circ g_1=\mathrm{I}_{\mathrm{VR} \left(\mathcal{M};r' \right)}\circ\iota\circ g_1\simeq g_2\circ\operatorname{id}_2\circ\iota\circ g_1=g_2\circ\iota_m\circ\operatorname{id}_1\circ g_1\simeq g_2\circ\iota_m$$ Hence, $$\iota\circ \left(g_1\circ g_m\circ\operatorname{id}_2 \right)\simeq g_2\circ\iota_m\circ g_m\circ\operatorname{id}_2$$
As, $\iota_m\circ g_m\simeq \mathrm{I}_{\mathrm{VR}^m \left(\mathcal{M};r' \right)}$ and $g_2\circ\operatorname{id}_2\simeq \mathrm{I}_{\mathrm{VR} \left(\mathcal{M};r' \right)}$ we get $\iota\circ \left(g_1\circ g_m\circ\operatorname{id}_2 \right)\simeq\mathrm{I}_{\mathrm{VR} \left(\mathcal{M};r' \right)}$. Thus, the inclusion $\iota:\mathrm{VR} \left(\mathcal{M};r \right)\xhookrightarrow[\hspace{0.5cm}]{}\mathrm{VR} \left(\mathcal{M};r' \right)$ is a homotopy equivalence.

    \end{proof}
\end{lemma}

\subsection{Bounds for two subsets}

We now combine the previous results into a lower bound for the Gromov-Hausdorff distance between two subsets of $\mathcal{M}$, using the reach-based projection estimate of Lemma \ref{lemma:estimate}.

The next two lemmas, from \cite{adams2025hausdorffvsgromovhausdorffdistances}, produce simplicial maps between Vietoris-Rips complexes from a Gromov-Hausdorff and a Hausdorff bound respectively.

    \begin{lemma}(Lemma 4.1 \cite{adams2025hausdorffvsgromovhausdorffdistances})\label{lemma:lemvrc}
Let $\mathrm{A}$ and $\mathrm{B}$ be metric spaces and let $r>2  \operatorname{d}_{\mathrm{GH}} \left(\mathrm{A}, \mathrm{B} \right)$. Then for any $\epsilon>0$ and $\eta\geq0$ , there exist simplicial maps:
$$
\mathrm{VR} \left(\mathrm{B} ; \epsilon \right) \xrightarrow[\hspace{0.75cm}]{\bar{h}} \mathrm{VR} \left(\mathrm{A} ; r+\epsilon \right) \xhookrightarrow[\hspace{0.75cm}]{\iota_\eta}\mathrm{VR} \left(\mathrm{A} ; \eta+r+\epsilon \right)\xrightarrow[\hspace{0.75cm}]{\bar{g}} \mathrm{VR} \left(\mathrm{B}  ; 2 r+\eta+\epsilon \right)
$$
such that the composition $\bar{g} \circ \bar{h}$ is contiguous to the inclusion $\mathrm{VR} \left(\mathrm{B} ; \epsilon \right) \hookrightarrow \mathrm{VR} \left(\mathrm{B} ; 2r+\eta+\epsilon \right)$.
\end{lemma}
\begin{lemma}(Lemma 4.1 \cite{adams2025hausdorffvsgromovhausdorffdistances})\label{lemma:lemvrh}
Let $\mathrm{Z}$ be a metric space and let $\mathrm{B} \subset \mathrm{Z}$ satisfy $\delta>2 \operatorname{d}_{\mathrm{H}} \left(\mathrm{B}, \mathrm{Z} \right)$. Then for any $\epsilon>0$ and $\nu\geq0$, there exist simplicial maps:
$$
\mathrm{VR} \left(\mathrm{Z} ; \epsilon \right) \xrightarrow[\hspace{0.75cm}]{\bar{f}} \mathrm{VR} \left(\mathrm{B} ; \delta+\epsilon \right) \xrightarrow[\hspace{0.75cm}]{\iota_\nu} \mathrm{VR} \left(\mathrm{B} ; \delta+\nu+\epsilon \right) \xhookrightarrow[\hspace{0.75cm}]{\iota_2} \mathrm{VR} \left(\mathrm{Z} ; \delta+\nu+\epsilon \right)
$$
such that the composition $\iota_2\circ\iota_\nu \circ \bar{f}$ is contiguous to the inclusion $\iota_\mathrm{Z}:\operatorname{VR} \left(\mathrm{Z} ; \epsilon \right) \hookrightarrow \operatorname{VR} \left(\mathrm{Z} ; \delta+\nu+\epsilon \right)$. Here $\iota_2$ is the simplicial map induced by the vertex map given by the inclusion $\mathrm{B}\xhookrightarrow[\hspace{0.5cm}]{}\mathrm{Z}$.
\end{lemma}

Using the estimate of Lemma \ref{lemma:estimate}, we interleave the Vietoris-Rips and intrinsic \v{C}ech complexes of a subset at a scale controlled by the ratio of the scale to the reach.

\begin{lemma}\label{lemma:interleaving}
    Let $\mathcal{M}\subset\mathbb{R}^n$ be a compact submanifold with positive reach $\operatorname{r}_1(\mathcal{M})>0$ and $\mathrm{A}\subset\mathcal{M}$ be a subset. For $0<r\leq\alpha\operatorname{r}_1(\mathcal{M})$, where $\alpha\in(0,1]$ there exists a simplicial inclusion 
    $\iota_\alpha:\mathrm{VR}(\mathrm{A};r)\xhookrightarrow[\hspace{0.5cm}]{}\check{\mathrm{C}}_\mathcal{M}\left(\mathrm{A};\sqrt{\frac{2}{1+\sqrt{1-\alpha^2}}}\lambda_n r\right)$.

\begin{proof}
    Take a simplex $\sigma=[z_1,\dots,z_k]\in\mathrm{VR}(\mathrm{A};r)$ then $\operatorname{d}(z_i,z_j)<r$ for all $1\leq i,j\leq k$. By Jung's theorem, there exists a point $z\in \mathbb{R}^n$ such that $\operatorname{d}(z,z_i)<\lambda_nr$ for each $1\leq i\leq k$. As $\lambda_nr<\operatorname{r}_1$ from Lemma \ref{lemma:estimate} we get,$$\operatorname{d}(\pi_\mathcal{M}(z),z_i)<\sqrt{\frac{2\operatorname{r}_1(\mathcal{M})}{\operatorname{r}_1(\mathcal{M})+\sqrt{\operatorname{r}_1(\mathcal{M})^2-\operatorname{d}(z,z_i)^2}}}\operatorname{d}(z,z_i)$$
    As $\sqrt{\frac{2}{1+\sqrt{1-x^2}}}$ increases with $x$ and $\operatorname{d}(z,z_i)<\lambda_nr\leq\lambda_n\alpha \operatorname{r}_1<\alpha \operatorname{r}_1$, we  get $$\operatorname{d}(\pi_\mathcal{M}(z),z_i)<\sqrt{\frac{2\operatorname{r}_1}{\operatorname{r}_1+\sqrt{\operatorname{r}_1^2-\alpha^2 \operatorname{r}_1^2}}}\operatorname{d}(z,z_i)\leq \sqrt{\frac{2}{1+\sqrt{1-\alpha^2}}}\lambda_nr$$ Thus, the simplex $\sigma$ lies inside $\mathbb{B}(\pi_\mathcal{M}(z),\sqrt{\frac{2}{1+\sqrt{1-\alpha^2}}}\lambda_nr)$ therefore, $\sigma \in \check{\mathrm{C}}_\mathcal{M}\left(\mathrm{A};\sqrt{\frac{2}{1+\sqrt{1-\alpha^2}}}\lambda_n r\right)$
\end{proof}
\begin{rem}\label{remark:inmain}
    Note that as we relax the condition on $r$ in the above Lemma, that is, as $\alpha \to 1$  we get $\sqrt{\frac{2}{1+\sqrt{1-\alpha^2}}}\lambda_nr\to \sqrt2\lambda_nr$. $\lambda_n$ increases with $n$ and $\lambda_n<\frac{1}{\sqrt2}$ so  $\sqrt{\frac{2}{1+\sqrt{1-\alpha^2}}}\lambda_nr<r$. Also, as $1+\sqrt{1-\alpha^2}\leq2$ we have $\sqrt{\frac{2}{1+\sqrt{1-\alpha^2}}}\geq1$ and, $\lambda_n$ being increasing in $n$, $\lambda_n\geq\lambda_1=\frac12$ for every $n\geq1$. Hence, $\frac{1}{2}\leq\sqrt{\frac{2}{1+\sqrt{1-\alpha^2}}}\lambda_n<1$.
\end{rem}
\end{lemma}

Combining the interleaving with the maps above and the Hausmann-type equivalence of Lemma \ref{lemma:Hausmanneq}, we obtain the main bound.

\begin{thm}\label{theorem:main2set}
    Let $\mathcal{M}\subset \mathbb{R}^n$ be a connected, closed $d$-dimensional $C^{1,1}$ Euclidean submanifold with reach $\operatorname{r}_1=\operatorname{r}_1 \left(\mathcal{M} \right)>0$. For  $\alpha$ such that $\alpha\in(0,1]$, let $c_\alpha^n = \sqrt{\frac{2}{1+\sqrt{1-\alpha^2}}}\lambda_n$ and $c=\operatorname{min}\left(\alpha,\frac{1}{1+2c_\alpha^n}\right)$. Then for any subsets $\mathrm{A}, \mathrm{B} \subset \mathcal{M}$, 
    $$
\operatorname{d}_{\mathrm{GH}} \left(\mathrm{A},\mathrm{B} \right) \geq \min \left\{\frac{1}{ 2c_\alpha^n} \operatorname{d}_{\mathrm{H}} \left(\mathrm{A},\mathcal{M} \right)- \operatorname{d}_{\mathrm{H}} \left(\mathrm{B},\mathcal{M} \right),\, \frac{c}{2}\operatorname{r}_1(\mathcal{M})-\operatorname{d}_{\mathrm{H}} \left(\mathrm{B},\mathcal{M}\right)\right\}
$$
 Alternatively, for subsets $\mathrm{A}, \mathrm{B} \subset \mathcal{M}$ such that $\operatorname{d}_\mathrm{H}\left(\mathrm{A},\mathcal{M}\right)<c\,c_\alpha^n\operatorname{r}_1(\mathcal{M})$ for some $0<\alpha<1$ we have $\operatorname{d}_{\mathrm{GH}} \left(\mathrm{A},\mathrm{B} \right) \geq \frac{1}{2c_\alpha^n} \operatorname{d}_{\mathrm{H}} \left(\mathrm{A},\mathcal{M} \right)- \operatorname{d}_{\mathrm{H}} \left(\mathrm{B},\mathcal{M} \right)$.

\begin{proof}
 Assume that  $\operatorname{d}_{\mathrm{GH}} \left(\mathrm{A}, \mathrm{B} \right)<\frac{c}{2}\operatorname{r}_1- \operatorname{d}_{\mathrm{H}} \left(\mathrm{B},\mathcal{M} \right)$. For $\delta>0$ such that $2\operatorname{d}_{\mathrm{H}} \left(\mathrm{B},\mathcal{M} \right)<\delta<c\operatorname{r}_1-2\operatorname{d}_{\mathrm{GH}} \left(\mathrm{A}, \mathrm{B} \right)$ pick $r$ such that $2\, \operatorname{d}_{\mathrm{GH}} \left(\mathrm{A}, \mathrm{B} \right)<r<c\operatorname{r}_1-\delta$. Now, take $\epsilon>0$ so that
$r+\delta+\epsilon<c\operatorname{r}_1$.  Note that from Remark \ref{remark:inmain} $2c_\alpha^n-1\geq0$ for all $\alpha\in(0,1]$ and $n\geq1$. Let $\nu=(2c_\alpha^n+1) r+(2c_\alpha^n-1)(\delta+\epsilon)>0$. Since $\delta>2\operatorname{d}_\mathrm{H}(\mathrm{B},\mathcal{M})$, from Lemma \ref{lemma:lemvrh} we have the sequence of maps,
$$
\mathrm{VR} \left(\mathcal{M} ; \epsilon \right) \xrightarrow[\hspace{0.75cm}]{\bar{f}} \mathrm{VR} \left(\mathrm{B} ; \delta+\epsilon \right) \xrightarrow[\hspace{0.75cm}]{\iota_\nu} \mathrm{VR} \left(\mathrm{B} ; r+2c_\alpha^n(r+\delta+\epsilon) \right) \xhookrightarrow[\hspace{0.75cm}]{\iota_{2}} \mathrm{VR} \left(\mathcal{M} ; r+2c_\alpha^n(r+\delta+\epsilon)\right)
$$ 
such that $\iota_2\circ \iota_\nu \circ \bar{f}\simeq\iota_\mathcal{M}$ where $\iota_\mathcal{M}$ is the inclusion $\mathrm{VR} \left(\mathcal{M} ; \epsilon \right) \xhookrightarrow[\hspace{0.5cm}]{\iota_\mathcal{M}}\mathrm{VR} \left(\mathcal{M} ;r+2c_\alpha^n \left(r+\delta+\epsilon \right) \right)$. As $r>2\operatorname{d}_{\mathrm{GH}}\left(\mathrm{A},\mathrm{B}\right)$, take $\eta=(2c_\alpha^n-1) \left(r+\delta+\epsilon \right)\geq 0$, from Lemma \ref{lemma:lemvrc} we also have the simplicial maps $\bar{h}$ and $\bar{g}$,
$$
\mathrm{VR} \left(\mathrm{B} ;\delta + \epsilon \right) \xrightarrow[\hspace{0.6cm}]{\bar{h}} \mathrm{VR} \left(\mathrm{A} ; r+\delta + \epsilon \right) \xhookrightarrow[\hspace{0.6cm}]{\iota_\eta}\mathrm{VR} \left(\mathrm{A} ; 2c_\alpha^n \left(r+\delta+\epsilon \right)\right)\xrightarrow[\hspace{0.6cm}]{\bar{g}} \mathrm{VR} \left(\mathrm{B}  ; r+ 2c_\alpha^n \left(r+\delta+\epsilon \right) \right)
$$
such that $\bar{g}\circ\iota_\eta\circ\bar{h}\simeq\iota_\nu$ as $\iota_\nu$ is the inclusion $\mathrm{VR} \left(\mathrm{B} ; \delta+\epsilon \right) \xrightarrow[\hspace{0.75cm}]{\iota_\nu} \mathrm{VR} \left(\mathrm{B} ; r+2c_\alpha^n(r+\delta+\epsilon)\right)$. As $r+\delta+\epsilon<c\operatorname{r}_1\leq\alpha \operatorname{r}_1$ from Lemma \ref{lemma:interleaving} we also have a simplicial inclusion $\iota_\alpha:\mathrm{VR} \left(\mathrm{A};r+\delta+\epsilon\right) \xhookrightarrow[\hspace{0.5cm}]{}\check{\mathrm{C}}_\mathcal{M} \left(\mathrm{A} ;c_\alpha^n \left(r+\delta+\epsilon \right) \right)$. Thus, combining the sequences of simplicial maps above we have the following diagram,

\begin{center}
     \begin{tikzcd}[column sep=0.8em]
         & \check{\mathrm{C}}_\mathcal{M} \left(\mathrm{A} ;c_\alpha^n \left(r+\delta+\epsilon \right) \right)\arrow[rd,"\iota_1"]&\\
     \mathrm{VR} \left(\mathrm{A};r+\delta+\epsilon \right)\arrow[rr,dashed,hook,"\iota_\eta"]\arrow[ru,"\iota_\alpha"]&& \mathrm{VR} \left(\mathrm{A};2c_\alpha^n\left(r+\delta+\epsilon \right) \right)\arrow[d,"\bar{g}"]\\
      \mathrm{VR} \left(\mathrm{B};\delta+\epsilon \right)\arrow[rr,dashed,hook,"\iota_\nu"]\arrow[u,"\bar{h}"]&&\mathrm{VR} \left(\mathrm{B};r+2c_\alpha^n \left(r+\delta+\epsilon \right) \right)\arrow[d,"\iota_2"]\\
      \mathrm{VR} \left(\mathcal{M} ; \epsilon \right)\arrow[rr,dashed,hook,"\iota_\mathcal{M}"] \arrow[u,"\bar{f}"]&&\mathrm{VR} \left(\mathcal{M} ;r+2c_\alpha^n \left(r+\delta+\epsilon \right) \right)\\
        \end{tikzcd}
\end{center}
where $\iota_1$ is the inclusion given by the triangle inequality. As $\bar{g}\circ\iota_\eta\circ\bar{h}\simeq\iota_\nu$, $\iota_2\circ \iota_\nu \circ \bar{f}\simeq\iota_\mathcal{M}$ and $\iota_1\circ\iota_\alpha=\iota_\eta$ we have,
$$\iota_2\circ\bar{g}\circ(\iota_1\circ\iota_\alpha)\circ\bar{h}\circ\bar{f}=\iota_2\circ(\bar{g}\circ\iota_\eta\circ\bar{h})\circ\bar{f}\simeq\iota_2\circ\iota_\nu\circ\bar{f}\simeq\iota_\mathcal{M}$$
Also, as
$ r+2c_\alpha^n(r+\delta+\epsilon)<(2c_\alpha^n+1)(r+\delta+\epsilon)<(2c_\alpha^n+1)c\operatorname{r}_1\leq \operatorname{r}_1$
from Lemma \ref{lemma:Hausmanneq} we get $\mathrm{VR} \left(\mathcal{M} ; \epsilon \right)\simeq\mathrm{VR} \left(\mathcal{M} ;r+2c_\alpha^n \left(r+\delta+\epsilon \right) \right)\simeq\mathcal{M}$. Hence, the map induced by the inclusion on homology $(\iota_\mathcal{M})_*$ maps the fundamental class $[\mathcal{M}]$ non-trivially. If $c_\alpha^n \left(r+\delta+\epsilon \right)<\operatorname{d}_\mathrm{H}
(\mathrm{A},\mathcal{M})$ from Lemma \ref{lemma:obstruct} we have 
$$\mathrm{H}_d(\check{\mathrm{C}}_\mathcal{M} \left(\mathrm{A} ;c_\alpha^n \left(r+\delta+\epsilon \right) \right);\mathbb{Z}_2)\cong\mathrm{H}_d\left(\bigcup\limits_{x\in\mathrm{A}}\mathbb{B}(x,c_\alpha^n \left(r+\delta+\epsilon\right))\cap\mathcal{M};\mathbb{Z}_2\right)=0 $$ This would imply that $(\iota_\eta)_*=(\iota_1\circ\iota_\alpha)_*=0$ on the $d$-homology groups. Thus, $(\iota_\mathcal{M})_*([\mathcal{M}])=(\iota_2\circ\bar{g}\circ(\iota_1\circ\iota_\alpha)\circ\bar{h}\circ\bar{f})_*([\mathcal{M}])=(\iota_2\circ\bar{g})_*\circ(\iota_\eta)_*\circ(\bar{h}\circ\bar{f})_*([\mathcal{M}])=0$, which is a contradiction. 
Hence, we get, $c_\alpha^n \left(r+\delta+\epsilon \right)\geq\operatorname{d}_\mathrm{H}
(\mathrm{A},\mathcal{M})$. This is true for all small enough $\epsilon>0$, $\delta>0$ such that $2\operatorname{d}_{\mathrm{H}} \left(\mathrm{B},\mathcal{M} \right)<\delta<c\operatorname{r}_1-2\operatorname{d}_{\mathrm{GH}} \left(\mathrm{A}, \mathrm{B} \right)$ and $r>0$ such that $2\, \operatorname{d}_{\mathrm{GH}} \left(\mathrm{A}, \mathrm{B} \right)<r<c\operatorname{r}_1-\delta$. Thus we have $\operatorname{d}_{\mathrm{GH}} \left(\mathrm{A},\mathrm{B} \right) \geq \frac{1}{ 2c_\alpha^n} \operatorname{d}_{\mathrm{H}} \left(\mathrm{A},\mathcal{M} \right)- \operatorname{d}_{\mathrm{H}} \left(\mathrm{B},\mathcal{M} \right)$.
\end{proof}
\end{thm}

Taking $\mathrm{B}=\mathcal{M}$ in the theorem, so that $\operatorname{d}_\mathrm{H}(\mathrm{B},\mathcal{M})=0$, gives
$$\operatorname{d}_\mathrm{GH}(\mathrm{A},\mathcal{M})\geq\min\left\{\frac{1}{2c_\alpha^n}\operatorname{d}_\mathrm{H}(\mathrm{A},\mathcal{M}),\,\frac{c}{2}\operatorname{r}_1(\mathcal{M})\right\}.$$
Since $c_\alpha^n=\sqrt{\frac{2}{1+\sqrt{1-\alpha^2}}}\lambda_n<1$ by Remark \ref{remark:inmain}, the coefficient $\frac{1}{2c_\alpha^n}$ exceeds $\frac12$, so this improves the factor $\frac12$ of Theorem \ref{thm:firstbound}.

\begin{rem}
The coefficient $\frac{1}{2c_\alpha^n}$ decreases with $\alpha$, while the range $\operatorname{d}_\mathrm{H}(\mathrm{A},\mathcal{M})<c\,c_\alpha^n\operatorname{r}_1(\mathcal{M})$ on which the linear bound $\frac{1}{2c_\alpha^n}\operatorname{d}_\mathrm{H}(\mathrm{A},\mathcal{M})$ holds widens as $\alpha$ increases. Hence the bound is sharpest when $\mathrm{A}$ is Hausdorff-close to $\mathcal{M}$: a small $\operatorname{d}_\mathrm{H}(\mathrm{A},\mathcal{M})$ permits a small $\alpha$ and thus a larger coefficient. Figure \ref{fig:twosubsetbound} plots the resulting lower bound for $\operatorname{d}_\mathrm{GH}(\mathrm{A},\mathcal{M})$ against $\operatorname{d}_\mathrm{H}(\mathrm{A},\mathcal{M})$ for $\alpha\in\{0.25,0.5,0.75,1\}$, taking $\mathrm{B}=\mathcal{M}$, $n=2$ and $\operatorname{r}_1(\mathcal{M})=1$. In the linear region each curve lies above the line of slope $\frac12$, beyond which it saturates at $\frac{c}{2}\operatorname{r}_1(\mathcal{M})$.
\end{rem}

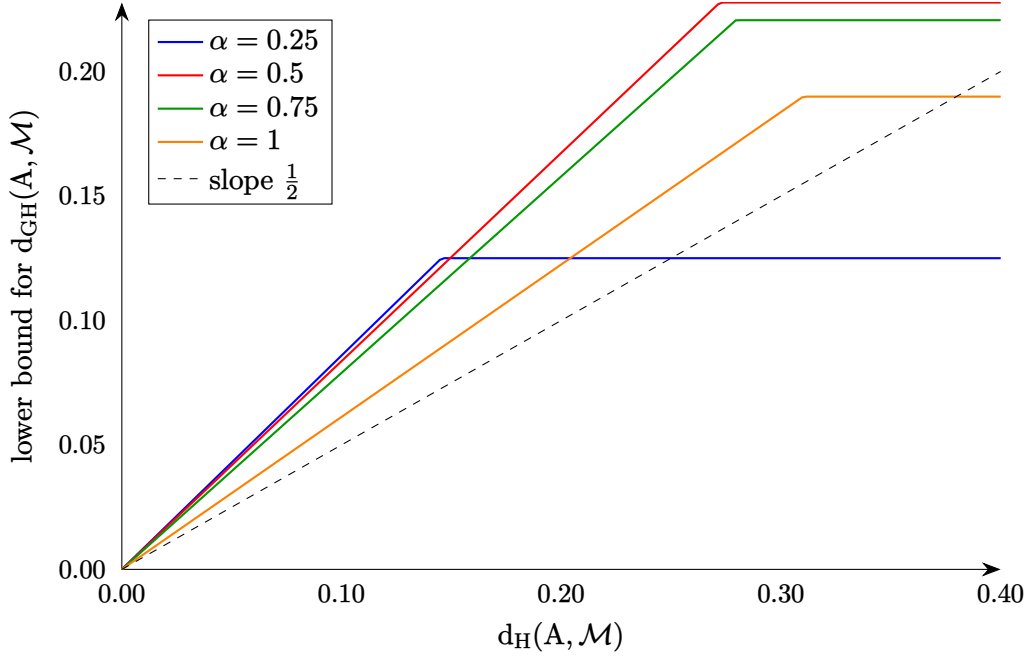
\begin{figure}[ht]
\centering
\begin{tikzpicture}
\begin{axis}[
    xlabel={$\operatorname{d}_\mathrm{H}(\mathrm{A},\mathcal{M})$},
    ylabel={lower bound for $\operatorname{d}_\mathrm{GH}(\mathrm{A},\mathcal{M})$},
    domain=0:0.4,
    samples=200,
    legend pos=north west,
    legend cell align={left},
    axis lines=left,
    axis line style={-{Stealth[length=2.5mm]}},
    tick style={draw=none},
    width=0.8\textwidth,
    height=0.55\textwidth,
    no markers,
    xmin=0, ymin=0,
    xtick={0,0.1,0.2,0.3,0.4},
    ytick={0,0.05,0.1,0.15,0.2},
    xticklabel style={/pgf/number format/fixed,
                      /pgf/number format/precision=2,
                      /pgf/number format/fixed zerofill,
                      font=\small},
    yticklabel style={/pgf/number format/fixed,
                      /pgf/number format/precision=2,
                      /pgf/number format/fixed zerofill,
                      font=\small},
    scaled ticks=false,
]
\addplot[blue,thick]{min(0.8590*x,0.1250)};
\addlegendentry{$\alpha=0.25$}
\addplot[red,thick]{min(0.8365*x,0.2277)};
\addlegendentry{$\alpha=0.5$}
\addplot[green!60!black,thick]{min(0.7893*x,0.2206)};
\addlegendentry{$\alpha=0.75$}
\addplot[orange,thick]{min(0.6124*x,0.1899)};
\addlegendentry{$\alpha=1$}
\addplot[black,dashed]{0.5*x};
\addlegendentry{slope $\frac12$}
\end{axis}
\end{tikzpicture}
\caption{Lower bound on $\operatorname{d}_\mathrm{GH}(\mathrm{A},\mathcal{M})$ as a function of $\operatorname{d}_\mathrm{H}(\mathrm{A},\mathcal{M})$ from Theorem \ref{theorem:main2set} with $\mathrm{B}=\mathcal{M}$, for $n=2$ and $\operatorname{r}_1(\mathcal{M})=1$. Each curve rises with slope $\frac{1}{2c_\alpha^n}>\frac12$ until it saturates at $\frac{c}{2}\operatorname{r}_1(\mathcal{M})$, the dashed line has slope $\frac12$.}
\label{fig:twosubsetbound}
\end{figure}

\section{Using local conditions and ambient \v{C}ech complexes}\label{sec:ambient}
In this section we use interleavings between the Vietoris-Rips complexes and ambient \v{C}ech complexes to improve on our results. If we have a simplex in a Vietoris-Rips complex at some scale, then we can expect to find a smaller circumscribing ball of the simplex if the center is unrestricted as in the case of ambient \v{C}ech complexes versus if it is restricted in the case of intrinsic \v{C}ech complexes. In order to get a contradiction as before we would want the ambient \v{C}ech complexes at a scale below a certain threshold to have trivial $d$-homology. This cannot be guaranteed even if the projection of the ambient cover does not cover the manifold. Thus we look at inclusions from the ambient \v{C}ech complexes to intrinsic \v{C}ech complexes and their induced homomorphisms.

\begin{Def}
   Let $\mathcal{M}\subset\mathbb{R}^n$ be a Euclidean submanifold with positive reach $\operatorname{r}_1>0$ and let $\pi_\mathcal{M}:\mathcal{M}^{\operatorname{r}_1}\to\mathcal M$ be the nearest point projection map to $\mathcal{M}$. Define $\psi_\mathcal{M}:\mathcal{M}\times(0,\operatorname{r}_1)\to\mathbb{R}^{\geq0}$ as $\psi_\mathcal{M}(x,r)=\sup\limits_{y\in\mathbb{B}(x,r)}\operatorname{d}(x,\pi_\mathcal{M}(y))$. For any subset $\mathrm{A}$ define, $\Psi_\mathrm{A}=\sup\limits_{0<r<h_\mathrm{A}}\sup\limits_{x\in \mathrm{A}}\frac{\psi_\mathcal{M}(x,r)}{r}$, where $h_\mathrm{A}=\min\left\{\operatorname{d}_\mathrm{H}(\mathrm{A},\mathcal{M}),\operatorname{r}_1\right\}$. Note that from Lemma \ref{lemma:estimate} $\psi_\mathcal{M}(x,r)\leq\sqrt2\,r$, thus the supremum exists.
\end{Def}
In \cite{kim2019homotopy} the authors show that the nerves of the covers formed by intersecting balls of varying radii with $\mathcal{M}$, that is, the intrinsic \v{C}ech complexes of such balls, are homotopy equivalent to $\mathcal{M}$ provided their union covers $\mathcal{M}$ and their radii are small. This can be seen as a consequence of Lemma \ref{lemma:contractible} and is recorded in Lemma \ref{lemma:mapzero}.
Since we want to use our interleaving to compare the scale of the \v{C}ech complex with the Hausdorff distance, we construct an interleaving in which the ambient \v{C}ech complex uses ambient balls of the same radius and the intrinsic \v{C}ech complex balls of different radii. This lets us formulate a sufficient condition under which, if the scale radius of the ambient balls is small compared to a function of the Hausdorff distance, the inclusion induces a trivial homomorphism on the top homology.

\begin{lemma}\label{lemma:ambtriv}
    Let $\mathcal{M}\subset\mathbb{R}^n$ be a Euclidean submanifold  with positive reach $\operatorname{r}_1>0$. Consider a subset $\mathrm{A}\subset\mathcal{M}$ with $\operatorname{d}_{\mathrm{GH}}(\mathrm{A},\mathcal{M})>0$, let $0<r<\frac{1}{\sqrt2}\operatorname{r}_1(\mathcal{M})$ and write $r_\mathrm{A}=\left\{\psi_\mathcal{M}(x,r)\right\}_{x\in\mathrm{A}}$ for the family of radii indexed by the points of $\mathrm{A}$. Then we have the inclusions
    $$\check{\mathrm{C}} \left(\mathrm{A};r \right)\xhookrightarrow[\hspace{1cm}]{\iota_{a}}\check{\mathrm{C}}_\mathcal{M} \left(\mathrm{A};r_\mathrm{A} \right) \xhookrightarrow[\hspace{1cm}]{\iota_b}\mathrm{VR} \left(\mathrm{A};2\sqrt2r \right)$$
    Let the composition be, $\iota_1=\iota_b\circ\iota_a:\check{\mathrm{C}} \left(\mathrm{A};r \right)\xhookrightarrow[\hspace{0.5cm}]{}\mathrm{VR} \left(\mathrm{A};2\sqrt2r \right)$. If $r<\frac{1}{\Psi_\mathrm{A}}\operatorname{d}_\mathrm{H}(\mathrm{A},\mathcal{M})$ then the homomorphism induced by $\iota_1$ is trivial on the $d$-th homology. More precisely, if $\sup\limits_{x\in\mathrm{A}}\psi_\mathcal{M}(x,r)<\operatorname{d}_\mathrm{H}(\mathrm{A},\mathcal{M})$ then the homomorphism induced by $\iota_1$ is trivial on the $d$-th homology. As $\psi_\mathcal{M}(x,r)\leq\Psi_\mathrm{A}r$ for every $x\in\mathrm{A}$, the condition $r<\frac{1}{\Psi_\mathrm{A}}\operatorname{d}_\mathrm{H}(\mathrm{A},\mathcal{M})$ suffices.

\begin{proof}
    For a simplex $\sigma=[z_1,\dots,z_k]\in\check{\mathrm{C}} \left(\mathrm{A};r \right)$ then $\bigcap\limits_{i=1}^k\mathbb{B}(z_i,r)\neq\emptyset$ therefore for any $z\in\bigcap\limits_{i=1}^k\mathbb{B}(z_i,r)$ we have $\operatorname{d}(\pi_\mathcal{M}(z),z_i)\leq\psi_\mathcal{M}(z_i,r)$ thus, $\pi_\mathcal{M}(z)\in \bigcap\limits_{i=1}^k\mathbb{B}(z_i,\psi_\mathcal{M}(z_i,r))\cap\mathcal{M} $ that is $\sigma \in \check{\mathrm{C}}_\mathcal{M} \left(\mathrm{A};r_\mathrm{A} \right)$. As $\psi_\mathcal{M}(x,r)\leq\sqrt{2}\,r$ for every $x\in\mathrm{A}$, using the triangle inequality we have the inclusion $\iota_b$. Now if $r<\frac{1}{\Psi_\mathrm{A}}\operatorname{d}_\mathrm{H}(\mathrm{A},\mathcal{M})$, then there exists $y\in \mathcal{M}$ with $\operatorname{d}_\mathrm{A}(y)>\Psi_\mathrm{A}r\geq \psi_\mathcal{M}(x,r)$ for every $x\in\mathrm{A}$. Thus, the cover $\mathcal{U}^\mathcal{M}_{r_\mathrm{A}}=\left\{\mathbb{B}(x,\psi_\mathcal{M}(x,r))\cap\mathcal{M}\right\}_{x\in\mathrm{A}}$ does not cover $\mathcal{M}$. As $\psi_\mathcal{M}(x,r)\leq\sqrt2\,r<\operatorname{r}_1(\mathcal{M})$ for every $x\in\mathrm{A}$, Lemma \ref{lemma:mapzero} applies and gives $\mathrm{H}_d\left(\check{\mathrm{C}}_\mathcal{M} \left(\mathrm{A};r_\mathrm{A} \right);\mathbb{Z}_2\right)=0$ and $(\iota_b)_*=0$ on $\mathrm{H}_d\left(\check{\mathrm{C}}_\mathcal{M} \left(\mathrm{A};r_\mathrm{A} \right);\mathbb{Z}_2\right)$. Thus, $(\iota_1)_*=(\iota_b)_*\circ(\iota_a)_*:\mathrm{H}_d\left(\check{\mathrm{C}}\left(\mathrm{A};r\right);\mathbb{Z}_2\right)\to \mathrm{H}_d\left(\mathrm{VR}\left(\mathrm{A};2\sqrt2r\right);\mathbb{Z}_2\right)$ is trivial.
\end{proof}
\end{lemma}
In specific cases, the function $\psi_\mathcal{M}$ can be found explicitly. For example in the case when $\mathcal{M}=\mathbb{S}^n$ we have $\psi_{\mathbb{S}^n}(x,r) = \sqrt{2-2\sqrt{1-r^2}}$. We start by proving some bounds on the scales of the ambient \v{C}ech complex to have an interleaving with Vietoris-Rips complex at appropriate scale. We use the properties of the nearest point projection map as defined earlier and the nearest point projection map to the tangent space at a particular point on the manifold.

\noindent The following lemma is due to Federer (Theorem 4.8 (7) \cite{federer1959curvature}) and is stated as in \cite{boissonnat2019reach}.
\begin{lemma} \label{lemma:alphabeta} (Lemma 11 \cite{boissonnat2019reach})
    Let $\mathcal{M}$ be a $C^{1,1}$ Euclidean submanifold with reach $\operatorname{r}_1 \left(\mathcal{M} \right)>0$. Let $p, q \in \mathcal{M} \subset \mathbb{R}^{n}$ such that $|p-q|<\operatorname{r}_1 \left(\mathcal{M} \right)$. We have

$$
\operatorname{d} \left(q, T_{p} \mathcal{M} \right) \leq  \left(1-\sqrt{1- \left(\frac{\operatorname{d} \left(\pi_{T_p\mathcal{M}} \left(q \right),p \right)}{\operatorname{r}_1 \left(\mathcal{M} \right)} \right)^{2}} \right)\operatorname{r}_1 \left(\mathcal{M} \right)
$$
\end{lemma}
To bound a simplex in $\mathbb{R}^n$ with a fixed diameter $D$, we can always find a circumscribing sphere of radius $\lambda_n D$ where $\lambda_n$ is the constant $\lambda_n=\sqrt{\frac{n}{2(n+1)}}$ obtained from Jung's theorem. If the simplex was on some $d$-dimensional hyperplane for $d<n$ instead then we could find a smaller circumscribing sphere of radius $\lambda_dr<\lambda_nr$ as $\lambda_n$ increases with $n$. If the simplex has vertices on a manifold $\mathcal{M}\subset\mathbb{R}^n$ of dimension $d$, then we can pick the tangent space at any of the vertices and project the rest of the points to this hyperplane orthogonally. This will allow us to construct a smaller circumscribing sphere for the projected simplex than if we were to use Jung's theorem in the ambient space $\mathbb{R}^n$ (note that the diameter of the projected simplex is less than or equal to the original as orthogonal projection to a plane is 1-Lipschitz). We can then see how big the sphere has to be to contain the original simplex using Lemma \ref{lemma:alphabeta} and see if this bound is better than the one obtained by using Jung's theorem in $\mathbb{R}^n$. 

\begin{lemma}\label{lemma:gamma}
    Let $\mathcal{M}\subset\mathbb{R}^n$ be a $d$-dimensional  $C^{1,1}$ Euclidean submanifold of reach $\operatorname{r}_1 \left(\mathcal{M} \right)>0$ and $\mathrm{A}\subset \mathcal{M}$. With $\lambda_n = \sqrt{\frac{n}{2 \left(n+1 \right)}}$ define $\alpha_{d,n},\,\beta_{d,n}$ and $\gamma$ as follows,
    \begin{align*}
        & \alpha_{d,n} = \sqrt{\lambda_n^2-\lambda_d^2}\\
        & \beta_{d,n} = \frac{2\alpha_{d,n}}{\alpha_{d,n}^2+1}\\
        & \gamma \left(x \right) = \sqrt{\lambda_d^2+\frac{ \left(1-\sqrt{1-x^2} \right)^2}{x^2}}
    \end{align*}
    Then we have the inclusion $\iota_r:\mathrm{VR} \left(\mathrm{A};r \right)\xhookrightarrow[\hspace{0.5cm}]{}\check{\mathrm{C}} \left(\mathrm{A};\gamma \left(\tau \right) r \right)$ and $\gamma(\tau)<\lambda_n$ for any $r$ and $\tau$ such that $\frac{r}{\operatorname{r}_1 \left(\mathcal{M} \right)}\leq\tau<\beta_{d,n}$.
    \begin{proof}
        Take a simplex $\sigma = [z_1,\dots,z_m]\in \mathrm{VR} \left(\mathrm{A};r \right)$ and consider the $d$-dimensional tangent plane $T_{z_1}\mathcal{M}$. Let $\pi_{1}$ be the orthogonal projection map onto the plane $T_{z_1}\mathcal{M}$. As the orthogonal projection is 1-Lipschitz,
        we get $\operatorname{d} \left(y_i,y_j \right)\leq \operatorname{d} \left(z_i,z_j \right)\leq \operatorname{diam} \left(\sigma \right)<r$ for $i,j\in\{1,\dots, m\}$, where $y_i = \pi_1 \left(z_i \right)$. Now the set $\pi_1 \left(\sigma \right)=\{y_1=z_1,y_2\dots,y_m\}$ has diameter $\operatorname{diam} \left(\pi_1 \left(\sigma \right) \right)\leq \operatorname{diam} \left(\sigma \right)<r$. Using Jung's theorem in the $d$-dimensional affine subspace $T_{z_1}\mathcal{M}$ there exists a $d$-dimensional ball $\mathbb{B}_{T_{z_1}\mathcal{M}} \left(c,r_0 \right)$ circumscribing $\pi_1 \left(\sigma \right)$ with radius $r_0< \lambda_d\,\operatorname{diam}(\sigma)<\lambda_d\,r$. As $\operatorname{d} \left(z_1,z_i \right)<\operatorname{r}_1 \left(\mathcal{M} \right)$, we use Lemma ~\ref{lemma:alphabeta} to get 
        $$\operatorname{d} \left(y_i,z_i \right)\leq  \left(1-\sqrt{1- \left(\frac{\operatorname{d}\left(y_i,z_1 \right)}{\operatorname{r}_1 \left(\mathcal{M} \right)} \right)^2} \right)\operatorname{r}_1 \left(\mathcal{M} \right) $$ As, we have  $z_i-y_i$ is orthogonal to $T_{z_1}\mathcal{M}$  we have $$\operatorname{d} \left(c,z_i \right)=\sqrt{\operatorname{d} \left(c,y_i \right)^2+\operatorname{d} \left(y_i,z_i \right)^2}<\sqrt{r_0^2+ \left(1-\sqrt{1- \left(\frac{\operatorname{d}\left(y_i,z_1 \right)}{\operatorname{r}_1 \left(\mathcal{M} \right)} \right)^2} \right)^2 \left(\operatorname{r}_1 \left(\mathcal{M} \right) \right)^2}$$ As, $\operatorname{d} \left(y_i,z_1 \right)<\operatorname{d} \left(z_i,z_1 \right)<r\leq\operatorname{r}_1 \left(\mathcal{M} \right)\,\tau$ and $r_0<\lambda_d\,r$
        $$\operatorname{d} \left(c,z_i \right)<\sqrt{ \left(\lambda_d\,r \right)^2+ \left(1-\sqrt{1- \left(\frac{r}{\operatorname{r}_1 \left(\mathcal{M} \right)} \right)^2} \right)^2 \left(\operatorname{r}_1 \left(\mathcal{M} \right) \right)^2}=r\,\gamma \left(\frac{r}{\operatorname{r}_1 \left(\mathcal{M} \right)} \right)$$
        Now $\gamma \left(x \right)$ is increasing thus, $\operatorname{d} \left(c,z_i \right)<\gamma \left(\tau \right)r$ and $\sigma \subset\mathbb{B} \left(c,\gamma \left(\tau \right)\,r \right)$. Therefore, $\sigma\in\check{\mathrm{C}} \left(\mathrm{A};\gamma \left(\tau \right) r \right)$. Now, $$ \gamma \left(\beta_{d,n} \right)=\sqrt{\lambda_d^2+\frac{ \left(1-\sqrt{1- \left(\beta_{d,n} \right)^2} \right)^2}{\beta_{d,n}^2}} =\sqrt{\lambda_d^2+\alpha_{d,n}^2}=\lambda_n$$
        so, $\gamma \left(\tau \right)<\gamma \left(\beta_{d,n} \right)=\lambda_n$ for $\tau<\beta_{d,n}$.
    \end{proof} 
\end{lemma}
Thus, we see that the bound obtained by projecting the simplex to the tangent space and then using Jung's theorem on the tangent space and Lemma \ref{lemma:alphabeta} to obtain a circumradius improves on the circumradius obtained by using Jung's theorem in the ambient space when the simplices are small enough.
\begin{thm}\label{theorem:ambientmain}
    Let $\mathcal{M}\subset \mathbb{R}^n$ be a connected closed $d$-dimensional $C^{1,1}$ Euclidean submanifold with reach $\operatorname{r}_1 \left(\mathcal{M} \right)>0$. For  $\tau$ such that $0<\tau<\beta_{d,n}$, let $c=\operatorname{min}\left(\tau,\frac{1}{2\sqrt2\lambda_n+1}\right)$. Then for subsets $\mathrm{A}, \mathrm{B} \subset \mathcal{M}$,
\begin{align*}
    &1.)\;\operatorname{d}_{\mathrm{GH}} \left(\mathrm{A},\mathcal{M} \right) \geq \min \left\{\frac{1}{2 \Psi_\mathrm{A}\,\gamma \left(\tau \right)} \operatorname{d}_{\mathrm{H}} \left(\mathrm{A},\mathcal{M} \right), \frac{c}{2}\operatorname{r}_1(\mathcal{M})\right\}\\
    &2.)\;\operatorname{d}_{\mathrm{GH}} \left(\mathrm{A},\mathrm{B} \right) \geq \min \left\{\frac{1}{2\Psi_\mathrm{A}\gamma \left(\tau \right)} \operatorname{d}_{\mathrm{H}} \left(\mathrm{A},\mathcal{M} \right)- \operatorname{d}_{\mathrm{H}} \left(\mathrm{B},\mathcal{M} \right), \frac{c}{2}\operatorname{r}_1(\mathcal{M})-\operatorname{d}_{\mathrm{H}} \left(\mathrm{B},\mathcal{M}\right)\right\}\\
\end{align*}

\begin{proof}
   Let $\operatorname{r}_1 \left(\mathcal{M} \right)=\operatorname{r}_1$ and assume  $\operatorname{d}_{\mathrm{GH}} \left(\mathrm{A}, \mathrm{B} \right)<\frac{c}{2}\operatorname{r}_1- \operatorname{d}_{\mathrm{H}} \left(\mathrm{B},\mathcal{M} \right)$. For $2\operatorname{d}_{\mathrm{H}} \left(\mathrm{B},\mathcal{M} \right)<\delta<c\operatorname{r}_1-2\operatorname{d}_{\mathrm{GH}} \left(\mathrm{A}, \mathrm{B} \right)$ pick $r$ such that $2\, \operatorname{d}_{\mathrm{GH}} \left(\mathrm{A}, \mathrm{B} \right)<r<c\operatorname{r}_1-\delta$. Now, take $\epsilon>0$ so that
$r+\delta+\epsilon<c\operatorname{r}_1$. As $r+\delta+\epsilon<c\operatorname{r}_1\leq\tau\operatorname{r}_1$ from Lemma ~\ref{lemma:gamma}, ~\ref{lemma:lemvrh} and ~\ref{lemma:lemvrc} we get the following sequence of simplicial maps, 

\begin{center}
     \begin{tikzcd}[column sep=0.8em]
         & \check{\mathrm{C}} \left(\mathrm{A} ;\gamma \left(\tau \right)\left(r+\delta+\epsilon \right) \right)\arrow[rd,"\iota_1"]&\\
     \mathrm{VR} \left(\mathrm{A};r+\delta+\epsilon \right)\arrow[rr,dashed,hook,"\iota_0"]\arrow[ru,"\iota_r"]&& \mathrm{VR} \left(\mathrm{A};2\sqrt{2}\gamma \left(\tau \right)\left(r+\delta+\epsilon \right) \right)\arrow[d,"\bar{g}"]\\
      \mathrm{VR} \left(\mathrm{B};\delta+\epsilon \right)\arrow[u,"\bar{h}"]&&\mathrm{VR} \left(\mathrm{B};r+2\sqrt2\gamma \left(\tau \right)\left(r+\delta+\epsilon \right) \right)\arrow[d,"\iota_2"]\\
      \mathrm{VR} \left(\mathcal{M} ; \epsilon \right) \arrow[u,"\bar{f}"]&&\mathrm{VR} \left(\mathcal{M} ;r+2\sqrt2\gamma \left(\tau \right)\left(r+\delta+\epsilon \right) \right)\\
        \end{tikzcd}
\end{center}
Note that as $\lambda_d<\gamma(\tau)<\lambda_n$, we have,
$$ r+2\sqrt2\gamma \left(\tau \right)(r+\delta+\epsilon)<(2\sqrt2\gamma \left(\tau \right)+1)(r+\delta+\epsilon)<(2\sqrt2\gamma \left(\tau \right)+1)c\operatorname{r}_1\leq\left(\frac{2\sqrt2\gamma \left(\tau \right)+1}{2\sqrt2\lambda_n+1}\right)\operatorname{r}_1<\operatorname{r}_1$$
From Lemma \ref{lemma:Hausmanneq} we get $\mathrm{VR} \left(\mathcal{M} ; \epsilon \right)\simeq\mathrm{VR} \left(\mathcal{M} ;r+2\sqrt2\gamma \left(\tau \right)\left(r+\delta+\epsilon \right) \right)\simeq\mathcal{M}$.
Therefore, from Lemma \ref{lemma:ambtriv} we have $\gamma \left(\tau \right)\left(r+\delta+\epsilon \right)>\frac{1}{\Psi_\mathrm{A}}\operatorname{d}_\mathrm{H} \left(\mathrm{A},\mathcal{M} \right)$ for all small enough $\epsilon>0$, all $\delta$ such that $2 \operatorname{d}_{\mathrm{H}} \left(\mathrm{B}, \mathcal{M} \right)<\delta<c\operatorname{r}_1-2\operatorname{d}_{\mathrm{GH}} \left(\mathrm{A}, \mathrm{B} \right)$ and all $r$ such that $2\, \operatorname{d}_{\mathrm{GH}} \left(\mathrm{A}, \mathrm{B} \right)<r<c\operatorname{r}_1-\delta$. Substituting we get, $2\, \operatorname{d}_{\mathrm{GH}} \left(\mathrm{A}, \mathrm{B} \right)+2 \operatorname{d}_{\mathrm{H}} \left(\mathrm{B}, \mathcal{M} \right)>\frac{1}{\Psi_\mathrm{A}\gamma \left(\tau \right)}\operatorname{d}_{\mathrm{H}} \left(\mathrm{A}, \mathcal{M} \right)$. Therefore, $\operatorname{d}_{\mathrm{GH}} \left(\mathrm{A}, \mathrm{B} \right)>\frac{1}{2\Psi_\mathrm{A}\gamma \left(\tau \right)} \operatorname{d}_{\mathrm{H}} \left(\mathrm{A},\mathcal{M} \right)- \operatorname{d}_{\mathrm{H}} \left(\mathrm{B},\mathcal{M} \right)$.
\end{proof}
\end{thm}

We note that from Lemma \ref{lemma:estimate}, as $\operatorname{d}_\mathrm{H}(\mathrm{A},\mathcal{M})\to0$ we have $\Psi_\mathrm{A}\to 1$. The constant $c_\alpha^n=\sqrt{\frac{2}{1+\sqrt{1-\alpha^2}}}\lambda_n$ appearing in Theorem \ref{theorem:main2set} is strictly larger than $\lambda_n$, whereas $\gamma(\tau)$ from Theorem \ref{theorem:ambientmain} is strictly smaller than $\lambda_n$. So, for dense enough (in the sense that $\Psi_\mathrm{A}-1$ is small) subsets, we have $\frac{1}{2\Psi_\mathrm{A}\gamma(\tau)}\geq \frac{1}{2\lambda_n}>\frac{1}{2c_\alpha^n}$.
As mentioned previously, $\psi_{\mathbb{S}^n}(x,r) = \sqrt{2-2\sqrt{1-r^2}}$, we use this and some geometric properties  of $\mathbb{S}^n$ to give an example of using ambient \v{C}ech complexes and Lemma \ref{lemma:ambtriv} to obtain lower bounds on Gromov-Hausdorff distance.
In \cite{dekster1997jung} the author extends Jung's theorem to manifolds with curvature bounded above, using this we can get an interleaving between the Vietoris-Rips and \v{C}ech complexes in the case of $\mathbb{S}^n$.
\begin{thm}\label{theorem:dekster}
  Let $\mathcal{M}$ be a compact, connected, $d$-dimensional Euclidean submanifold with sectional curvatures bounded above by $\kappa \in \mathbb{R}$. Let $\mathrm{U} \subset \mathcal{M}$ be compact such that, $\operatorname{diam} \left(\mathrm{U} \right)<\rho \left(\mathcal{M} \right)$ for $\kappa \leq 0$ and $\operatorname{diam} \left(\mathrm{U} \right)<\min \left\{\rho \left(\mathcal{M} \right), \frac{2 \pi}{3 \sqrt{\kappa}}\right\}$ if $\kappa>0$. Then the circumcenter $c \left(\mathrm{U} \right)$ exists in $\mathcal{M}$ and
$$
\operatorname{diam} \left(\mathrm{U} \right) \geq \begin{cases}\frac{2}{\sqrt{-\kappa}} \sinh ^{-1} \left(\sqrt{\frac{d+1}{2 d}} \sinh  \left(\sqrt{-\kappa}\; \operatorname{R} \left(\mathrm{U} \right) \right) \right) & \text { for } \kappa<0 \\ 2 \operatorname{R} \left(\mathrm{U} \right) \sqrt{\frac{d+1}{2 d}} & \text { for } \kappa=0 \\ \frac{2}{\sqrt{\kappa}} \sin ^{-1} \left(\sqrt{\frac{d+1}{2 d}} \sin  \left(\sqrt{\kappa} \operatorname{R} \left(\mathrm{U} \right) \right) \right) & \text { for } \kappa>0 \text { and } \operatorname{R} \left(\mathrm{U} \right) \in\left[0, \frac{\pi}{2 \sqrt{\kappa}}\right]\end{cases}
$$
where $\operatorname{R} \left(\mathrm{U} \right)$ is the circumradius and $\rho \left(\mathcal{M} \right)$ is the convexity radius of $\mathcal{M}$, that is, the supremum of all $s>0$ such that every geodesic ball in $\mathcal{M}$ of radius less than $s$ is convex. For the unit sphere $\rho \left(\mathbb{S}^{d} \right)=\frac{\pi}{2}$.
\end{thm}
\noindent The sectional curvature of an $n$-sphere of radius $R$ is constant everywhere and is equal to $\frac{1}{R^2}$.
\begin{lemma}\label{lemma:incv}
Let $\mathrm{A}$ be a subset of the unit $(n-1)$-sphere $\mathbb{S}^{n-1}\subset\mathbb{R}^{n}$. For $0<r<\sqrt2$, we have the inclusion, 
$\iota_r:\mathrm{VR} \left(\mathrm{A};r \right)\xhookrightarrow[\hspace{0.5cm}]{}\check{\mathrm{C}} \left(\mathrm{A};\lambda_{n-1}r \right)$, where $\lambda_{n}=\sqrt{\frac{n}{2 \left(n+1 \right)}}$.
\begin{proof}
    Consider $\sigma\in\mathrm{VR} \left(\mathrm{A};r \right)$ and $\sigma=\{z_1,z_2,\dots,z_k\}$. As $\operatorname{diam} \left(\sigma \right)<r<\sqrt2$ the diameter of the set $\sigma$ in the geodesic metric on $\mathbb{S}^{n-1}$ denoted by $\operatorname{diam}_g \left(\sigma \right)$ is less than $2\sin^{-1} \left(\frac{r}{2} \right)$. Thus using the extended Jung's Theorem there exists a geodesic ball on $\mathbb{S}^{n-1}$, $\mathbb{B}_{\mathbb{S}^{n-1}} \left(c,r_0 \right)$ containing $\sigma$ with  $\operatorname{diam}_g \left(\sigma \right)\geq 2\sin^{-1} \left(\frac{1}{2\lambda_{n-1}}\sin \left(r_0 \right) \right)$. As this geodesic ball is the spherical cap $\{z\in\mathbb{S}^{n-1} | \langle z,c\rangle>\cos{r_0}\}$, it is enclosed by the Euclidean ball of radius $r'=\sin{r_0}$ centered at $\frac{c}{\lVert c\rVert}\cos \left(r_0 \right)$. As $\operatorname{diam}_g \left(\sigma \right)<2\sin^{-1} \left(\frac{r}{2} \right)$, $$r'=\sin{r_0}<2\lambda_{n-1}\sin \left(\frac{\operatorname{diam}_g \left(\sigma \right)}{2} \right)<2\lambda_{n-1}\sin \left(\frac{2\sin^{-1} \left(\frac{r}{2} \right)}{2} \right)=\lambda_{n-1}r$$
    Thus, $\sigma\in\check{\mathrm{C}} \left(\mathrm{A};\lambda_{n-1}r \right)$.
\end{proof}
\end{lemma}
\noindent Thus we have the following for $\mathbb{S}^{n-1}$.
\begin{thm}\label{theorem:spherebound}
    Consider two subsets $\mathrm{A},\mathrm{B}\subset\mathbb{S}^{n-1}\subset\mathbb{R}^n$ with $\operatorname{d}_\mathrm{H}\left(\mathrm{A},\mathbb{S}^{n-1}\right)\leq\sqrt2$. Let \\ $c=\min\left\{\sqrt2,\frac{1}{\sqrt2\lambda_n(2\lambda_{n-1}+1)}\right\}$. Then the following hold.
\begin{align*}
&1.) \;\operatorname{d}_{\mathrm{GH}} \left(\mathrm{A},\mathbb{S}^{n-1} \right)\geq\operatorname{min}\{\frac{1}{2\lambda_{n-1}}\psi_{\mathbb{S}^{n-1}}^{-1}\left(\operatorname{d}_\mathrm{H} \left(\mathrm{A},\mathbb{S}^{n-1} \right)\right),\frac{c}{2}\}\;\\
&2.) \;\operatorname{d}_{\mathrm{GH}} \left(\mathrm{A},\mathrm{B} \right)\geq\operatorname{min}\{\frac{1}{2\lambda_{n-1}}\psi_{\mathbb{S}^{n-1}}^{-1}\left(\operatorname{d}_\mathrm{H} \left(\mathrm{A},\mathbb{S}^{n-1} \right)\right)-\operatorname{d}_{\mathrm{H}} \left(\mathrm{B},\mathbb{S}^{n-1} \right),\frac{c}{2}-\operatorname{d}_{\mathrm{H}} \left(\mathrm{B},\mathbb{S}^{n-1} \right)\}\;
\end{align*} 
\begin{proof}
     Assume, $\operatorname{d}_{\mathrm{GH}} \left(\mathrm{A}, \mathrm{B} \right)<\frac{c}{2}-\operatorname{d}_{\mathrm{H}} \left(\mathrm{B},\mathbb{S}^{n-1} \right)$. For any $2\operatorname{d}_{\mathrm{H}}\left(\mathrm{B},\mathbb{S}^{n-1}\right)<\delta<c-2\operatorname{d}_{\mathrm{GH}}\left(\mathrm{A},\mathrm{B}\right)$ fix $r$ such that $2 \operatorname{d}_{\mathrm{GH}} \left(\mathrm{A},\mathrm{B} \right)<r<c-\delta$ and take $\epsilon>0$ so that $r+\delta+\epsilon<c$
Using Lemmas \ref{lemma:lemvrc}, \ref{lemma:lemvrh} and \ref{lemma:incv}, we have the following sequence of maps,

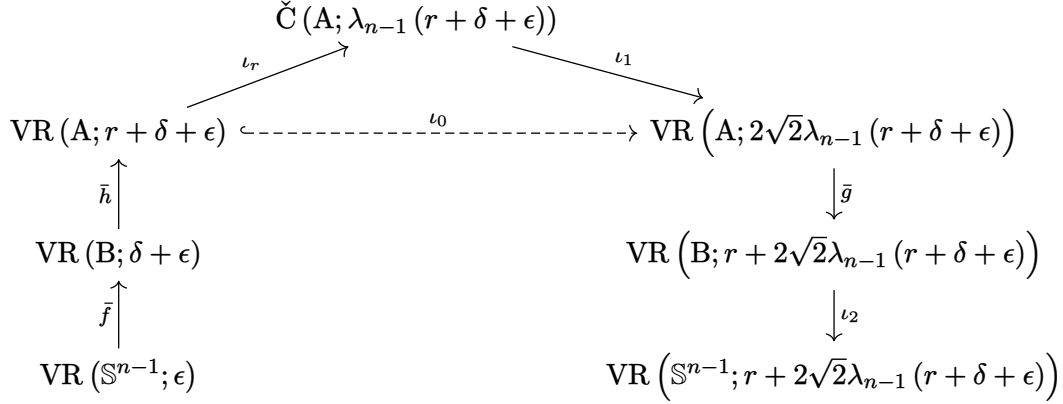
\begin{center}
     \begin{tikzcd}[column sep=0.8em]
         & \check{\mathrm{C}} \left(\mathrm{A} ;\lambda_{n-1} \left(r+\delta+\epsilon \right) \right)\arrow[rd,"\iota_1"]&\\
     \mathrm{VR} \left(\mathrm{A};r+\delta+\epsilon \right)\arrow[rr,dashed,hook,"\iota_0"]\arrow[ru,"\iota_r"]&& \mathrm{VR} \left(\mathrm{A};2\sqrt2\lambda_{n-1} \left(r+\delta+\epsilon \right) \right)\arrow[d,"\bar{g}"]\\
      \mathrm{VR} \left(\mathrm{B};\delta+\epsilon \right)\arrow[u,"\bar{h}"]&&\mathrm{VR} \left(\mathrm{B};r+2\sqrt{2}\lambda_{n-1} \left(r+\delta+\epsilon \right) \right)\arrow[d,"\iota_2"]\\
      \mathrm{VR} \left(\mathbb{S}^{n-1} ; \epsilon \right) \arrow[u,"\bar{f}"]&&\mathrm{VR} \left(\mathbb{S}^{n-1} ;r+2\sqrt2\lambda_{n-1} \left(r+\delta+\epsilon \right) \right)\\
        \end{tikzcd}
\end{center}
Note that $r+2\sqrt2\lambda_{n-1}(r+\delta+\epsilon)<(1+2\sqrt2\lambda_{n-1})(r+\delta+\epsilon)<(1+2\sqrt2\lambda_{n-1})c\leq\frac{1}{\lambda_{n}}$. Hence, from Theorem \ref{theorem:snthe}, Theorem \ref{theorem:gillespie} and Whitehead's theorem we have,$$\mathrm{VR} \left(\mathbb{S}^{n-1} ; \epsilon \right)\simeq\mathrm{VR} \left(\mathbb{S}^{n-1} ;r+2\sqrt2\lambda_{n-1} \left(r+\delta+\epsilon \right) \right)\simeq\mathbb{S}^{n-1}$$ Note that the inclusion $\iota_1$ is as in Lemma \ref{lemma:ambtriv}. So, using Lemma \ref{lemma:ambtriv}, $\psi_{\mathbb{S}^{n-1}} \left(\lambda_{n-1}\left(r+\delta+\epsilon \right)\right)>\operatorname{d}_\mathrm{H} \left(\mathrm{A},\mathbb{S}^{n-1} \right)$ for all $\delta$ such that, $2\operatorname{d}_{\mathrm{H}}\left(\mathrm{B},\mathbb{S}^{n-1}\right)<\delta<c-2\operatorname{d}_{\mathrm{GH}}\left(\mathrm{A},\mathrm{B}\right)$ and $r$ such that, $2 \operatorname{d}_{\mathrm{GH}} \left(\mathrm{A},\mathrm{B} \right)<r<c-\delta$. Thus, we get $2\operatorname{d}_{\mathrm{H}} \left(\mathrm{B},\mathbb{S}^{n-1} \right)+2\operatorname{d}_{\mathrm{GH}} \left(\mathrm{A},\mathrm{B} \right)>\frac{1}{\lambda_{n-1}}\psi^{-1}_{\mathbb{S}^{n-1}}\left(\operatorname{d}_\mathrm{H} \left(\mathrm{A},\mathbb{S}^{n-1} \right)\right)$ ($\psi_{\mathbb{S}^{n-1}}$ is an increasing function), therefore $\operatorname{d}_{\mathrm{GH}} \left(\mathrm{A},\mathrm{B} \right)>\frac{1}{2\lambda_{n-1}}\psi_{\mathbb{S}^{n-1}}^{-1}\left(\operatorname{d}_\mathrm{H} \left(\mathrm{A},\mathbb{S}^{n-1} \right)\right)-\operatorname{d}_{\mathrm{H}} \left(\mathrm{B},\mathbb{S}^{n-1} \right)$.
\end{proof}
\end{thm}

For $\mathbb{S}^2\subset\mathbb{R}^3$ (so $n=3$, $\lambda_n=\lambda_3$, $\lambda_{n-1}=\lambda_2$ and $\operatorname{r}_1=1$) we compare the two lower bounds for $\operatorname{d}_\mathrm{GH}(\mathrm{A},\mathbb{S}^2)$: the general bound $\frac{1}{2c_\alpha^n}\operatorname{d}_\mathrm{H}(\mathrm{A},\mathbb{S}^2)$ from Theorem \ref{theorem:main2set} and the sphere-specific bound $\frac{1}{2\lambda_{n-1}}\psi^{-1}_{\mathbb{S}^2}(\operatorname{d}_\mathrm{H}(\mathrm{A},\mathbb{S}^2))$.  Figure \ref{fig:spherebound} plots the two bounds.

\begin{figure}[b]
\centering
\begin{tikzpicture}
\begin{axis}[
    xlabel={$\operatorname{d}_\mathrm{H}(\mathrm{A},\mathbb{S}^2)$},
    ylabel={lower bound for $\operatorname{d}_\mathrm{GH}(\mathrm{A},\mathbb{S}^2)$},
    domain=0:0.25,
    samples=100,
    legend pos=north west,
    legend cell align={left},
    axis lines=left,
    axis line style={-{Stealth[length=2.5mm]}},
    tick style={draw=none},
    width=0.7\textwidth,
    height=0.5\textwidth,
    no markers,
    xmin=0, ymin=0,
    xtick={0.05,0.1,0.15,0.2,0.25},
    xticklabel style={/pgf/number format/fixed,
                      /pgf/number format/precision=2,
                      /pgf/number format/fixed zerofill,
                      font=\small},
    scaled x ticks=false,
    ytick={0.05,0.1,0.15,0.2,0.25},
        yticklabel style={/pgf/number format/fixed,
                      /pgf/number format/precision=2,
                      /pgf/number format/fixed zerofill,
                      font=\small},
    scaled y ticks=false,
]
\addplot[blue,thick]{min(0.8100*x,10)};
\addlegendentry{$\tfrac{1}{2c_\alpha^n}\operatorname{d}_\mathrm{H}$, $\alpha=0.25$}
\addplot[red,thick]{min(0.7887*x,10)};
\addlegendentry{$\tfrac{1}{2c_\alpha^n}\operatorname{d}_\mathrm{H}$, $\alpha=0.5$}
\addplot[green!60!black,thick]{min(0.7442*x,10)};
\addlegendentry{$\tfrac{1}{2c_\alpha^n}\operatorname{d}_\mathrm{H}$, $\alpha=0.75$}
\addplot[orange,thick]{min(0.5774*x,10)};
\addlegendentry{$\tfrac{1}{2c_\alpha^n}\operatorname{d}_\mathrm{H}$, $\alpha=1$}
\addplot[black,very thick]{min(0.8660254*x*sqrt(1-x^2/4),10)};
\addlegendentry{$\tfrac{1}{2\lambda_{n-1}}\psi^{-1}_{\mathbb{S}^2}(\operatorname{d}_\mathrm{H})$}
\end{axis}
\end{tikzpicture}
\caption{Comparison for $\mathbb{S}^2\subset\mathbb{R}^3$ ($n=3$, $\operatorname{r}_1(\mathbb{S}^2)=1$) of the lower bounds on $\operatorname{d}_\mathrm{GH}(\mathrm{A},\mathbb{S}^2)$ as functions of $\operatorname{d}_\mathrm{H}(\mathrm{A},\mathbb{S}^2)$ for a sufficiently dense subset $\mathrm{A}$.}
\label{fig:spherebound}
\end{figure}
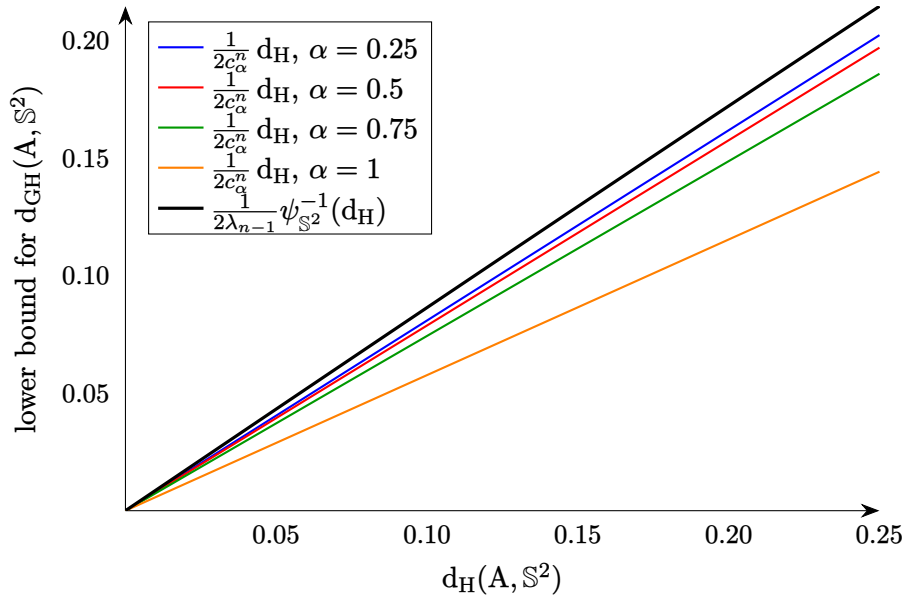

\section{Discussion and future work}
In the article above the submanifold is assumed to have positive reach, in recent work there are results relating the distance function of the submanifold and the geometry of the offsets for weaker restrictions on regularity of the submanifold and its distance function (\cite{Chazal2009,CHAZAL2005304}). In particular $\mu$-reach has been used as a sampling parameter to  geometric
and topological approximation problems on compact subsets \cite{Chazal2007-jz,Chazal2009}. We would next want to see if analogous lower bounds on Gromov-Hausdorff distance between subsets of Euclidean submanifold can be found for a weaker regularity condition specifically, if the submanifold has positive $\mu$-reach for $\mu<1$. In \cite{kim2019homotopy} the authors give a homotopy between a submanifold of positive $\mu$-reach and its offset, this can be seen as a homotopy between the submanifold and the geometric realization of the ambient \v{C}ech complex on the manifold. The bounds on the distortion of the nearest point projection map when the submanifold has positive reach and well defined tangent spaces help us in establishing the lower bounds and interleavings between complexes found in Section \ref{sec:GH2rips} and Section \ref{sec:ambient}. Finding such bounds in the case of positive $\mu$-reach can help in extending the lower bounds to submanifolds of positive $\mu$-reach and their subsets.

The homotopy type of the Vietoris-Rips complex is of interest in proving results as above and otherwise \cite{zaremsky2021bestvinabradydiscretemorsetheory,Adamaszek_2017,Hausmann+2016+175+188}. In the above results the homotopy between the Vietoris-Rips complex (at small scales) and the underlying manifold is established via the Vietoris-Rips metric thickening. In \cite{Virk2021} the authors prove that the Vietoris-Rips complex can be seen as the nerve of an appropriate cover. Can we find an explicit good cover in the Euclidean case  with the nerve of the cover being the Vietoris-Rips complex and is homotopic to the underlying space? For example, in case of $\mathbb{S}^1$ equipped with the Euclidean metric, the projections of the interiors of the 1-simplices of the Vietoris-Rips complex $\mathrm{VR}\left(\mathbb{S}^1;r\right)$, seen as chords, are geodesically convex (arcs of arc-length less than $\frac{2\pi}{3}$) for small enough $r$. Thus, finite intersections of such projections are contractible. Therefore the set of the projections of the interiors of the chords forms a good cover. Also the nerve of this cover is the Vietoris-Rips complex $\mathrm{VR}\left(\mathbb{S}^1;r\right)$.In general for embedded submanifolds of positive reach it can be shown that the image of a simplex with diameter smaller than the reach under the projection map is contractible. 

In \cite{adams2025lowerboundinggromovhausdorffdistance} the authors prove that the Hausdorff and Gromov-Hausdorff distance between a circle and its subset seen as metric spaces with the geodesic metric are equal, given the Gromov-Hausdorff distance is smaller than $\frac\pi3$. They also give a counterexample to prove that this bound in terms of the Gromov-Hausdorff closeness is tight. We would like to see if the bound in Corollary \ref{cor:circeq} above is tight in this sense. That is, for any $\epsilon>0$, can we find a subset $\mathrm{A}\subset \mathbb{S}^1$ where $c_0<\psi(d_\mathrm{H}(\mathrm{A},\mathbb{S}^1))<c_0+\epsilon$ with $\operatorname{d}_{\mathrm{GH}}(\mathrm{A},\mathbb{S}^1)<\frac{1}{2}\psi(\operatorname{d}_{\mathrm{H}}(\mathrm{A},\mathbb{S}^1))$.

Another question of interest is that of proving similar bounds for manifolds with a boundary as mentioned in \cite[Section 9 Question 8]{adams2025hausdorffvsgromovhausdorffdistances}. One thing to consider is that the $n$-dimensional homology of an $n$-manifold with a non-empty boundary is trivial. The double of a manifold with boundary is the quotient space of two copies of $\mathcal{M}$ glued along its boundary $\partial \mathcal{M}$. It is defined as $\tilde{\mathcal{M}} = (\mathcal{M}\times\{0,1\})/\sim$ where $(x,0)\sim(x,1)\,\forall x\in\partial\mathcal{M}$. The double of a smooth manifold is compact and can be endowed with a smooth structure and thus a fundamental class can be defined for the double. In \cite{adamaszek_et_al:LIPIcs.SoCG.2018.3} the authors prove results relating the Vietoris-Rips complex of metric gluings with the quotient space obtained by gluing Vietoris-Rips complexes of the metric spaces along the Vietoris-Rips complex of their intersection. Although gluing may not preserve geometric properties like the reach, can we extend the inequalities to manifolds with boundary using such results. 

\bibliographystyle{plain}
\bibliography{references_HvGH_v1_final}
\end{document}